\documentclass{amsart}
\usepackage[margin=1in]{geometry}
\usepackage{amsfonts}
\usepackage{lineno,hyperref}
\usepackage{amsmath}
\usepackage{amssymb}
\usepackage{amsthm}
\usepackage{cancel}
\usepackage{tikz-cd}
\usepackage{arydshln}
\usepackage{graphicx}

\newcommand{\RR}{\mathbb{R}}
\newcommand{\CC}{\mathbb{C}}

\newcommand{\ZZ}{\mathbb{Z}}
\newcommand{\QQ}{\mathbb{Q}}
\newcommand{\OO}{\mathcal{O}}
\newcommand{\HH}{\mathbb{H}}
\newcommand{\re}{\text{re}}
\newcommand{\im}{\text{im}}
\newcommand{\nrm}{\text{nrm}}
\newcommand{\tr}{\text{tr}}
\newcommand{\disc}{\text{disc}}

\newcommand{\inv}{\text{inv}}
\newcommand{\End}{\text{End}}
\newcommand{\Mat}{\text{Mat}}
\newcommand{\Isom}{\text{Isom}}

\newtheorem{definition}{Definition}[section]
\newtheorem{theorem}{Theorem}[section]
\newtheorem{corollary}{Corollary}[section]
\newtheorem{lemma}{Lemma}[section]

\begin{document}
\title{Quaternion Orders and Sphere Packings}
\author{Arseniy Sheydvasser}
\address{Department of Mathematics, Graduate Center at CUNY, 365 5th Ave, New York, NY 10016}
\email{sheydvasser@gmail.com}


\date{\today}


\begin{abstract}
We introduce an analog of Bianchi groups for rational quaternion algebras and use it to construct sphere packings that are analogs of the Apollonian circle packing known as integral crystallographic packings.
\end{abstract}

\maketitle

\section{Introduction:}

While the history of the Apollonian gasket goes back at least to Leibniz, it is only relatively recently that there has been any significant progress in studying its number theoretic properties. In particular, there has been great interest in studying the bends of this packing; here the bend is defined as the reciprocal of the radius. It was apparently first observed by Soddy \cite{Soddy} that if the bends of the initial four circles in the configuration---that is the cluster of mutually tangent circles---are integers, then the bends of all of the circles in the gasket are integers. It was conjectured in \cite{GLMWY2} and \cite{FuchsSanden} that all sufficiently large integers satisfying certain congruence restrictions appear as the bends of some circle in the Apollonian circle packing. Although this local-global conjecture remains open, it was shown by Bourgain and Kontorovich in \cite{BourgainKontorovich} that the integers that appear as bends are density one in the subset of integers satisfying the required congruence restrictions.

\begin{figure}
\centering
\begin{tabular}{ccc}
\includegraphics[width=0.22\textwidth]{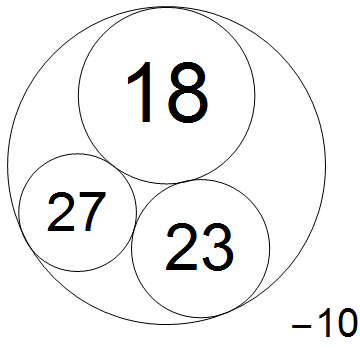} & \includegraphics[width=0.22\textwidth]{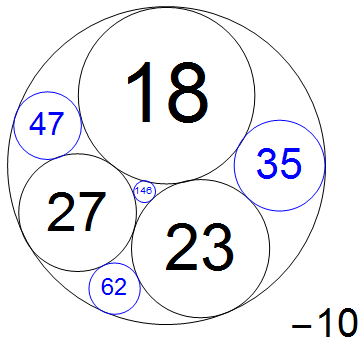} &\includegraphics[width=0.22\textwidth]{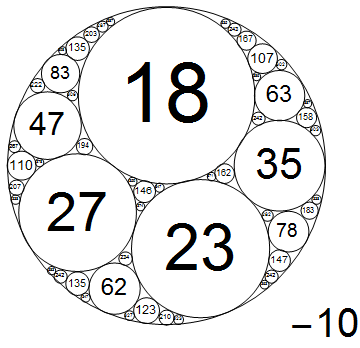}
\end{tabular}

\caption{Apollonian circle packing with initial cluster $(-10,18,23,27)$.}
\label{ApollonianPacking}
\end{figure}

We can formulate the precise statement of the local-global conjecture by realizing the Apollonian circle packing as the orbit of some initial cluster of circles under the action of a discrete, geometrically finite subgroup $\Gamma$ of the hyperbolic isometry group $\text{Isom}(\HH^3)$---indeed, the limit set of this group is the closure of the Apollonian gasket. Consequently, there has been been recent interest in studying analogs of the Apollonian gasket defined by the fact that they are limit sets of subgroups of $\text{Isom}(\HH^k)$. To be precise, we make the following definition due to Kontorovich and Nakamura \cite{KN}.

	\begin{definition}
	Let $\mathcal{P}$ be a set of oriented $n$-spheres in $\RR^{n + 1} \cup \{\infty\}$. We say that $\mathcal{P}$ is a \emph{packing} if
	
		\begin{enumerate}
		\item the interiors of $n$-spheres in $\mathcal{P}$ do not intersect, and
		\item the union of $n$-spheres in $\mathcal{P}$, together with their interiors, is dense in $\RR^{n + 1}$.
		\end{enumerate}
		
\noindent Furthermore, we say that $\mathcal{P}$ is \emph{integral} if the bends of all $n$-spheres in $\mathcal{P}$ are integers. Finally, we say that a packing $\mathcal{P}$ is \emph{crystallographic} if there exists a discrete, geometrically finite, hyperbolic reflection subgroup $\Gamma$ of $\text{Isom}(\HH^{n + 2})$ such that the limit set of $\Gamma$ is the closure of $\mathcal{P}$ in $\RR^{n + 1} \cup \{\infty\}$.
	\end{definition}

An important special case are the super-integral crystallographic packings, which have  the additional restriction that the superpacking of the original packing is also integral (see Section \ref{Superpacking Section} for definitions and motivation). There are many examples of super-integral crystallographic packings in the literature for $n = 1$---see \cite{GuettlerMallows}, \cite{StangeFuture}, and \cite{KN}. In particular, Kontorovich and Nakamura give a method for constructing an infinite family of super-integral crystallographic packings. In all these cases, there is a corresponding local-global conjecture, just as for the Apollonian circle packing. This conjecture is open for all such circle packings, but there are known density one results like the theorem of Bourgain and Kontorovich---see, for example, \cite{FuchsStangeZhang}. For $n = 2$, there is the generalization of the Apollonian circle packing to spheres due to Soddy \cite{Soddy}; additionally, Dias \cite{Dias} and Nakamura \cite{Nakamura} independently constructed the orthoplicial sphere packing. However, further examples are presently few in number. Unlike the $n = 1$ case, the associated local-global conjecture has been proven for some known packings. Kontorovich proved it for the Soddy packing \cite{Kontorovich}; building off that paper, Dias \cite{Dias} and Nakamura \cite{Nakamura} proved it for the orthoplicial sphere packing.

In the present paper, we present a method of constructing super-integral crystallographic packings for $n = 2$, and give an explicit list of ten examples satisfying certain properties---three of these are shown to be equivalent descriptions of packings already in the literature, while the remaining seven are entirely new. These are illustrated in Figure \ref{PackingGraphic}.

\begin{figure}
\begin{tabular}{ccccc}
\includegraphics[width=0.15\textwidth]{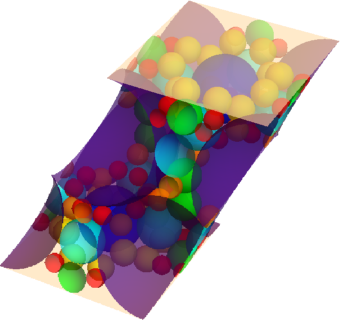} & \includegraphics[width=0.15\textwidth]{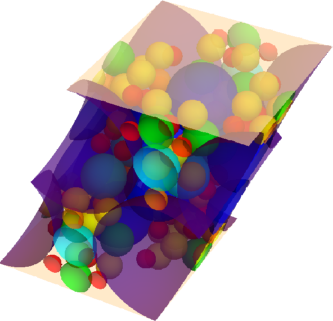} & \includegraphics[width=0.15\textwidth]{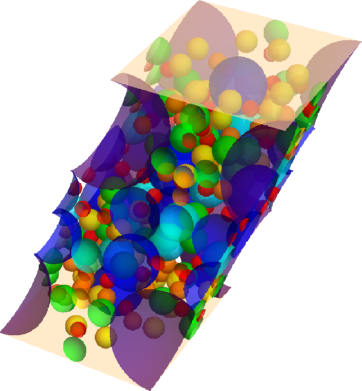} & \includegraphics[width=0.15\textwidth]{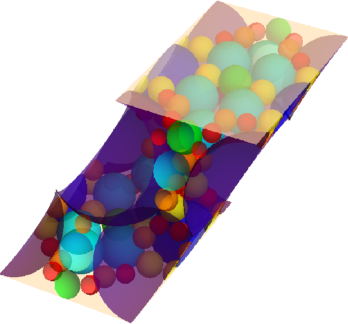} & \includegraphics[width=0.15\textwidth]{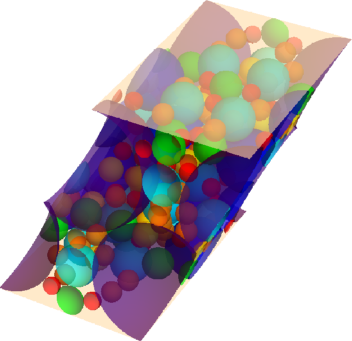} \\
\includegraphics[width=0.15\textwidth]{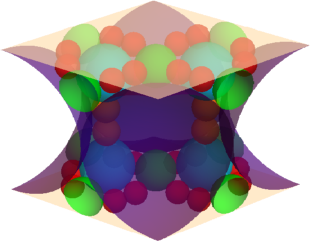} & \includegraphics[width=0.15\textwidth]{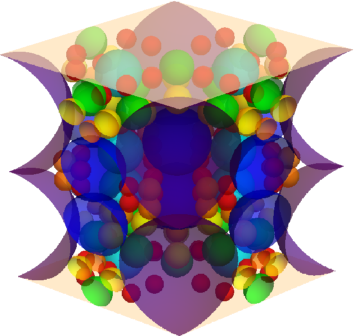} & \includegraphics[width=0.15\textwidth]{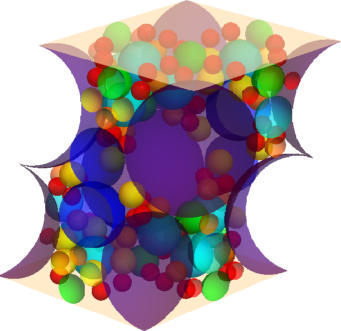} & \includegraphics[width=0.15\textwidth]{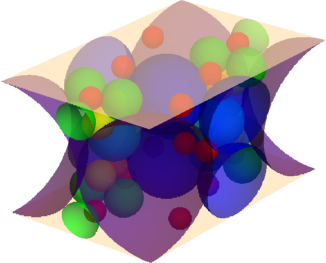} & \includegraphics[width=0.15\textwidth]{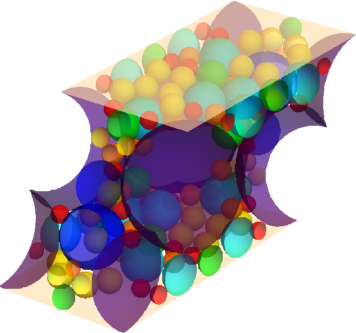}
\end{tabular}

\caption{From left to right, and top to bottom: super-integral crystallographic packings corresponding to quaternion algebras $\left(\frac{-1,-6}{\QQ}\right)$, $\left(\frac{-1,-7}{\QQ}\right)$, $\left(\frac{-1,-10}{\QQ}\right)$, $\left(\frac{-2,-5}{\QQ}\right)$, $\left(\frac{-2,-26}{\QQ}\right)$, $\left(\frac{-3,-1}{\QQ}\right)$, $\left(\frac{-3,-2}{\QQ}\right)$, $\left(\frac{-3,-15}{\QQ}\right)$, $\left(\frac{-7,-1}{\QQ}\right)$, $\left(\frac{-11,-143}{\QQ}\right)$.}
\label{PackingGraphic}
\end{figure}

\section{Summary of Main Results:}
In \cite{Stange2017}, Stange showed that the superpacking of the Apollonian circle packing is simply described as the orbit of $\mathbb{R}$ under the action of $SL(2,\ZZ[i])$. In analogy, she also considered the orbit of $\mathbb{R}$ under the action of other Bianchi groups $SL(2,\OO)$---here $\OO$ is the ring of integers of some imaginary quadratic field. It can be seen from Stange's work in \cite{StangeFuture} that if this orbit is a connected set, then it is the superpacking of an associated super-integral crystallographic packing. This gives a general strategy by which to attempt constructing super-integral crystallographic packings---look at the orbit of a fixed $n$-sphere under the action of some arithmetic subgroup of $\text{Isom}(\RR^{n + 2})$, and if the resulting collection satisfies the conditions for being a superpacking, search for an underlying integral crystallographic packing.

We replace the group $SL(2,\OO)$ with a corresponding group $SL^\ddagger(2,\OO)$, where $\OO$ is a maximal $\ddagger$-order of a rational quaternion algebra $H$ and $\ddagger$ is an orthogonal involution of $H$ (see Sections \ref{DaggerOrders} and \ref{Superpacking Section} for definitions). We then consider the orbit $\mathcal{S}_{\OO,j}$ of a fixed plane $\hat{S}_j$ under the action of the group $SL^\ddagger(2,\OO)$. This gives a collection of spheres that is a candidate to be the superpacking of an associated integral crystallographic packing, as long as this collection

	\begin{enumerate}
		\item has bends that are all integers after scaling by some fixed constant $C > 0$
		\item has only tangential intersections,
		\item is dense in $\RR^3$, and
		\item is connected.
	\end{enumerate}
	
\noindent We therefore want to find all isomorphism classes of maximal $\ddagger$-orders $\OO$ such that the corresponding sphere collection $\mathcal{S}_{\OO,j}$ satisfies these properties. However, due to various phenomena that do not occur in the $n = 1$ case (see Section \ref{WeirdIntersections} for examples), we prove a slightly weaker result. First, we replace the requirement that $\mathcal{S}_{\OO,j}$ be connected with the stronger requirement that it be \emph{tangency-connected}---that is, the graph with the spheres as vertices and edges corresponding to tangencies is connected. Secondly, we consider only a restricted set of maximal $\ddagger$-orders whose intersection with $\CC$ is Euclidean (see Section \ref{DefinitionOfEuclideanIntersection} for definitions). With those restrictions, we are able to fully classify all sphere packings satisfying the desired conditions, as encapsulated below.

\begin{theorem}\label{SuperpackingClassification}
Let $H$ be a rational, definite quaternion algebra with a maximal $\ddagger$-order $\OO$ whose intersection with $\CC$ is Euclidean. The sphere collection $\mathcal{S}_{\OO,j}$ is integral, tangential, dense, and tangency-connected for only a finite number of isomorphism classes of $H, \OO$, given below.

	\begin{minipage}{.45\textwidth}
	\begin{align*}
	\begin{array}{l|l}
	H & \OO \\ \hline
	\left(\frac{-1,-6}{\QQ}\right) & \ZZ \oplus \ZZ i \oplus \ZZ \frac{1 + i + j}{2} \oplus \ZZ \frac{j + ij}{2} \\
	\left(\frac{-1,-7}{\QQ}\right) & \begin{cases} \ZZ \oplus \ZZ i \oplus \ZZ \frac{1 + j}{2} \oplus \ZZ \frac{i + ij}{2} \\ \ZZ \oplus \ZZ i \oplus \ZZ \frac{i + j}{2} \oplus \ZZ \frac{1 + ij}{2} \end{cases} \\
	\left(\frac{-1,-10}{\QQ}\right) & \ZZ \oplus \ZZ i \oplus \ZZ \frac{1 + i + j}{2} \oplus \ZZ \frac{j + ij}{2} \\
	\left(\frac{-2,-5}{\QQ}\right) & \ZZ \oplus \ZZ i \oplus \ZZ \frac{1 + i + j}{2} \oplus \ZZ \frac{i + ij}{2} \\
	\left(\frac{-2,-26}{\QQ}\right) & \begin{cases} \ZZ \oplus \ZZ i \oplus \ZZ \frac{2 + i + j}{4} \oplus \ZZ \frac{i - j + ij}{4} \\ \ZZ \oplus \ZZ i \oplus \ZZ \frac{2 + i - j}{4} \oplus \ZZ \frac{i + j + ij}{4} \end{cases}
    \end{array}
    \end{align*}
    \end{minipage}%
    \begin{minipage}{0.45\textwidth}
    \begin{align*}
    \begin{array}{l|l}
	\left(\frac{-3,-1}{\QQ}\right) & \ZZ \oplus \ZZ \frac{1 + i}{2} \oplus \ZZ j \oplus \ZZ \frac{j + ij}{2} \\
	\left(\frac{-3,-2}{\QQ}\right) & \ZZ \oplus \ZZ \frac{1 + i}{2} \oplus \ZZ j \oplus \ZZ \frac{j + ij}{2} \\
	\left(\frac{-3,-15}{\QQ}\right) & \begin{cases} \ZZ \oplus \ZZ \frac{1 + i}{2} \oplus \ZZ \frac{i + j}{3} \oplus \ZZ \frac{3j + ij}{6} \\ \ZZ \oplus \ZZ \frac{1 + i}{2} \oplus \ZZ \frac{i - j}{3} \oplus \ZZ \frac{3j + ij}{6} \end{cases} \\
	\left(\frac{-7,-1}{\QQ}\right) & \ZZ \oplus \ZZ \frac{1 + i}{2} \oplus \ZZ j \oplus \ZZ \frac{j + ij}{2} \\
	\left(\frac{-11,-143}{\QQ}\right) & \begin{cases} \ZZ \oplus \ZZ \frac{1 + i}{2} \oplus \ZZ \frac{3i + j}{11} \oplus \ZZ \frac{11j + ij}{22} \\ \ZZ \oplus \ZZ \frac{1 + i}{2} \oplus \ZZ \frac{3i - j}{11} \oplus \ZZ \frac{11j + ij}{22} \end{cases}.
	\end{array}
	\end{align*}
    \end{minipage}
    
\noindent Any two of the above sphere collections $\mathcal{S}_{\OO,j}, \mathcal{S}_{\OO',j}$ are conformally equivalent if and only if $\OO,\OO'$ are orders over the same quaternion algebra.
\end{theorem}

\begin{figure}
\begin{tabular}{ccccc}
\includegraphics[width=0.15\textwidth]{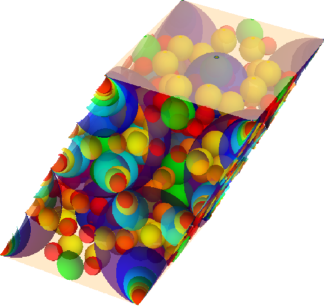} & \includegraphics[width=0.15\textwidth]{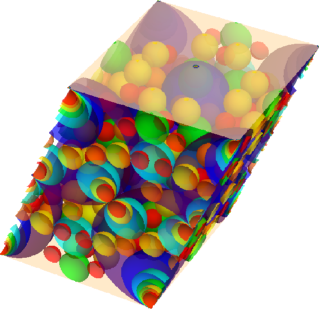} & \includegraphics[width=0.15\textwidth]{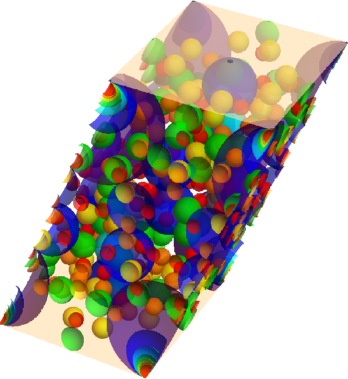} & \includegraphics[width=0.15\textwidth]{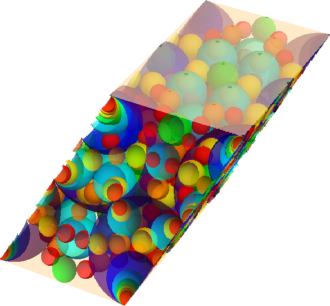} & \includegraphics[width=0.15\textwidth]{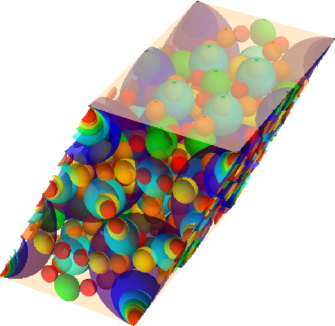} \\
\includegraphics[width=0.15\textwidth]{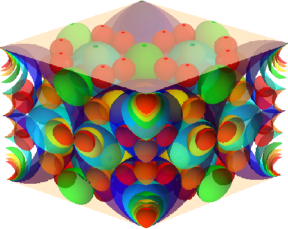} & \includegraphics[width=0.15\textwidth]{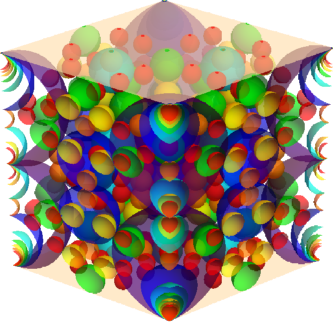} & \includegraphics[width=0.15\textwidth]{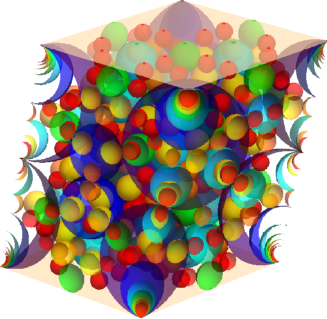} & \includegraphics[width=0.15\textwidth]{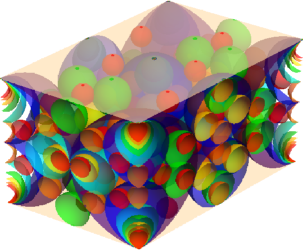} & \includegraphics[width=0.15\textwidth]{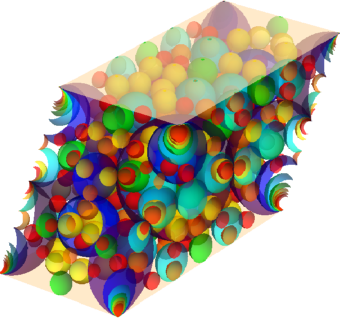}
\end{tabular}
\caption{From left to right: superpackings for the quaternion algebras $\left(\frac{-1,-6}{\QQ}\right)$, $\left(\frac{-1,-7}{\QQ}\right)$, $\left(\frac{-1,-10}{\QQ}\right)$, $\left(\frac{-2,-5}{\QQ}\right)$, $\left(\frac{-2,-26}{\QQ}\right)$, $\left(\frac{-3,-1}{\QQ}\right)$, $\left(\frac{-3,-2}{\QQ}\right)$, $\left(\frac{-3,-15}{\QQ}\right)$, $\left(\frac{-7,-1}{\QQ}\right)$, $\left(\frac{-11,-143}{\QQ}\right)$.}
\label{SuperPackingGraphic}
\end{figure}

On the other hand, we demonstrate that all of the sphere collections $\mathcal{S}_{\OO,j}$ that we claim are candidates to be superpackings of integral crystallographic packings in fact are such.

\begin{theorem}\label{SICP Classification}
The sphere collections $\overline{\mathcal{S}_{\OO,j}}$ corresponding to maximal $\ddagger$-orders

\begin{minipage}{0.45\textwidth}
	\begin{align*}
	\begin{array}{ll}
	\ZZ \oplus \ZZ i \oplus \ZZ \frac{1 + i + j}{2} \oplus \ZZ \frac{j + ij}{2} &\subset \left(\frac{-1, -6}{\QQ}\right) \\
	\ZZ \oplus \ZZ i \oplus \ZZ \frac{1 + j}{2} \oplus \ZZ \frac{i + ij}{2} &\subset \left(\frac{-1, -7}{\QQ}\right) \\
	\ZZ \oplus \ZZ i \oplus \ZZ \frac{1 + i + j}{2} \oplus \ZZ \frac{j + ij}{2} &\subset \left(\frac{-1, -10}{\QQ}\right) \\
	\ZZ \oplus \ZZ i \oplus \ZZ \frac{1 + i + j}{2} \oplus \ZZ \frac{i + ij}{2} &\subset \left(\frac{-2, -5}{\QQ}\right) \\
	\ZZ \oplus \ZZ i \oplus \ZZ \frac{2 + i + j}{4} \oplus \ZZ \frac{i - j + ij}{4} &\subset \left(\frac{-2, -5}{\QQ}\right)
	\end{array}
	\end{align*}
\end{minipage}%
\begin{minipage}{0.45\textwidth}
	\begin{align*}
	\begin{array}{ll}
	\ZZ \oplus \ZZ \frac{1 + i}{2} \oplus \ZZ j \oplus \ZZ \frac{j + ij}{2} &\subset \left(\frac{-3, -1}{\QQ}\right) \\
	\ZZ \oplus \ZZ \frac{1 + i}{2} \oplus \ZZ j \oplus \ZZ \frac{j + ij}{2} &\subset \left(\frac{-3, -2}{\QQ}\right) \\
	\ZZ \oplus \ZZ \frac{1 + i}{2} \oplus \ZZ \frac{i + j}{3} \oplus \ZZ \frac{3j + ij}{6} &\subset \left(\frac{-3, -15}{\QQ}\right) \\
	\ZZ \oplus \ZZ \frac{1 + i}{2} \oplus \ZZ j \oplus \ZZ \frac{j + ij}{2} &\subset \left(\frac{-7, -1}{\QQ}\right) \\
	\ZZ \oplus \ZZ \frac{1 + i}{2} \oplus \ZZ \frac{3i + j}{11} \oplus \ZZ \frac{11j + ij}{22} &\subset \left(\frac{-11, -143}{\QQ}\right)
	\end{array}
	\end{align*}
\end{minipage}
	
\noindent are the superpackings of super-integral crystallographic packings $\mathcal{P}_{\OO,j}$.
\end{theorem}

\noindent The proof is constructive---we show explicitly how to produce the desired packings $\mathcal{P}_{\OO,j}$. Furthermore, we prove that all but three of the packings are not conformally equivalent to other super-integral crystallographic packings that exist in the literature.

\subsection*{Acknowledgments:}
The author would like to thank Alex Kontorovich for suggesting this problem, providing insight into crystallographic packings, and giving helpful advice in the course of writing this paper. Additionally, the author thanks Yair Minsky and Katherine Stange for their comments on the geometry of limit sets and the structure of ghost circles, respectively.

\section{Notation:}

Throughout this paper, we use the following standard conventions.

\begin{tabular}{ll}
	$F$ & a local or global field. \\
    $\text{char}(F)$ & the characteristic of $F$. \\
    $\mathfrak{o}$ & the ring of integers of $F$. \\
	$H$ & a quaternion algebra over $F$. \\
	$x \mapsto \overline{x}$ & the standard involution on $H$. \\
	$\nrm(x) = x\overline{x}$ & the (reduced) norm map on $H$. \\
	$\tr(x) = x + \overline{x}$ & the (reduced) trace map on $H$. \\
	$H^0$ & the trace $0$ subspace of $H$. \\
    $\left(\frac{a,b}{F}\right)$ & the quaternion algebra generated by $i,j$, where $i^2 = a$, $j^2 = b$, $ij = -ji$. \\
    $\text{M\"{o}b}(\RR^n)$ & the group of M\"{o}bius transformations of $\RR^n \cup \{\infty\}$. \\
    $\Isom(\HH^n)$ & the group of isometries of $n$-dimensional hyperbolic space. \\
    $\Isom^0(\HH^n)$ & the group of orientation-preserving isometries of $n$-dimensional hyperbolic space.
\end{tabular}

\section{Orthogonal Involutions and $\ddagger$-Orders:}\label{DaggerOrders}

We review some basic facts about orthogonal involutions and the results of \cite{Sheydvasser2017}. Let $\ddagger$ an orthogonal involution on $H$---up to a change of basis, any such involution is of the form
	\begin{align*}
	(w + xi + yj + zij)^\ddagger = w + xi + yj - zij.
	\end{align*}
	
\noindent It is easy to see that $H$ is a direct product of the two subspaces $H^+$ and $H^-$, where
	\begin{align*}
	H^+ &= \left\{h \in H \middle| h^\ddagger = h\right\} \\
	H^- &= \left\{h \in H \middle| h^\ddagger = -h\right\}.
	\end{align*}
	
\noindent A homomorphism of quaternion algebras with involution is an $F$-linear ring homomorphism
	\begin{align*}
	\varphi: (H_1, \ddagger_1) \rightarrow (H_2, \ddagger_2)
	\end{align*}
	
\noindent such that
	\begin{align*}
	\varphi(x)^{\ddagger_2} &= \varphi\left(x^{\ddagger_1}\right) \ \forall x \in H_1.
	\end{align*}
	
\noindent We denote the automorphism group of $(H, \ddagger)$ by $GO(H, \ddagger)$. One easily checks that Noether's theorem implies there is an isomorphism
	\begin{align*}
	\left(F(\alpha)^\times \cup F(\alpha)^\times \beta\right)/F^\times &\rightarrow GO(H, \ddagger) \\
	x &\mapsto \left(h \mapsto xhx^{-1}\right),
	\end{align*}
	
\noindent where $\alpha \in H^-$ and $\beta \in H^0 \cap H^+ \backslash \{0\}$ such that $\alpha \beta = -\beta \alpha$. An important invariant of the orthogonal involution $\ddagger$ is its discriminant
	\begin{align*}
	\disc(\ddagger) &= \nrm\left(H^- \backslash \{0\}\right) \in F^\times/{F^\times}^2.
	\end{align*}

\noindent The discriminant uniquely determines the involution up to isomorphism---that is, given a quaternion algebra $H$ and two orthogonal involutions $\ddagger_1, \ddagger_2$, then there is an isomorphism $(H, \ddagger_1) \rightarrow (H, \ddagger_2)$ if and only if $\disc(\ddagger_1) = \disc(\ddagger_2)$. We shall be interested in considering orders inside the quaternion algebra $H$ that behave nicely with respect to $\ddagger$. With this motivation, we make the following definition.

	\begin{definition}
	Let $F$ be a local or global field. Let $H$ be a quaternion algebra over $F$ with involution $\ddagger$. A $\ddagger$-\emph{order} is an order of $H$ closed under $\ddagger$. A \emph{maximal} $\ddagger$-\emph{order} is a $\ddagger$-order not properly contained inside any other $\ddagger$-order of $H$.
	\end{definition}
	
\noindent We shall need the fact that maximal $\ddagger$-orders are characterized by their discriminant. Defining a map
	\begin{align*}
	\iota: F^\times/{F^\times}^2 &\rightarrow \left\{\text{square-free ideals of } \mathfrak{o}\right\} \\
	[\lambda] &\mapsto \bigcup_{\lambda \in [\lambda] \cap \mathfrak{o}} \lambda \mathfrak{o},
	\end{align*}
	
\noindent we can state this main condition as follows.

	\begin{theorem}\label{MainTheoremLastPaper}
	Given a quaternion algebra $H$ over a local or global field $F$, the maximal $\ddagger$-orders of $H$ are exactly the Eichler orders of the form $\OO \cap \OO^\ddagger$ with discriminant $\disc(H) \cap \iota(\disc(\ddagger))$.
	\end{theorem}
	
\noindent The proof of this theorem follows from considering the localizations $H_\mathfrak{p} = H \otimes_F F_\mathfrak{p}$,	$\OO_\mathfrak{p} = \OO \otimes_{\mathfrak{o}} \mathfrak{o}_\mathfrak{p}$.

\section{Maximal $\ddagger$-orders with a Fixed Element}
In addition to the results of Section \ref{DaggerOrders}, we shall also need a description of maximal $\ddagger$-orders containing some fixed element $j \in H^0 \cap H^+$ of square-free integer norm. We now classify the localizations of such orders; we split into two cases for orders over $\QQ_p$ where $p$ is an odd prime and orders over $\QQ_2$. To start, we define an equivalence relation on $\ZZ_p$ lattices. Given a $\ZZ_p$ lattice $\Lambda$ of full rank in a two-dimensional quadratic space $V$ over $\QQ_p$, with bilinear form $b$, the dual lattice is defined as
	\begin{align*}
	\Lambda^\sharp = \left\{v \in V \middle| b(v, \Lambda) \subset \ZZ_p\right\}.
	\end{align*}
	
\noindent Given two such lattices $\Lambda_1, \Lambda_2$, we write $\Lambda_1 \sim \Lambda_2$ if and only if $\Lambda_1 = \lambda \Lambda_2$ or $\Lambda_1 = \lambda \Lambda_2^\sharp$ for some $\lambda \in \ZZ_p^\times$.

\begin{lemma}\label{MostPrimesAreUnnecessary}
Let $H$ be a rational quaternion algebra with involution $\ddagger$. Fix an element $j \in H^0 \cap H^+$ with square-free integral norm. For all odd primes $p$, there is a unique maximal $\ddagger$-order in $H_p$ that contains $j$, unless $p|\nrm(j)$, $p\nmid \disc(H)$, and $-\disc(\ddagger) = \left(\QQ_p^\times\right)^2$, in which case there are precisely two distinct maximal $\ddagger$-orders in $H_p$ that contain $j$.
\end{lemma}

\begin{proof}
There is always at least one maximal $\ddagger$-order containing $j$---it suffices to show that this order is unique. If $p$ ramifies, then by \cite[Theorem 4.1]{Sheydvasser2017} there is a unique maximal $\ddagger$-order. If $p$ is unramified, then $H_p$ is isomorphic to $\Mat(2,\QQ_p)$ with involution
	\begin{align*}
	\begin{pmatrix} a & b \\ c & d \end{pmatrix}^\ddagger &= \begin{pmatrix} a & c/\lambda \\ b\lambda & d \end{pmatrix},
	\end{align*}
	
\noindent for some square-free $\lambda \in \ZZ_p$. Define a bilinear form
	\begin{align*}
	b_\ddagger(x,y) &= x^T \begin{pmatrix} \lambda & 0 \\ 0 & 1 \end{pmatrix} y
	\end{align*}
	
\noindent on $\QQ_p^2$. By Theorem \cite[Theorem 8.1]{Sheydvasser2017}, there is a bijection
	\begin{align*}
	\varphi: \left\{\text{maximal lattices in } \QQ_p^2\right\} / \sim &\rightarrow \left\{\text{maximal } \ddagger\text{-orders of } \Mat(2,\QQ_p)\right\} \\
	[\Lambda] &\mapsto \End(\Lambda) \cap \End(\Lambda^\sharp).
	\end{align*}
	
\noindent By this correspondence, if $b_\ddagger$ is anisotropic, there can be at most two maximal $\ddagger$-orders---one corresponding to the maximal lattice
	\begin{align*}
	\Lambda_1 &= \left\{v \in \QQ_p^2\middle|b_\ddagger(v,v) \in \ZZ_p\right\} \\
	&= \ZZ_p\begin{pmatrix} 1 \\ 0 \end{pmatrix} \oplus \ZZ_p\begin{pmatrix} 0 \\ 1 \end{pmatrix},
	\end{align*}
	
\noindent and the other to the maximal lattice
	\begin{align*}
	\Lambda_p &= \left\{v \in \QQ_p^2\middle|b_\ddagger(v,v) \in p\ZZ_p\right\}\\
	&= \ZZ_p \begin{pmatrix} p/\lambda \\ 0 \end{pmatrix} \oplus \ZZ_p \begin{pmatrix} 0 \\ p \end{pmatrix}.
	\end{align*}
	
\noindent If $p\nmid \lambda$, it is clear that $\Lambda_1 \sim \Lambda_p$. If $p|\lambda$, then
	\begin{align*}
	\Lambda_1^\sharp &= \ZZ_p\begin{pmatrix} 1/p \\ 0 \end{pmatrix} \oplus \ZZ_p\begin{pmatrix} 0 \\ 1 \end{pmatrix} \\ &\sim \ZZ_p\begin{pmatrix} 1 \\ 0 \end{pmatrix} \oplus \ZZ_p\begin{pmatrix} 0 \\ p \end{pmatrix} \\ &= \Lambda_p.
	\end{align*}
	
\noindent Therefore, in both cases there is in fact exactly one maximal $\ddagger$-order. It remains to consider the case where $b_\ddagger$ is isotropic---this occurs when $-\lambda \in \left(\QQ_p^\times\right)^2$. Since $\lambda$ is only well-defined up to multiplication by squares, we can take $\lambda = -1$ without loss of generality. To complete the proof we must determine the number of maximal $\ddagger$-orders that contain a fixed element
	\begin{align*}
	j = \begin{pmatrix} s & t \\ -t & -s \end{pmatrix},
	\end{align*}
	
\noindent where $s,t\in \ZZ_p$ are coprime and not both zero. It is easy to check that $\Mat(2,\ZZ_p)$ is a maximal $\ddagger$-order, and by \cite[Theorem 8.1]{Sheydvasser2017}, there is only one isomorphism class of maximal $\ddagger$-orders in $H_p$. Ergo, we need only consider all maximal $\ddagger$-orders of the form $M\left(\Mat(2,\ZZ_p)\right)M^{-1}$, where
	\begin{align*}
	M &= \begin{cases} \begin{pmatrix} a & b \\ -b & -a \end{pmatrix} \\ \begin{pmatrix} a & b \\ b & a \end{pmatrix} \end{cases},
	\end{align*}
	
\noindent with $a,b \in \ZZ_p$ coprime, and not both zero. Let $r \geq 0$ be the integer such that $a^2 - b^2 \in p^r\ZZ_p^\times$. Since $a,b$ are coprime, we have that $a - b \in p^r \ZZ_p^\times$ or $a + b \in \ZZ_p^\times$. Depending on which one it is, we can decompose
	\begin{align*}
	M \in \begin{cases} \begin{pmatrix} \frac{p^r+1}{2} & \frac{1-p^r}{2} \\ \frac{1-p^r}{2} & \frac{p^r+1}{2}\end{pmatrix} GL(2,\ZZ_p) \\
	\begin{pmatrix} \frac{p^r+1}{2} & \frac{1-p^r}{2} \\ -\frac{1-p^r}{2} & -\frac{p^r+1}{2}\end{pmatrix} GL(2,\ZZ_p) \end{cases}.
	\end{align*}
	
\noindent Since conjugation by $GL(2,\ZZ_p)$ stabilizes $\Mat(2,\ZZ_p)$, it remains to determine for what elements
	\begin{align*}
	M = \begin{pmatrix} \frac{p^r+1}{2} & \frac{1-p^r}{2} \\ \pm\frac{1-p^r}{2} & \pm\frac{p^r+1}{2}\end{pmatrix}
	\end{align*}
	
\noindent we have $M^{-1} j M \in \Mat(2,\ZZ_p)$---it is easy to check that different choices of $r$ yield distinct maximal $\ddagger$-orders. By direct computation,
	\begin{align*}
	\begin{pmatrix} \frac{p^r+1}{2} & \frac{1-p^r}{2} \\ \frac{1-p^r}{2} & \frac{p^r+1}{2}\end{pmatrix}^{-1} \begin{pmatrix} s & t \\ -t & -s \end{pmatrix} \begin{pmatrix} \frac{p^r+1}{2} & \frac{1-p^r}{2} \\ \frac{1-p^r}{2} & \frac{p^r+1}{2}\end{pmatrix} &= \begin{pmatrix} \frac{s + t}{2p^r} + \frac{(s - t)p^r}{2} & \frac{s + t}{2p^r} + \frac{(t - s)p^r}{2} \\ -\frac{s + t}{2p^r} -\frac{(t - s)p^r}{2} & -\frac{s + t}{2p^r} - \frac{(s - t)p^r}{2} \end{pmatrix} \\
	\begin{pmatrix} \frac{p^r+1}{2} & \frac{1-p^r}{2} \\ -\frac{1-p^r}{2} & -\frac{p^r+1}{2}\end{pmatrix}^{-1} \begin{pmatrix} s & t \\ -t & -s \end{pmatrix} \begin{pmatrix} \frac{p^r+1}{2} & \frac{1-p^r}{2} \\ -\frac{1-p^r}{2} & -\frac{p^r+1}{2}\end{pmatrix} &= \begin{pmatrix} \frac{s - t}{2p^r} + \frac{(s + t)p^r}{2} & \frac{s - t}{2p^r} + \frac{(s + t)p^r}{2} \\ -\frac{s - t}{2p^r} - \frac{(s + t)p^r}{2} & -\frac{s - t}{2p^r} - \frac{(s + t)p^r}{2} \end{pmatrix}.
	\end{align*}
	
\noindent We see from this that if $s - t, s + t \in \ZZ_p^\times$, then $r = 0$, and there is just one maximal $\ddagger$-order. However, if one of them is in $p\ZZ_p^\times$, then we get two solutions, one with $r = 0$, and the other with $r = 1$, yielding two distinct maximal $\ddagger$-orders. This latter case happens precisely when $p|\nrm(u)$.
\end{proof}

\begin{lemma}\label{TwoIsNotThatBad}
Let $H$ be a quaternion algebra over $\QQ_2$ generated by $i,j$, where $i,j$ have norms in $\ZZ_2$. Let $m = \nrm(i)$, $n = \nrm(j)$. Define an equivalence relations on pairs $(a,b),(c,d) \in \ZZ_2^2$, by $(a,b) \approx (c,d)$ if and only if $a/c,b/d \in (\ZZ_2^\times)^2$. Then the maximal $\ddagger$-orders $\OO$ containing $j$ are the ones given in Table \ref{TwoOrders}.
\end{lemma}

	\begin{table}
	\begin{align*}\renewcommand\arraystretch{1.5}
	\begin{array}{l|l} \OO & \left(m,n\right)/\approx \\ \hline
	\ZZ_2 \oplus \ZZ_2 i \oplus \ZZ_2 \frac{i + j}{2} \oplus \ZZ_2 \frac{2 + 2j + ij}{4} & (-6,-6),(-2,6),(2,2),(6,-2) \\
	\ZZ_2 \oplus \ZZ_2 i \oplus \ZZ_2 \frac{i + j}{2} \oplus \ZZ_2 \frac{2 + ij}{4} & (-6,-2),(-2,-6),(2,6),(6,2) \\
	\ZZ_2 \oplus \ZZ_2 i \oplus \ZZ_2 \frac{1 + i + j}{2} \oplus \ZZ_2 \frac{i + ij}{2} & (-6,1),(-2,-3),(2,1),(6,-3) \\
	\ZZ_2 \oplus \ZZ_2 i \oplus \ZZ_2 \frac{1 + j}{2} \oplus \ZZ_2 \frac{i + ij}{2} & (-6,3),(-2,3),(2,3),(6,3) \\
	\ZZ_2 \oplus \ZZ_2 i \oplus \ZZ_2 j \oplus \ZZ_2 \frac{1 + i + j + ij}{2} & (-3,-3),(-3,1),(1,-3),(1,1) \\
	\ZZ_2 \oplus \ZZ_2 i \oplus \ZZ_2 \frac{1 + i + j}{2} \oplus \ZZ_2 \frac{j + ij}{2} & (-3,-2),(-3,6),(1,-6),(1,2) \\
	\ZZ_2 \oplus \ZZ_2 \frac{1 + i}{2} \oplus \ZZ_2 j \oplus \ZZ_2 \frac{j + ij}{2} & (3,-6),(3,-2),(3,2),(3,6) \\ \hdashline
		\begin{cases} \ZZ_2 \oplus \ZZ_2 \frac{1 + i}{2} \oplus \ZZ_2 j \oplus \ZZ_2 \frac{j + ij}{2} \\ \ZZ_2 \oplus \ZZ_2 i \oplus \ZZ_2 \frac{1 + j}{2} \oplus \ZZ_2 \frac{i + ij}{2} \end{cases} & (-1,-1),(-1,3),(3,-1),(3,3) \\
	\begin{cases} \ZZ_2 \oplus \ZZ_2 i \oplus \ZZ_2 \frac{2 + i + j}{4} \oplus \ZZ_2 \frac{i - j + ij}{4} \\ \ZZ_2 \oplus \ZZ_2 i \oplus \ZZ_2 \frac{2 - i + j}{4} \oplus \ZZ_2 \frac{i + j + ij}{4} \end{cases} & (-6,2),(-2,-2),(2,-6),(6,6) \\	\hdashline
	\ZZ_2 \oplus \ZZ_2 i \oplus \ZZ_2 \frac{1 + i + j}{2} \oplus \ZZ_2 \frac{i + ij}{2} & (-6,-3),(-2,1),(2,-3),(6,1) \\
	\ZZ_2 \oplus \ZZ_2 i \oplus \ZZ_2 \frac{1 + j}{2} \oplus \ZZ_2 \frac{i + ij}{2} & (-6,-1),(-2,-1),(2,-1),(6,-1) \\
	\ZZ_2 \oplus \ZZ_2 i \oplus \ZZ_2 \frac{1 + i + j}{2} \oplus \ZZ_2 \frac{j + ij}{2} & (-3,-6),(-3,2),(1,-2),(1,6) \\
	\ZZ_2 \oplus \ZZ_2 \frac{1 + i}{2} \oplus \ZZ_2 j \oplus \ZZ_2 \frac{j + ij}{2} & (-1,-6),(-1,-2),(-1,2),(-1,6) \\ \hdashline
	\begin{cases} \ZZ_2 \oplus \ZZ_2 i \oplus \ZZ_2 \frac{i + j}{4} \oplus \ZZ_2 	\frac{2 + ij}{4} \\ \ZZ_2 \oplus \ZZ_2 i \oplus \ZZ_2 \frac{i - j}{4} \oplus \ZZ_2 	\frac{2 + ij}{4} \end{cases} & (-6,6),(-2,2),(2,-2),(6,-6) \\
	\begin{cases} \ZZ_2 \oplus \ZZ_2 i \oplus \ZZ_2 \frac{i + j}{2} \oplus \ZZ_2 	\frac{1 + ij}{2} \\ \ZZ_2 \oplus \ZZ_2 i \oplus \ZZ_2 \frac{1 + j}{2} \oplus \ZZ_2 	\frac{i + ij}{2} \end{cases} & (-3,1),(-3,3),(1,-1),(1,3) \\
	\begin{cases} \ZZ_2 \oplus \ZZ_2 i \oplus \ZZ_2 \frac{i + j}{2} \oplus \ZZ_2 	\frac{1 + ij}{2} \\ \ZZ_2 \oplus \ZZ_2 \frac{1 + i}{2} \oplus \ZZ_2 j \oplus \ZZ_2 	\frac{j + ij}{2} \end{cases} & (-1,3),(-1,1),(3,-3),(3,1).
	\end{array}
	\end{align*}
	\caption{Orders in a quaternion algebra over $\QQ_2$.}
	\label{TwoOrders}
	\end{table}

\begin{proof}
A simple computation shows that all the given lattices are maximal $\ddagger$-orders, so it remains to show that these are the only possibilities. If $2$ ramifies, then there is a unique maximal $\ddagger$-order. This happens precisely when
	\begin{align*}
	(m,n) \approx &(-6, -6), (-6, -2), (-6, 1), (-6, 3), (-3, -3), (-3, -2), (-3, 1), (-3, 6), (-2, -6), \\
&(-2, -3), (-2, 3), (-2, 6), (1, -6), (1, -3), (1,1), (1, 2), (2, 1), (2, 2), (2, 3), (2, 6), \\
&(3, -6), (3, -2), (3, 2), (3, 6), (6, -3), (6, -2), (6, 2), (6, 3),
	\end{align*}
	
\noindent corresponding to the first part of the table. If $2$ is unramified, then $H$ is isomorphic to $\Mat(2,\QQ_2)$ with involution
	\begin{align*}
	\begin{pmatrix} a & b \\ c & d \end{pmatrix}^\ddagger &= \begin{pmatrix} a & c/\lambda \\ b\lambda & d \end{pmatrix},
	\end{align*}
	
\noindent for some square-free $\lambda \in \ZZ_2$. As before, we define a bilinear form
	\begin{align*}
	b_\ddagger(x,y) &= x^T \begin{pmatrix} \lambda & 0 \\ 0 & 1 \end{pmatrix} y
	\end{align*}
	
\noindent on $\QQ_2^2$. As long as $\lambda \not\equiv -1 \mod 4$, then by \cite[Theorem 8.1]{Sheydvasser2017}, there is a bijection
	\begin{align*}
	\varphi: \left\{\text{maximal lattices in } \QQ_2^2\right\} / \sim &\rightarrow \left\{\text{maximal } \ddagger\text{-orders of } \Mat(2,\QQ_2)\right\} \\
	[\Lambda] &\mapsto \End(\Lambda) \cap \End(\Lambda^\sharp).
	\end{align*}
	
\noindent In this case, $b_\ddagger$ is anisotropic, and therefore there are at most two maximal $\ddagger$-orders, corresponding to maximal lattices
	\begin{align*}
	\Lambda_1 &= \left\{v \in \QQ_2^2\middle|b_\ddagger(v,v) \in \ZZ_2\right\} \\
	\Lambda_2 &= \left\{v \in \QQ_2^2\middle|b_\ddagger(v,v) \in 2\ZZ_2\right\},
	\end{align*}

\noindent and the question is if these two lattices are equivalent or not. If $\lambda \equiv 1 \mod 4$, one can check that they are not. This happens precisely when
	\begin{align*}
	(m,n) \approx & (-1,-1),(-1,3),(3,-1),(3,3),(-6,2),(-2,-2),(2,-6),(6,6),
	\end{align*}
	
\noindent and is represented by the second part of the table. Otherwise, $-\lambda \equiv 2 \mod 4$. In this case, we have
	\begin{align*}
	\Lambda_1 &= \ZZ_2 \begin{pmatrix} 1 \\ 0 \end{pmatrix} \oplus \ZZ \begin{pmatrix} 0 \\ 1 \end{pmatrix} \\
	\Lambda_2 &= \ZZ_2 \begin{pmatrix} 1 \\ 0 \end{pmatrix} \oplus \ZZ \begin{pmatrix} 0 \\ 2 \end{pmatrix},
	\end{align*}
	
\noindent and
	\begin{align*}
	\Lambda_1^\sharp &= \ZZ_2 \begin{pmatrix} 1/2 \\ 0 \end{pmatrix} \oplus \ZZ \begin{pmatrix} 0 \\ 1 \end{pmatrix} \\
	&\sim \Lambda_2,
	\end{align*}
	
\noindent hence there is only one maximal $\ddagger$-order. This happens precisely when
	\begin{align*}
	(m,n) \approx & (-6, -3), (-6, -1), (-3, -6), (-3, 2), (-2, -1), (-2, 1), (-1, -6), \\
&(-1, -2), (-1, 2), (-1, 6), (1, -2), (1, 6), (2, -3), (2, -1), \\
&(6,-1), (6, 1),
	\end{align*}
	
\noindent and is represented by the third part of the table. Finally, we consider the case where $\lambda \equiv -1 \mod 4$, which occurs when
	\begin{align*}
	(m,n) \approx & (-6, 6), (-3, -1), (-3, 3), (-2, 2), (-1, -3), (-1, 1), (1, -1), \\
&(1, 3), (2, -2), (3, -3), (3, 1), (6, -6).
	\end{align*}
	
\noindent In this case, by \cite[Theorem 8.2]{Sheydvasser2017}, there are two isomorphism classes of maximal $\ddagger$-orders, represented by
	\begin{align*}
	\OO_1 &= \Mat(2,\ZZ_2) \\
	\OO_2 &= \begin{pmatrix} 1 & -1 \\ 1 & 1 \end{pmatrix} \Mat(2,\ZZ_2) \begin{pmatrix} 1 & -1 \\ 1 & 1 \end{pmatrix}^{-1}.
	\end{align*}
	
\noindent That these correspond to distinct isomorphism classes can be seen from the fact that
	\begin{align*}
	\tr\left(\OO_1 \cap H^+\right) &= 2 \ZZ_2 \\
	\tr\left(\OO_2 \cap H^+\right) &= \ZZ_2.
	\end{align*}
	
\noindent Note that $\OO_1 \cap H^+$ does not contain any elements with norm in $2\ZZ_2^\times$---from this, we conclude that if $2|m$, only one isomorphism class of maximal $\ddagger$-orders has representatives containing $j$. On the other hand, if $2 \nmid m$, then by inspection both isomorphism classes have maximal $\ddagger$-orders containing $j$.

Next, note that since $i$ is one of two solutions to $z^2 = -\nrm(j)$ and $jz = -zj$ in $H^+$, the other being $-i$. However, isomorphism of maximal $\ddagger$-orders must respect both polynomial relations and whether or not the element is in $H^+$ or not. Therefore, any isomorphism of maximal $\ddagger$-orders that fixes $j$ must either send $i$ to $i$ or $i$ to $-i$. Since every other basis element can be written in terms of $i$ and $j$, we see that in fact there can be at most two maximal $\ddagger$-orders in each isomorphism classes, related by conjugation by $j$. This gives all of the maximal $\ddagger$-orders listed in the table.
\end{proof}

\section{Superpackings and Arithmetic Groups}\label{Superpacking Section}

A fundamental problem in studying integral crystallographic packings is that it is possible to construct pathological examples that do not correspond to orbits of an arithmetic group in the way that the Apollonian circle packing does. One approach to fix this is to introduce the notion of a superpacking. We use the following definition of Kontorovich and Nakamura \cite{KN}.

	\begin{definition}
	Let $\mathcal{P}$ be an integral crystallographic packing. Let $\Gamma$ be the largest reflection group in $Isom(\HH^{n + 2})$ that stabilizes $\mathcal{P}$. The \emph{supergroup} of $\mathcal{P}$ is the smallest subgroup of $Isom(\HH^{n + 2})$ containing $\Gamma$ and all reflections through the spheres of $\mathcal{P}$. The \emph{superpacking} of $\mathcal{P}$ is the image of $\mathcal{P}$ under the action of the supergroup. We say that $\mathcal{P}$ is \emph{super-integral} if the curvatures of spheres in its superpacking are integers.
	\end{definition}

\noindent The term ``superpacking" comes from the work of Graham, Lagarias, Mallows, Wilks, and Yan \cite{GLMWY}, who used it to describe the specific case of the Apollonian circle packing---in that setting, the supergroup is the group generated by the generators of the Apollonian group and their transposes. It is known that for any super-integral crystallographic packing the supergroup is arithmetic \cite{KN}.

For example, Stange showed that the superpacking of the Apollonian circle packing is simply described as the orbit of $\mathbb{R}$ under the action of $SL(2,\ZZ[i])$ \cite{Stange2017}. In analogy, she also considered the orbit of $\mathbb{R}$ under the action of other Bianchi groups $SL(2,\OO)$---here $\OO$ is the ring of integers of some imaginary quadratic field. In \cite{StangeFuture}, Stange then showed that if this orbit is a connected set, then it is the superpacking of an associated super-integral crystallographic packing. This gives a general strategy by which to attempt constructing super-integral crystallographic packings---look at the orbit of a fixed $n$-sphere under the action of some arithmetic subgroup of $\text{Isom}(\RR^{n + 2})$, and if the resulting collection satisfies the conditions for being a superpacking, search for an underlying integral crystallographic packing.

For any $n$, there is an isomorphism $\text{Isom}^0(\RR^{n + 2}) \cong \text{M\"{o}b}(\RR^{n + 1})$. For the specific case of $n = 2$, there is a classical isomorphism due to Vahlen \cite{Vahlen1902} and popularized by Ahlfors \cite{Ahlfors1986}.
	\begin{align*}
    \text{M\"{o}b}(\RR^3) \cong \left\{\begin{pmatrix} a & b \\ c & d \end{pmatrix} \middle| a,b,c,d \in H_\RR, ab^\ddagger, cd^\ddagger \in H_\RR^+, ad^\ddagger - bc^\ddagger = 1\right\}/\{\pm 1\},
    \end{align*}
    
\noindent where
	\begin{align*}
    H_\RR :&= \left(\frac{-1,-1}{\RR}\right), \\
    \left(t + xi + yj + zij\right)^\ddagger :&= t + xi + yj - zij, \\
    H_\RR^+ :&= \left\{\alpha \in H_\RR\middle| \alpha = \alpha^\ddagger\right\}.
    \end{align*}
    
\noindent One checks that inverses in this group are given by
	\begin{align*}
    \begin{pmatrix} a & b \\ c & d \end{pmatrix}^{-1} = \begin{pmatrix} d^\ddagger & -b^\ddagger \\ -c^\ddagger & a^\ddagger \end{pmatrix}.
    \end{align*}
    
\noindent With this motivation, given a subring $R$ of a quaternion algebra $H$ we define the following sets.
	\begin{align*}
	SL^\ddagger(2,R) &= \left\{\begin{pmatrix} a & b \\ c & d \end{pmatrix} \middle| a,b,c,d \in R, \ ab^\ddagger \in H^+, \ cd^\ddagger \in H^+, \ ad^\ddagger - bc^\ddagger = 1 \right\}, \\
	PSL^\ddagger(2,R) &= SL^\ddagger(2,R)/\{\pm id\}.
	\end{align*}
	
\noindent Of course, these sets will be groups if and only if $R = R^\ddagger$, which we assume hereafter. These groups act on $H^+$ by
	\begin{align*}
	\begin{pmatrix} a & b \\ c & d \end{pmatrix}.z = (az + b)(cz + d)^{-1},
	\end{align*}
	
\noindent which we recognize as the usual M\"{o}bius action. Indeed, as noted above, $PSL^\ddagger(2,H_\RR) \cong \text{M\"{o}b}(\RR^3)$. Furthermore, if $H$ is a definite, rational quaternion algebra with some orthogonal involution $\ddagger$, then there is an embedding $PSL^\ddagger(2,R) \hookrightarrow \Isom(\HH^4)$. Letting $R$ be a maximal $\ddagger$-order $\OO$, we consider the action of $SL^\ddagger(2,\OO)$ on some plane in $H^+$. To be precise, we fix some element $j \in \OO \cap H^0 \cap H^+$ with square-free norm, and consider the plane
	\begin{align*}
	S_j &= \left\{h \in H^+ \middle|\tr(jh) = 0\right\} \cup \{\infty\}.
	\end{align*}
	
\noindent Since $j$ is the normal vector to $S_j$, it also makes sense to write $\hat{S}_j$ for the sphere in $H^+ \cup \{\infty\}$ with orientation such that $j$ points toward the interior of $\hat{S}_j$. We have a corresponding sphere packing, in analogy to the Schmidt arrangements studied in \cite{Stange2017}.
	
	\begin{definition}
	Let $H$ be a rational, definite quaternion algebra with maximal $\ddagger$-order $\OO$. Let $j$ be a square-free element of $\OO \cap H^0 \cap H^+$. We define $\mathcal{S}_{\OO,j}$ to be the orbit of $\hat{S}_j$ under the action of $PSL^\ddagger(2,\OO)$. We define $\hat{\mathcal{S}}_{\OO,j}$ to be the union of the oriented spheres in $\mathcal{S}_{\OO,j}$ and the same spheres with reversed orientation.
	\end{definition}
	
\noindent Our goal in the remainder of this paper is to determine the isomorphism classes of $\OO$ for which the sphere collection $\mathcal{S}_{\OO,j}$ is a superpacking of an integral crystallographic packing. Note that the superpacking of a crystallographic sphere packing
	\begin{enumerate}
    	\item has bends that are all integers (possibly after scaling by a constant, if we allow conformal transformations),
       \item has only tangential intersections
       \item is dense in $\RR^3$, and
       \item is connected.
    \end{enumerate}

\noindent Consequently, we shall first try to determine for what choices of $\OO$ the collection $\mathcal{S}_{\OO,j}$ satisfies the above properties.

\section{Inversive Coordinates:}\label{InversiveCoordinates}

Our first task is to show how to obtain the geometric data of $\gamma \hat{S}_j \in \mathcal{S}_{\OO,j}$ given $\gamma \in SL^\ddagger(2,\OO)$. Specifically, we seek to understand the bend, co-bend, and bend-center.

	\begin{definition}
	Given an oriented sphere $\hat{S}$ in $\RR^3$, we define its \emph{bend} $\kappa(S)$ to be $1/\text{radius}$, taken to be positive if $S$ is positively oriented, and negative otherwise. If $S$ is a plane, $\kappa(S) = 0$. The \emph{co-bend} $\kappa'(S)$ is the bend of the image of $S$ under the map $z \mapsto -z^{-1}$. If $S$ is not a plane, the \emph{bend-center} $\xi(S)$ is the product of the bend $\kappa(S)$ and the center of $S$. If $S$ is a plane, $\xi(S)$ is the unique unit normal vector to $S$, pointing in the direction of the interior of $S$.
	\end{definition}

\noindent It isn't hard to see that

	\begin{enumerate}
	\item $\kappa(S)\kappa'(S) = \left|\xi(S)\right|^2 - 1$, and
	\item $\kappa(\gamma \hat{S}_u),\kappa'(\gamma \hat{S}_u),\xi(\gamma \hat{S}_u)$ are continuous functions in $\gamma$.
	\end{enumerate}
	
\noindent We shall show that they are in fact rational functions in the coefficients of $\gamma$. To prove this, we shall need to make use of Hermitian forms over quaternion algebras.

We recall that a Hermitian form over a quaternion algebra $H$ is a bi-additive map
	\begin{align*}
	T: H^n \times H^n \rightarrow H
	\end{align*}
	
\noindent satisfying

	\begin{enumerate}
		\item $T(x,yh) = T(x,y)h$ for all $x,y \in H^n$ and $h \in H$, and
		\item $T(x,y) = \overline{T(y,x)}$.
	\end{enumerate}
	
\noindent Consider the set $H_\RR^2 \times H_\RR^2$. Interpreting the components as columns of a matrix in $\Mat(2,H_\RR)$, we can define the right action of $SL\left(2,\CC)\right)$ by
	\begin{align*}
	\left(x,y\right) \begin{pmatrix} a & b \\ c & d \end{pmatrix} &= \left(xa + yc, xb + yd\right).
	\end{align*}
	
\begin{lemma}\label{InvariantsOfHermitianForms}
The $j$-th and $ij$-th components of any Hermitian form $T$ on $H_\RR^2$ are invariant under the right action of $SL\left(2,\CC\right)$.
\end{lemma}

\begin{proof}
Let $x, y \in H$, and
	\begin{align*}
	\gamma = \begin{pmatrix} a & b \\ c & d \end{pmatrix} \in SL\left(2,\CC\right).
	\end{align*}

\noindent We have that
\begin{align*}
T(xa + yc,xb + yd) &= T(xa + yc,x)b + T(xa + yc,y)d \\
&= \overline{a}T(x,x)b + \overline{c}T(y)(x)b + \overline{a}T(x,y)d + \overline{c}T(y,y)d \\
&= \underbrace{\left(T(x,x)\overline{a}b + T(y,y) \overline{c}d\right)}_{\in \CC} + \overline{c}T(y,x)b + \overline{a}T(x,y)d.
\end{align*}
	
\noindent As we are interested solely in the $j$-th and $ij$-th component, we can freely subtract by any element of $\CC$---therefore, it shall suffice to consider the part outside of the braces. For convenience, we write:
\begin{align*}
T(x,y) &= \alpha_0 + \alpha_1 i_\RR + \alpha_2 j_\RR + \alpha_3 i_\RR j_\RR,
\end{align*}
	
\noindent so that:
\begin{align*}
\overline{c}T(y,x)b &+ \overline{a}T(x,y)d \\
&= \overline{c} \left(\alpha_0 - \alpha_1 i_\RR - j_\RR(\alpha_2 - \alpha_3 i_\RR)\right) b + \overline{a} \left(\alpha_0 + \alpha_1 i_\RR + j_\RR(\alpha_2 - \alpha_3 i_\RR)\right)d \\
&= \underbrace{\left(\overline{c}(\alpha_0 - \alpha_1 i_\RR) b + \overline{a} (\alpha_0 + \alpha_1 i_\RR)d\right)}_{\in \CC} - j_\RR c(\alpha_2 - \alpha_3 i_\RR)b + j_\RR a(\alpha_2 - \alpha_3 i_\RR)d.
\end{align*}
	
\noindent Once again, we remove the piece in $\CC$, to see that:
\begin{align*}
- j c(\alpha_2 - \alpha_3 i_\RR) b + j_\RR a(\alpha_2 - \alpha_3 i_\RR) d &= j_\RR (ad - bc)(\alpha_2 - \alpha_3 i_\RR) \\ &= j_\RR(\alpha_2 - \alpha_3 i_\RR) \\ &= \alpha_2 j_\RR + \alpha_3 i_\RR j_\RR,
\end{align*}
	
which proves the lemma.
\end{proof}

Consider the real quaternion algebra $H_\RR$ with involution
	\begin{align*}
	(w + xi_\RR + yj_\RR + zi_\RR j_\RR)^\ddagger = w + xi_\RR + yj_\RR - zi_\RR j_\RR.
	\end{align*}
	
\noindent Let $\pi_1, \pi_{i_\RR}, \pi_{j_\RR}, \pi_{i_\RR j_\RR}$ denote the projection maps taking an element of $H_\RR$ to its real, $i$-th, $j$-th, and $ij$-th component respectively.

\begin{lemma}\label{InversiveRational}
Let $S = \gamma \hat{S}_{j_\RR}$ where $\gamma = \left(\begin{smallmatrix} a & b \\ c & d \end{smallmatrix}\right) \in SL^\ddagger(2,H_\RR)$. Then
\begin{align*}
\kappa(S) &= 2\pi_{j_\RR}\left(\overline{c} d\right) \\
\kappa'(S) &= 2\pi_{j_\RR}\left(\overline{a} b\right) \\
\xi(S) &= a j_\RR \overline{d} - b j_\RR \overline{c}
\end{align*}
\end{lemma}

\begin{proof}
First, note that it suffices to consider only oriented spheres $S$ that are not planes, since $\kappa, \kappa',\xi$ are continuous. Secondly, without loss of generality we can assume that $S$ is positively oriented, since it is easy to see that
	\begin{align*}
	S_1 &= \begin{pmatrix} a & b \\ c & d \end{pmatrix} \hat{S}_{j_\RR} \\
	S_2 &= \begin{pmatrix} a & b \\ c & d \end{pmatrix} \begin{pmatrix} 0 & j_\RR \\ j_\RR & 0 \end{pmatrix} \hat{S}_{j_\RR} \\
	&= \begin{pmatrix} b j_\RR & a j_\RR \\ d j_\RR & c j_\RR \end{pmatrix}
	\end{align*}
	
\noindent are the same sphere, but with opposite orientations, and if
	\begin{align*}
	\kappa(S_1) &= 2\pi_{j_\RR}\left(\overline{c} d\right) \\
	\kappa'(S_1) &= 2\pi_{j_\RR}\left(\overline{a} b\right) \\
	\xi(S_1) &= a j_\RR \overline{d} - b j_\RR \overline{c},
	\end{align*}
	
\noindent then
	\begin{align*}
	\kappa(S_2) &= - \kappa(S_1) = 2\pi_{j_\RR}\left(\overline{d j_\RR} c j_\RR\right) \\
	\kappa'(S_2) &= -\kappa'(S_1) = 2\pi_{j_\RR}\left(\overline{b j_\RR} (a j_\RR)\right) \\
	\xi(S_2) &= -\xi(S_1) = b j_\RR j_\RR \overline{c j_\RR} - a j_\RR j_\RR\overline{d j_\RR}.
	\end{align*}
	
\noindent Finally, notice that if $\varphi(\alpha) = \alpha^{-1}$, then $\kappa'(S) = \kappa\left(\varphi(S)\right)$, which means that it suffices to prove the lemma for $\kappa(S), \xi(S)$ where $S$ is a positively oriented sphere that is not a plane. In this case, note that we can always choose some $z \in \CC, \lambda \in \RR$ such that
	\begin{align*}
	S &= \begin{pmatrix} 1 & z \\ & 1 \end{pmatrix} \begin{pmatrix} \lambda & \\ & \lambda^{-1} \end{pmatrix} \begin{pmatrix} \frac{1}{2} + \frac{j_\RR}{2} & \frac{1}{2} - \frac{j_\RR}{2} \\ -\frac{1}{2} - \frac{j_\RR}{2} & \frac{1}{2} - \frac{j_\RR}{2} \end{pmatrix} \hat{S}_{j_\RR} \\
	&= \begin{pmatrix} \left(\lambda - \frac{z}{\lambda}\right)\left(\frac{1}{2} + \frac{j_\RR}{2}\right) & \left(\lambda + \frac{z}{\lambda}\right)\left(\frac{1}{2} - \frac{j_\RR}{2}\right) \\ -\frac{1}{\lambda}\left(\frac{1}{2} + \frac{j_\RR}{2}\right) & \frac{1}{\lambda}\left(\frac{1}{2} - \frac{j_\RR}{2}\right) \end{pmatrix}\hat{S}_{j_\RR}.
	\end{align*}
	
\noindent This is because the first matrix sends $\hat{S}_{j_\RR}$ to the unit sphere, the second scales it by a factor of $\lambda^2$, and the last shifts the center to $z$. It is clear that
	\begin{align*}
	\kappa(S) &= \frac{1}{\lambda^2} \\
	\xi(S) &= \frac{z}{\lambda^2}.
	\end{align*}
	
\noindent On the other hand, if we take
	\begin{align*}
	\gamma &= \begin{pmatrix} a & b \\ c & d \end{pmatrix} = \begin{pmatrix} \left(\lambda - \frac{z}{\lambda}\right)\left(\frac{1}{2} + \frac{j_\RR}{2}\right) & \left(\lambda + \frac{z}{\lambda}\right)\left(\frac{1}{2} - \frac{j_\RR}{2}\right) \\ -\frac{1}{\lambda}\left(\frac{1}{2} + \frac{j_\RR}{2}\right) & \frac{1}{\lambda}\left(\frac{1}{2} - \frac{j_\RR}{2}\right) \end{pmatrix},
	\end{align*}
	
\noindent then
	\begin{align*}
	2\pi_{j_\RR}\left(\overline{c}d\right) &= -2\pi_{j_\RR}\left(\frac{1}{\lambda^2}\left(\frac{1}{2} - \frac{j_\RR}{2}\right)^2\right) \\
	&= -\frac{2}{\lambda^2}\pi_{j_\RR}\left(-\frac{j_\RR}{2}\right) \\
	&= \frac{1}{\lambda^2} = \kappa(\gamma\hat{S}_{j_\RR}), \\
	a j_\RR \overline{d} - b j_\RR \overline{c} &= \left(\lambda - \frac{z}{\lambda}\right)\frac{1 + j}{2} \frac{j}{\lambda} \frac{1 + j}{2} + \left(\lambda + \frac{z}{\lambda}\right)\frac{1 - j}{2} \frac{j}{\lambda} \frac{1 - j}{2} \\
	&= \frac{z}{\lambda^2} = \xi(\gamma\hat{S}_{j_\RR})
	\end{align*}
		
\noindent as desired. Therefore, to prove that
	\begin{align*}
	f_1(\gamma) :&= 2\pi_{j_\RR}\left(\overline{c}d\right) = \kappa(\gamma\hat{S}_{j_\RR}) \\
	f_2(\gamma) :&= a j_\RR \overline{d} - b j_\RR \overline{c} = \xi(\gamma\hat{S}_{j_\RR}),
	\end{align*}
	
\noindent regardless of the choice of $\gamma$, it shall suffice to show that
	\begin{align*}
	f_1\left(\gamma SL(2,\CC)\right) &= f_1(\gamma) \\
	f_2\left(\gamma SL(2,\CC)\right) &= f_2(\gamma).
	\end{align*}
	
\noindent We do this by defining Hermitian forms
	\begin{align*}
	T_1\left(\begin{pmatrix} a \\ c \end{pmatrix},\begin{pmatrix} b \\ d \end{pmatrix}\right) &= \overline{\begin{pmatrix} a & c \end{pmatrix}} \begin{pmatrix} 0 & 0 \\ 0 & 1 \end{pmatrix} \begin{pmatrix} b \\ d \end{pmatrix} \\
	&= \overline{c} d \\
	T_2\left(\begin{pmatrix} a \\ c \end{pmatrix},\begin{pmatrix} b \\ d \end{pmatrix}\right) &= \overline{\begin{pmatrix} a & c \end{pmatrix}} \begin{pmatrix} 0 & 1 \\ 1 & 0 \end{pmatrix} \begin{pmatrix} b \\ d \end{pmatrix} \\
	&= \overline{a}d + \overline{c}b \\
	T_3\left(\begin{pmatrix} a \\ c \end{pmatrix},\begin{pmatrix} b \\ d \end{pmatrix}\right) &= \overline{\begin{pmatrix} a & c \end{pmatrix}} \begin{pmatrix} 0 & i_\RR \\ -i_\RR & 0 \end{pmatrix} \begin{pmatrix} b \\ d \end{pmatrix} \\
	&= \overline{a}i_\RR d - \overline{c}i_\RR b \\
	T_4\left(\begin{pmatrix} a \\ c \end{pmatrix},\begin{pmatrix} b \\ d \end{pmatrix}\right) &= \overline{\begin{pmatrix} a & c \end{pmatrix}} \begin{pmatrix} 0 & j_\RR \\ -j_\RR & 0 \end{pmatrix} \begin{pmatrix} b \\ d \end{pmatrix} \\
	&= \overline{a}j_\RR d - \overline{c}j_\RR b \\
	T_5\left(\begin{pmatrix} a \\ c \end{pmatrix},\begin{pmatrix} b \\ d \end{pmatrix}\right) &= \overline{\begin{pmatrix} a & c \end{pmatrix}} \begin{pmatrix} 0 & i_\RR j_\RR \\ -i_\RR j_\RR & 0 \end{pmatrix} \begin{pmatrix} b \\ d \end{pmatrix} \\
	&= \overline{a}i_\RR j_\RR d - \overline{c}i_\RR j_\RR b.
	\end{align*}
		
\noindent One checks by inspection that
	\begin{align*}
	\pi_{j_\RR}(\overline{a}d + \overline{c}b) &+ \pi_{j_\RR}(\overline{a}i_\RR d + \overline{c}i_\RR b)i_\RR \\
	&+ \pi_{j_\RR}(\overline{a}j_\RR d + \overline{c}j_\RR b)j_\RR + \pi_{j_\RR}(\overline{a}i_\RR j_\RR d + \overline{c}i_\RR j_\RR b)i_\RR j_\RR\\
	&= a j_\RR \overline{d} - b j_\RR \overline{c}.
	\end{align*}
	
\noindent Since we have expressed both $\kappa$ and $\xi$ in terms of Hermitian forms, the formulas given are invariant under right multiplication by $SL(2,\CC)$, and we are done.
\end{proof}

Given a definite, rational quaternion algebra $H$ with orthogonal involution $\ddagger$ and a non-zero element $j \in H_0 \cap H^+$, there exists another $i \in H_0 \cap H^+$ such that $ij = -ji$, and therefore there is an embedding
	\begin{align*}
	(H, \ddagger, j) &\hookrightarrow (H_\RR, \ddagger, j_\RR) \\
	a + bi + cj + dij & \mapsto a + b \sqrt{\nrm(i)} i_\RR + c \sqrt{\nrm(j)} j_\RR + d \sqrt{\nrm(ij)} i_\RR j_\RR.
	\end{align*}

\noindent With this in mind, we define $\pi_j: H \rightarrow \QQ$ to be the projection map taking $x$ to the $j$-th coordinate of $x$. We also define the reduced bend, reduced co-bend, and reduced bend-center to be maps
	\begin{align*}
	\tilde{\kappa}_j\left(\begin{pmatrix} a & b \\ c & d \end{pmatrix}\right) &= 2\re(\overline{c} d j) = 2 \nrm(j)\pi_j\left(\overline{c} d\right) \\
	\tilde{\kappa'}_j\left(\begin{pmatrix} a & b \\ c & d \end{pmatrix}\right) &= 2\re(\overline{a} b j) = 2 \nrm(j) \pi_j\left(\overline{a} b\right) \\
	\tilde{\xi}_j\left(\begin{pmatrix} a & b \\ c & d \end{pmatrix}\right) &= a j \overline{d} - b j \overline{c}.
	\end{align*}
	
\noindent Using the above embedding to produce an embedding $SL^\ddagger(2,\OO) \hookrightarrow SL^\ddagger(2,H_\RR)$, we have
	\begin{align*}
	\kappa(\gamma \hat{S}_j) &= \tilde{\kappa}_j(\gamma)/\sqrt{\nrm(j)} \\
	\kappa'(\gamma \hat{S}_j) &= \tilde{\kappa'}_j(\gamma)/\sqrt{\nrm(j)} \\
	\xi(\gamma \hat{S}_j) &= \tilde{\xi}_j(\gamma)/\sqrt{\nrm(j)}.
	\end{align*}
	
\noindent We define a quadratic form on $\QQ^2 \times H^+$ by
	\begin{align*}
	q_{H,j}\left(\kappa, \kappa', \xi\right) = -\kappa\kappa' + \nrm(\xi).
	\end{align*}
	
\noindent This induces a symmetric bilinear form
	\begin{align*}
	b_{H,j}(x_1, x_2) = \frac{1}{2}\left(q_{H,j}(x_1 + x_2) - q_{H,j}(x_1) - q_{H,j}(x_2)\right).
	\end{align*}

	\begin{lemma}\label{DefinitionOfInversiveCoordinates}
	Let $H$ be a definite, rational quaternion algebra with orthogonal involution $\ddagger$ and maximal $\ddagger$-order $\OO$. Let $j \in \OO \cap H^0 \cap H^+$ have square-free norm. We have a map
	\begin{align*}
	\text{inv}: SL^\ddagger(2,H) &\rightarrow \QQ^2 \times H^+ \\
	\gamma & \mapsto \left(\tilde{\kappa}_j(\gamma), \tilde{\kappa'}_j(\gamma),\tilde{\xi}_j(\gamma)\right).
	\end{align*}
	
	\noindent The image of $\text{inv}$ is contained inside the hypercone
	\begin{align*}
	q_{H,j}(x) &= \nrm(j),
	\end{align*}
	
	\noindent and the image of $SL^\ddagger(2,\OO)$ is contained inside $\ZZ^2 \times \left(\OO \cap H^+\right)$. Furthermore, given $\gamma_1, \gamma_2 \in SL^\ddagger(2,\OO)$, we consider $x_i = \text{inv}(\gamma_i)$ and $S_i = \gamma_i \hat{S}_j$. If $S_1$ intersects $S_2$ but $S_1 \neq S_2$, and $\theta$ is the angle of intersection between them, then
	
	\begin{align*}
	\nrm(j)\cos(\theta) = b_{H,j}(x_1, x_2).
	\end{align*}
	\end{lemma}
	
	\begin{proof}
	After embedding $H$ inside $H_\RR$ and doing a change of coordinates, we see that $\text{inv}$ really just describes inversive coordinates on the space of spheres in $\RR^3$, which have the desired properties---see, for example, the proof in \cite{Kocik}.
	\end{proof}
	
We can view the action of $SL^\ddagger(2,H)$ on oriented spheres in $H^+$ in another, equivalent way.

\begin{lemma}\label{AlternativeAction}
Let $H$ be a definite, rational quaternion algebra with orthogonal involution $\ddagger$. The group $SL^\ddagger(2,H)$ acts on the set
	\begin{align*}
	\mathcal{M}_{H,j} = \left\{M = \begin{pmatrix} a & b \\ c & d \end{pmatrix} \in \Mat(2,H) \middle| \overline{M}^T = M, ab^\ddagger, cd^\ddagger \in H^+\right\}
	\end{align*}
	
\noindent by
	\begin{align*}
	\gamma.M = \gamma M \overline{\gamma}^T.
	\end{align*}
	
\noindent Using the identification
	\begin{align*}
	\left(\kappa, \kappa', \xi\right) \leftrightarrow \begin{pmatrix} \kappa' & \xi \\ \overline{\xi} & \kappa \end{pmatrix},
	\end{align*}
	
\noindent this gives an action on the set of oriented spheres in $H^+$; this action is equivalent to the action already described.
\end{lemma}

\begin{proof}
Note that any element in $\mathcal{M}_{H,j}$ is of the form
	\begin{align*}
	\begin{pmatrix} \kappa' & \xi \\ \overline{\xi} & \kappa \end{pmatrix}
	\end{align*}
	
\noindent for some $\kappa',\kappa \in \QQ$, $\xi \in H^+$. The quasi-determinant of this matrix is just $q_{H,j}\left((\kappa,\kappa',\xi)\right)$. As the quasi-determinant is multiplicative whenever it is real, and the quasi-determinant of $\gamma M \overline{\gamma}^T$ is the quasi-determinant of $M$. Furthermore,
	\begin{align*}
	\overline{\left(\gamma M \overline{\gamma}^T\right)}^T = \gamma \overline{M}^T \overline{\gamma}^T = \gamma M \overline{\gamma}^T,
	\end{align*}
	
\noindent as desired. Finally, the subset of $\Mat(2,H)$ of matrices with real quasi-determinant and the relations $ab^\ddagger, cd^\ddagger \in H^+$ is closed under multiplication, so we conclude that $\gamma M \overline{\gamma}^T \in \mathcal{M}_{H,j}$, as claimed. More than that, because the quasi-determinant is preserved, and the quasi-determinant corresponds to the quadratic form $q_{H,j}$, we see that the cone $q_{H,j}(x) = n$ is preserved under the action, meaning that we have defined a valid action of $SL^\ddagger(2,H)$ on the oriented spheres in $H^+$. To see that this action agrees with the action already described, it suffices to check that it agrees on $\text{inv}(\hat{S}_j) = (0,0,j)$. Choose any element
	\begin{align*}
	\gamma = \begin{pmatrix} a & b \\ c & d \end{pmatrix} \in SL^\ddagger(2,H),
	\end{align*}
	
\noindent and note that
	\begin{align*}
	\begin{pmatrix} a & b \\ c & d \end{pmatrix} \begin{pmatrix} 0 & j \\ -j & 0 \end{pmatrix} \begin{pmatrix} \overline{a} & \overline{c} \\ \overline{b} & \overline{d} \end{pmatrix} &= \begin{pmatrix} -bj\overline{a} + aj\overline{b} & -bj\overline{c} + aj\overline{d} \\ -dj\overline{a} + cj\overline{b} & -dj\overline{c} + cj\overline{d} \end{pmatrix} \\
	&= \begin{pmatrix} \kappa'(\gamma) & \xi(\gamma) \\ \overline{\xi(\gamma)} & \kappa(\gamma) \end{pmatrix},
	\end{align*}
	
\noindent which agrees with the previous defined action.
\end{proof}

\begin{corollary}\label{AlgebraicGroup}
As algebraic groups over $\QQ$, $SL^\ddagger(2,H) \cong \text{Spin}(q_{H,j})$.
\end{corollary}

\begin{proof}
By Lemma \ref{AlternativeAction}, we know that $SL^\ddagger(2,H)$ acts on the set of points $(\kappa, \kappa', \xi) \in \QQ^2 \times H^+$, preserving the quadratic form $q_{H,j}$. In other words, we have found a morphism $SL^\ddagger(2,H) \rightarrow SO(q_{H,j})$. It isn't hard to check that this morphism is surjective, with kernel $\{\pm 1\}$---in other words, $SL^\ddagger(2,H)$ is a double-cover of $SO(q_{H,j})$. But, up to isomorphism, $SO(q_{H,j})$ has a unique double cover; namely, $\text{Spin}(q_{H,j})$. This concludes the proof.
\end{proof}
	
\section{Sphere Intersections in $\mathcal{S}_{\OO,j}$:}\label{WeirdIntersections}

We have shown that for a sphere packing $\mathcal{S}_{\OO,j}$, there exists a constant $C = \sqrt{\nrm(j)}$ such that after rescaling by this constant the bends of all spheres in the packing are integers. We shall see later that showing $\mathcal{S}_{\OO,j}$ is dense in $\RR^3$ is straightforward. Consequently, the remainder of this paper will be devoted to determining under what conditions we can conclude that $\mathcal{S}_{\OO,j}$ is tangential and tangency-connected. We begin with a few preliminary results. Given $\gamma \in SL^\ddagger(2,\OO)$, we write
	\begin{align*}
	\xi_3(\gamma) = \pi_j\left(\tilde{\xi}_j(\gamma)\right)
	\end{align*}

\begin{lemma}\label{RationalIntersections}
Let $H = \left(\frac{-m,-n}{\QQ}\right)$ be a definite quaternion algebra, and $\OO \subset H$ be a maximal $\ddagger$-order. Then intersection curves between spheres in $\mathcal{S}_{\OO,j}$ are all $\QQ$-rational if and only if for every $\gamma \in SL^\ddagger(2,\OO)$, the equation
	\begin{align*}
	X^2 + mY^2 &= n\left(1 - \xi_3(\gamma)^2\right)
	\end{align*}
	
\noindent has a rational solution.
\end{lemma}

\begin{proof}
Since elements of $SL^\ddagger(2,\OO)$ act on $H^+$ as bi-rational maps, it suffices to prove the lemma in the case where one of the spheres is the plane $\hat{S}_j$. Furthermore, we can assume that the other sphere $S$ is not a plane. To see this, choose a point $z \in \OO \cap \hat{S}_j$ such that $-z^{-1} \notin S$. Then the transformation
	\begin{align*}
	\gamma' = \begin{pmatrix} 1 & 0 \\ z & 1 \end{pmatrix} \in SL^\ddagger(2,\OO)
	\end{align*}
	
\noindent sends $\hat{S}_j$ to $\hat{S}_j$ and $\infty \notin \gamma'S$. Fix $\gamma \in SL^\ddagger(2,\OO)$ such that $S = \gamma \hat{S}_j$. With the above assumptions, we know by basic geometry that the intersection curve is the set of points $a + bi$ such that
	\begin{align*}
	\nrm\left(a + bi - \frac{\tilde{\xi}(\gamma)}{\tilde{\kappa}(\gamma)}\right) = \frac{n}{\tilde{\kappa}(\gamma)^2}.
	\end{align*}
	
\noindent After a rational change of variables, this yields the curve
	\begin{align*}
	X^2 + mY^2 &= n\left(1 - \xi_3(\gamma)^2\right).
	\end{align*}
	
\noindent Since this is a conic section, it is rational if and only if it has at least one rational point.
\end{proof}

\begin{lemma}\label{IntersectionAngles}
Let $H = \left(\frac{-m,-n}{\QQ}\right)$ be a definite quaternion algebra, and $\OO \subset H$ be a maximal $\ddagger$-order. There exist two spheres in $\mathcal{S}_{\OO,j}$ intersecting at an angle $\theta$ if and only if there exists $\gamma \in SL^\ddagger(2,\OO)$ such that
	\begin{align*}
	\theta = \arccos\left(\xi_3(\gamma)\right).
	\end{align*}
\end{lemma}

\begin{proof}
Since the action of $SL^\ddagger(2,\OO)$ preserves angles, it suffices to consider the case where one of the spheres is the plane $\hat{S}_j$. Choose $\gamma \in SL^\ddagger(2,\OO)$ such that the other sphere $S = \gamma \hat{S}_j$. Letting $\theta$ be the angle between these two spheres; then from Lemma \ref{DefinitionOfInversiveCoordinates} we have
	\begin{align*}
	\cos(\theta) &= \frac{1}{\nrm(j)}b_{H,j}\left(\text{inv}(\gamma),\text{inv}(id)\right) \\ &= \frac{1}{\nrm(j)}b_{H,j}\left(\text{inv}(\gamma),(0,0,j)\right) = \xi_3(\gamma).
	\end{align*}
	
\noindent The claim follows.
\end{proof}

The unit group $\OO^\times$ has a standard embedding into $SL^\ddagger(2,\OO)$ by
	\begin{align*}
	\OO^\times &\hookrightarrow SL^\ddagger(2,\OO) \\ u &\mapsto \begin{pmatrix} u & 0 \\ 0 & {u^{-1}}^\ddagger \end{pmatrix},
	\end{align*}
	
\noindent which gives it an action by rotations on $H^+$.

\begin{lemma}\label{BasicImplication}
If intersections in $\mathcal{S}_{\OO,j}$ are tangential, then intersections in $\mathcal{S}_{\OO,j}$ are rational and the action of $\OO^\times$ preserves $\hat{S}_j$.
\end{lemma}

\begin{proof}
If all intersections are tangential, then for any $\gamma \in SL^\ddagger(2,\OO)$ such that $\hat{S}_j$ intersects $\gamma \hat{S}_j$, we have that $\xi_3 = \pm 1$ by Lemma \ref{IntersectionAngles}, and therefore
	\begin{align*}
	X^2 + mY^2 = n\left(1 - \xi_3(\gamma)^2\right) = 0,
	\end{align*}
	
\noindent which is rational. If the action of $\OO^\times$ does not preserve $\hat{S}_j$, then there is a rotation $\gamma \in SL^\ddagger(2,\OO)$ about the origin such that $\hat{S}_j \neq \gamma\hat{S}_j$, and the intersection between these two planes in $\mathcal{S}_{\OO,j}$ must be a line.
\end{proof}

The converse of Lemma \ref{BasicImplication} is false. To see this, consider the quaternion algebra $H = \left(\frac{-5,-1}{\QQ}\right)$ and the maximal $\ddagger$-order
	\begin{align*}
	\OO = \ZZ \oplus \ZZ i \oplus \ZZ j \oplus \ZZ \frac{1 + i + j + ij}{2}.
	\end{align*}
	
\noindent It is easily checked that $\OO^\times = \langle j \rangle$, and since $j^2 = -1$, it follows that it acts on $H^+$ by half-turns, which preserve $\hat{S}_j$. Furthermore, for any
	\begin{align*}
	\gamma = \begin{pmatrix} a & b \\ c & d \end{pmatrix} \in SL^\ddagger(2,\OO),
	\end{align*}
	
\noindent we must have
	\begin{align*}
	\tilde{\xi}(\gamma) &\in \OO \cap H^+ = \ZZ \oplus \ZZ i \oplus \ZZ j,
	\end{align*}
	
\noindent from which we conclude that $\xi_3(\gamma) \in \ZZ$. Therefore, we need only to check that
	\begin{align*}
	X^2 + 5 Y^2 &= 1 \\
	X^2 + 5 Y^2 &= 0
	\end{align*}
	
\noindent are rational to conclude that intersections in $\mathcal{S}_{j,\OO}$ are rational. However, intersections in $\mathcal{S}_{j,\OO}$ are not generally tangential. Indeed, one checks that
	\begin{align*}
	\gamma = \begin{pmatrix} \frac{-1 + i + j - ij}{2} & \frac{1 + i + j + ij}{2} \\ \frac{1 + i - j - ij}{2} & 1 + j \end{pmatrix} \in SL^\ddagger(2,\OO),
	\end{align*}
	
\noindent and $\xi_3(\gamma) = 0$, which shows that the intersection curve is not a point.

For that matter, the statement that tangential intersections imply rational intersections is not vacuous---it is possible for the intersection curves of $\mathcal{S}_{j,\OO}$ to be non-rational. For example, consider $H = \left(\frac{-7,-1}{\QQ}\right)$ and the maximal $\ddagger$-order
	\begin{align*}
	\OO = \ZZ \oplus \ZZ i \oplus \ZZ \frac{i + j}{2} \oplus \ZZ \frac{1 + ij}{2}.
	\end{align*}
	
\noindent One checks that
	\begin{align*}
	\xi_3\left(\begin{pmatrix} 1 & \frac{i + j}{2} \\ \frac{i + j}{2} & 1 \end{pmatrix}\right) = \frac{1}{2},
	\end{align*}
	
\noindent but
	\begin{align*}
	X^2 + 7 Y^2 = \frac{3}{4}
	\end{align*}
	
\noindent admits no rational solutions. We will show later that for
	\begin{align*}
	H = \left(\frac{-7,-1}{\QQ}\right) &\text{ and } \OO = \ZZ \oplus \ZZ \frac{1 + i}{2} \oplus \ZZ j \oplus \ZZ \frac{j + ij}{2} \\
	H = \left(\frac{-1,-7}{\QQ}\right) &\text{ and } \OO = \begin{cases} \ZZ \oplus \ZZ i \oplus \ZZ \frac{i + j}{2} \oplus \ZZ \frac{1 + ij}{2} \\ \ZZ \oplus \ZZ i \oplus \ZZ \frac{1 + j}{2} \oplus \ZZ \frac{i + ij}{2} \end{cases}
	\end{align*}
	
\noindent all intersections are tangential. This shows that whether or not intersections are rational or tangential is not determined solely by the quaternion algebra or the choice of normal vector $j$, but depends crucially on the choice of maximal $\ddagger$-order $\OO$.

In contrast, for circle packings constructed as orbits of $\RR$ under the action of a Bianchi group $SL(2,\OO)$, Stange proved that all intersections are rational and furthermore are tangential if and only if the action of the unit group preserves $\RR$ \cite{Stange2017}. Moreover, in that context, even if the unit group does not preserve $\RR$, it is nevertheless true that there exist two spheres that intersect at angle $\theta$ if and only if there exists a unit $u \in \OO^\times$ such that
	\begin{align*}
	\cos(\theta) = \pi_i\left(\xi\left(\begin{pmatrix} u & 0 \\ 0 & u^{-1} \end{pmatrix}\right)\right).
	\end{align*}
	
\noindent We have shown that the corresponding statement for maximal $\ddagger$-orders is false, but we might nevertheless conjecture that all angles of intersection must be rational multiples of $\pi$. This is also false---let $H = \left(\frac{-m,-n}{\QQ}\right)$ be a rational quaternion algebra, $\OO$ be a maximal $\ddagger$-order, and $z \in \OO \cap H^+$. Then
	\begin{align*}
	\gamma = \begin{pmatrix} 1 & z \\ 0 & 1 \end{pmatrix}\begin{pmatrix} 1 & 0 \\ z - \tr(z) & 1 \end{pmatrix} &= \begin{pmatrix} 1 - \nrm(z) & z \\ z - \tr(z) & 1 \end{pmatrix} \in SL^\ddagger(2,\OO)
	\end{align*}
	
\noindent and $\xi_3(\gamma) = 1 - 2 \pi_j(z)^2 n$. So, for example, if we take $H = \left(\frac{-5,-5}{\QQ}\right)$ and the order
	\begin{align*}
	\OO = \ZZ \oplus \ZZ i \oplus \frac{2i + j}{5} \oplus \frac{5 + i + 3j + ij}{10},
	\end{align*}
	
\noindent then
	\begin{align*}
	1 - 2 \pi_j\left(\frac{2i + j}{5}\right)^2 5 &= \frac{3}{5},
	\end{align*}
	
\noindent and $\arccos(3/5)$ is not a rational multiple of $\pi$.

\section{Strong Approximation:}

What we have seen is that in order to determine the tangency structure of $\hat{S}_{\OO,j}$, we need to understand the set $\xi_3(SL^\ddagger(2,\OO))$. This can be accomplished by appealing to the strong approximation theorem for algebraic groups, from which we obtain the following theorem.

\begin{theorem}\label{StrongApproximation}
Let $\OO$ be a $\ddagger$-order of a rational quaternion algebra $H$, and $\Psi: SL^\ddagger(2,\OO) \rightarrow \ZZ^k$ a coordinate-wise polynomial map. Then
	\begin{align*}
	\Psi\left(SL^\ddagger(2,\OO)\right) = \bigcap_{p \text{ prime}} \Psi\left(SL^\ddagger(2,\OO_p)\right).
	\end{align*}
\end{theorem}

\begin{proof}
That
	\begin{align*}
	\Psi\left(SL^\ddagger(2,\OO)\right) \subset \bigcap_{p \text{ prime}} \Psi\left(SL^\ddagger(2,\OO_p)\right).
	\end{align*}
	
\noindent is evident. To prove containment the other way, we appeal to strong approximation. We consider $SL^\ddagger(2,\OO)$ as the set of integer points on an algebraic group $G$. By Corollary \ref{AlgebraicGroup}, we know that $G(\QQ) \cong \text{Spin}(q_{H,j})$. This group is almost simple and simply-connected. Therefore, by the Strong Approximation theorem proved in characteristic zero by Kneser \cite{Kneser1965} and Platonov \cite{Platonov1969}, we know that for any finite set of places $S$ of $\QQ$, $G(\QQ)$ is dense in $G(\mathbb{A}_S)$ if and only if
	\begin{align*}
	G_S = \prod_{\nu \in S} G(\QQ_\nu)
	\end{align*}
	
\noindent is not compact---here, $\mathbb{A}_S$ denotes the ring of $S$-adeles. In our case, we take $S$ to be the unique archimedian place of $\QQ$ and note that then
	\begin{align*}
	G_S = G(\RR) \cong SL^\ddagger(2,\RR),
	\end{align*}
	
\noindent which is certainly not compact, and so we can conclude that $G(\QQ)$ is dense inside
	\begin{align*}
	G(\mathbb{A}) = \left\{g = (g_p) \in \prod_{p \text{ prime}} G(\QQ_p) \middle| g_p \in G(\ZZ_p) \text{ for almost all } p\right\}.
	\end{align*}
	
\noindent However, since $S$ precisely consists of the archimedian places, this implies that $G(\ZZ) = SL^\ddagger(2,\OO)$ is dense inside of
	\begin{align*}
	G(\hat{\ZZ}) = \prod_p SL^\ddagger(2,\OO_p).
	\end{align*}
	
\noindent On the other hand, since $\Psi$ is coordinate-wise polynomial, it can be enlarged to a continuous map $\Psi: G(\hat{\ZZ}) \rightarrow \hat{\ZZ}^k$. Since $SL^\ddagger(2,\OO)$ is a dense subset, its image must be dense in the image of $\Psi$. However, this implies that
	\begin{align*}
	\Psi(SL^\ddagger(2,\OO)) = \Psi\left(\prod_p SL^\ddagger(2,\OO_p)\right) \cap \ZZ^k \supset \bigcap_{p \text{ prime}} \Psi\left(SL^\ddagger(2,\OO_p)\right),
	\end{align*}
	
\noindent which proves the claim.
\end{proof}

With this result in hand, it now suffices to determine the image of the $p$-adic localizations. We consider two cases, depending on whether or not $p$ is odd.

\begin{lemma}\label{LocalImageNotTwo}
Let $p$ be an odd prime and $H$ a quaternion algebra over $\QQ_p$ with involution $\ddagger$. Fix an element $j \in H^0 \cap H^+$ with square-free integer norm, and a maximal $\ddagger$-order $\OO$ in $H$ containing $j$. Then
	\begin{align*}
	\xi_3\left(SL^\ddagger(2,\OO)\right) = \begin{cases} \frac{1}{p}\ZZ_p & \text{if } j\OO j^{-1} \neq \OO \\ 1 + \nrm(j)\ZZ_p & \text{otherwise}. \end{cases}
	\end{align*}
\end{lemma}

\begin{proof}
If
	\begin{align*}
	\gamma = \begin{pmatrix} a & b \\ c & d \end{pmatrix} \in SL^\ddagger(2,\OO),
	\end{align*}
	
\noindent then
	\begin{align*}
	\xi_3(\gamma) &= \pi_j\left(aj\overline{d} - bj\overline{c}\right) \\
	&= 1 + \pi_j\left((aj\overline{d} - bj\overline{c}) - (ad^\ddagger - bc^\ddagger)j\right) \\
	&= 1 + \pi_j\left(a(j\overline{d} - d^\ddagger j) - b(j\overline{c} - c^\ddagger j)\right)\\
	&= 1 + \pi_j\left(a(j\overline{d}j^{-1} - d^\ddagger)j - b(j\overline{c}j^{-1} - c^\ddagger)j\right)\\
	&= 1 + \re\left(a(j\overline{d}j^{-1} - d^\ddagger) - b(j\overline{c}j^{-1} - c^\ddagger)\right).
	\end{align*}
	
\noindent If $j\OO j^{-1} = \OO$, this implies that $j\overline{d}j^{-1} - d^\ddagger \in \OO$, with trace zero and no $i$ component. From this, we can conclude that $\xi_3(\gamma) \in 1 + \nrm(j)\ZZ_p$. If $j\OO j^{-1} \neq \OO$, then we can only assume that $p j\OO j^{-1} \subset \OO$, and so we only know that $\xi_3(\gamma) \in p^{-1}\ZZ_p$.

It is easy to show that $\xi_3(SL^\ddagger(2,\OO)) \supset 1 + \nrm(j)\ZZ_p$. To see this, it suffices to note that
	\begin{align*}
	\gamma &= \begin{pmatrix} 1 & \lambda j \\ j & 1 + \lambda j^2 \end{pmatrix} \in SL^\ddagger(2,\OO) \\
	\xi(\gamma) &= (1 + \lambda j^2)j + \lambda j^3 = (1 - 2\lambda \nrm(j))j
	\end{align*}
	
\noindent for any $\lambda \in \ZZ_p$. Therefore, it remains to prove that if $j \OO j^{-1} \neq \OO$, then $\xi_3(SL^\ddagger(2,\OO)) \supset p^{-1}\ZZ_p$. By Lemma \ref{MostPrimesAreUnnecessary}, if $j \OO j^{-1} \neq \OO$, then $\OO = \Mat(2,\ZZ_p)$ without loss of generality, where the involution is
	\begin{align*}
	\begin{pmatrix} a & b \\ c & d \end{pmatrix}^\ddagger = \begin{pmatrix} a & -c \\ -b & d \end{pmatrix}.
	\end{align*}
	
\noindent We conclude that there exist some $s,t \in \ZZ_p^\times$ such that
	\begin{align*}
	j = \begin{pmatrix} s & t \\ -t & -s \end{pmatrix},
	\end{align*}
	
\noindent and we choose an $\alpha \in \OO^\times$, so
	\begin{align*}
	\gamma = \begin{pmatrix} \alpha & 0 \\ 0 & {\alpha^{-1}}^\ddagger \end{pmatrix} \in SL^\ddagger(2,\OO).
	\end{align*}
	
\noindent Then
	\begin{align*}
	\xi(\gamma) &= \alpha j \overline{{\alpha^{-1}}^\ddagger} = \frac{\alpha j \alpha^\ddagger}{\nrm(\alpha)} \\
	\xi_3(\gamma) &= \re\left(\frac{\alpha j \alpha^\ddagger j^{-1}}{\nrm(\alpha)}\right).
	\end{align*}
	
\noindent Taking
	\begin{align*}
	\alpha = \begin{pmatrix} a & b \\ c & a \end{pmatrix},
	\end{align*}
	
\noindent we obtain
	\begin{align*}
	\xi_3(\gamma) = \frac{(b^2 + c^2)s^2 - 2bct^2 + 2a^2(s^2 - t^2)}{2 \left(s^2-t^2\right) \left(a^2-b c\right)}.
	\end{align*}
	
\noindent Note that $s^2 - t^2 = j^2 \in p\ZZ_p$, so in fact it shall suffice to show that we can find $a,b,c \in \ZZ_p$ such that for any $\lambda \in \ZZ_p^\times$,
	\begin{align*}
	(b^2 + c^2)s^2 - 2bct^2 + 2a^2(s^2 - t^2) &= \lambda\left(a^2-b c\right) \\
	a^2 - bc &\in \ZZ_p^\times.
	\end{align*}
	
\noindent This can be accomplished by Hensel's lemma. We shall require that $a^2 - bc = u^2 s^2/\lambda$, where $u \in \ZZ_p^\times$ will be chosen later. Then, noting that $s^t - t^2 = 0 \mod p$, we have
	\begin{align*}
	(b^2 + c^2)s^2 - 2bct^2 + 2a^2(s^2 - t^2) &= (b - c)^2 s^2 \mod p,
	\end{align*}
	
\noindent hence if we choose $c = b \pm u$, then
	\begin{align*}
	(b - c)^2 s^2 = u^2 s^2 = \lambda \frac{u^2 s^2}{\lambda} = \lambda\left(a^2-b c\right) \mod p.
	\end{align*}
	
\noindent Of course, this now requires that
	\begin{align*}
	a^2 - bc = a^2 - b(b \pm u) = a^2 \pm bu - b^2 = \frac{u^2 s^2}{\lambda} \mod p,
	\end{align*}
	
\noindent but we have free choice of $u$ and this is a quadratic form over $\ZZ/p\ZZ$, so this is satisfiable. We note that as long as $a \neq 0$, we can lift our solution to $\ZZ_p$ via Hensel's lemma, and so we are done.
\end{proof}

\begin{lemma}\label{TwoImage}
Let $H$ be a quaternion algebra over $\QQ_2$ with involution $\ddagger$. Let $\OO$ be a maximal $\ddagger$-order of $H$ containing an element $j \in H^0 \cap H^+$ of square-free integral norm. Then the image $\xi_3(SL^\ddagger(2,\OO))$ is as given in Table \ref{Two Image}.
\end{lemma}
	\begin{table}
	\begin{align*}\renewcommand\arraystretch{1.5}
	\begin{array}{l|l|l} \OO & \left(m,n\right)/\approx & \xi_3(SL^\ddagger(2,\OO)) \\ \hline
	\ZZ_2 \oplus \ZZ_2 i \oplus \ZZ_2 \frac{i + j}{2} \oplus \ZZ_2 \frac{2 + 2j + ij}{4} & (-6,-6),(-2,6),(2,2),(6,-2) & \frac{1}{2}\ZZ_2 \\
	\ZZ_2 \oplus \ZZ_2 i \oplus \ZZ_2 \frac{i + j}{2} \oplus \ZZ_2 \frac{2 + ij}{4} & (-6,-2),(-2,-6),(2,6),(6,2) & \frac{1}{2}\ZZ_2 \\
	\ZZ_2 \oplus \ZZ_2 i \oplus \ZZ_2 \frac{1 + i + j}{2} \oplus \ZZ_2 \frac{i + ij}{2} & (-6,1),(-2,-3),(2,1),(6,-3) & \frac{1}{2}\ZZ_2 \\
	\ZZ_2 \oplus \ZZ_2 i \oplus \ZZ_2 \frac{1 + j}{2} \oplus \ZZ_2 \frac{i + ij}{2} & (-6,3),(-2,3),(2,3),(6,3) & \frac{1}{2}\ZZ_2 \\
	\ZZ_2 \oplus \ZZ_2 i \oplus \ZZ_2 j \oplus \ZZ_2 \frac{1 + i + j + ij}{2} & (-3,-3),(-3,1),(1,-3),(1,1) & \ZZ_2 \\
	\ZZ_2 \oplus \ZZ_2 i \oplus \ZZ_2 \frac{1 + i + j}{2} \oplus \ZZ_2 \frac{j + ij}{2} & (-3,-2),(-3,6),(1,-6),(1,2) & \ZZ_2 \\
	\ZZ_2 \oplus \ZZ_2 \frac{1 + i}{2} \oplus \ZZ_2 j \oplus \ZZ_2 \frac{j + ij}{2} & (3,-6),(3,-2),(3,2),(3,6) & 1 + 4\ZZ_2 \\ \hdashline
		\begin{cases} \ZZ_2 \oplus \ZZ_2 \frac{1 + i}{2} \oplus \ZZ_2 j \oplus \ZZ_2 \frac{j + ij}{2} \\ \ZZ_2 \oplus \ZZ_2 i \oplus \ZZ_2 \frac{1 + j}{2} \oplus \ZZ_2 \frac{i + ij}{2} \end{cases} & (-1,-1),(-1,3),(3,-1),(3,3) & 1 + 2\ZZ_2 \\
	\begin{cases} \ZZ_2 \oplus \ZZ_2 i \oplus \ZZ_2 \frac{2 + i + j}{4} \oplus \ZZ_2 \frac{i - j + ij}{4} \\ \ZZ_2 \oplus \ZZ_2 i \oplus \ZZ_2 \frac{2 - i + j}{4} \oplus \ZZ_2 \frac{i + j + ij}{4} \end{cases} & (-6,2),(-2,-2),(2,-6),(6,6) & \frac{1}{4}\ZZ_2 \\	\hdashline
	\ZZ_2 \oplus \ZZ_2 i \oplus \ZZ_2 \frac{1 + i + j}{2} \oplus \ZZ_2 \frac{i + ij}{2} & (-6,-3),(-2,1),(2,-3),(6,1) & \frac{1}{2}\ZZ_2  \\
	\ZZ_2 \oplus \ZZ_2 i \oplus \ZZ_2 \frac{1 + j}{2} \oplus \ZZ_2 \frac{i + ij}{2} & (-6,-1),(-2,-1),(2,-1),(6,-1) & \frac{1}{2}\ZZ_2 \\
	\ZZ_2 \oplus \ZZ_2 i \oplus \ZZ_2 \frac{1 + i + j}{2} \oplus \ZZ_2 \frac{j + ij}{2} & (-3,-6),(-3,2),(1,-2),(1,6) & \frac{1}{2}\ZZ_2 \\
	\ZZ_2 \oplus \ZZ_2 \frac{1 + i}{2} \oplus \ZZ_2 j \oplus \ZZ_2 \frac{j + ij}{2} & (-1,-6),(-1,-2),(-1,2),(-1,6) & 1 + 4\ZZ_2 \\ \hdashline
	\begin{cases} \ZZ_2 \oplus \ZZ_2 i \oplus \ZZ_2 \frac{i + j}{4} \oplus \ZZ_2 	\frac{2 + ij}{4} \\ \ZZ_2 \oplus \ZZ_2 i \oplus \ZZ_2 \frac{i - j}{4} \oplus \ZZ_2 	\frac{2 + ij}{4} \end{cases} & (-6,6),(-2,2),(2,-2),(6,-6) & \frac{1}{4}\ZZ_2 \\
	\begin{cases} \ZZ_2 \oplus \ZZ_2 i \oplus \ZZ_2 \frac{i + j}{2} \oplus \ZZ_2 	\frac{1 + ij}{2} \\ \ZZ_2 \oplus \ZZ_2 i \oplus \ZZ_2 \frac{1 + j}{2} \oplus \ZZ_2 	\frac{i + ij}{2} \end{cases} & (-3,1),(-3,3),(1,-1),(1,3) & \frac{1}{2}\ZZ_2 \\
	\begin{cases} \ZZ_2 \oplus \ZZ_2 i \oplus \ZZ_2 \frac{i + j}{2} \oplus \ZZ_2 	\frac{1 + ij}{2} \\ \ZZ_2 \oplus \ZZ_2 \frac{1 + i}{2} \oplus \ZZ_2 j \oplus \ZZ_2 	\frac{j + ij}{2} \end{cases} & (-1,3),(-1,1),(3,-3),(3,1) & 1 + 2\ZZ_2.
	\end{array}
	\end{align*}
	\caption{Images $\xi_3(SL^\ddagger(2,\OO))$ for orders in quaternion algebras over $\QQ_2$.}
	\label{Two Image}
	\end{table}
	
\begin{proof}
By Lemma \ref{TwoIsNotThatBad}, the given orders are the only ones that need to be considered. As there are only finitely many possibilities, we can compute $\xi_3\left(SL^\ddagger(2,\OO \otimes_\ZZ \ZZ/4\ZZ)\right)$ for each of the given orders to conclude that the image of $\xi_3(SL^\ddagger(2,\OO))$ modulo $4$ surjects onto the given sets. To conclude that the entire image is represented, note that
	\begin{align*}
	\xi_3\left(\begin{pmatrix} 1 & 0 \\ \tau & 1 \end{pmatrix} \gamma\right) = \xi_3(\gamma) - \kappa(\gamma)\pi_j(\tau),
	\end{align*}
	
\noindent for any $\tau \in \OO \cap H^+$---by direct computation of $SL^\ddagger(2,\OO \otimes_\ZZ \ZZ/4\ZZ)$, we can be assured that we can always choose $\gamma$ such that this spans the entire set.
\end{proof}

Together with Theorem \ref{StrongApproximation}, Lemmas \ref{LocalImageNotTwo} and \ref{TwoImage} have a number of very useful corollaries.

\begin{corollary}\label{DescribeIntersectionSets}
Given a definite, rational quaternion algebra $H$ with involution $\ddagger$, an element $j \in H^0 \cap H^+$ with square-free integer norm, and a maximal $\ddagger$-order $\OO$ containing $j$, we have
	\begin{align*}
	\xi_3\left(SL^\ddagger(2,\OO)\right) = 1 + 2^\alpha\frac{\nrm(j)}{P^2}\ZZ,
	\end{align*}
	
\noindent where $P$ is the product of odd primes $p$ such that $j\OO_p j^{-1} \neq \OO_p$, and $\alpha \in \{-2,-1,0,1,2\}$ is chosen to match the corresponding entry in Table \ref{Two Image} for $\OO \otimes \ZZ_2$.
\end{corollary}

\begin{corollary}\label{DescribeWhenTangential}
Spheres in $\mathcal{S}_{\OO,j}$ have only tangential intersections if and only if
	\begin{align*}
	2^\alpha\frac{\nrm(j)}{P^2} \geq 2.
	\end{align*}
\end{corollary}

\begin{corollary}\label{DescribeWhenRational}
Spheres in $\mathcal{S}_{\OO,j}$ have only rational intersections if and only if
	\begin{align*}
	X^2 + \nrm(i) Y^2 = 2^\alpha k(2P^2 - 2^\alpha\nrm(j) k)
	\end{align*}
	
\noindent has rational solutions for all integers $0\leq k \leq 2^{1 - \alpha} P^2/\nrm(j)$.
\end{corollary}

\begin{proof}
By Lemma \ref{RationalIntersections}, intersections are rational if and only if
	\begin{align*}
	X^2 + \nrm(i)Y^2 = \nrm(j)\left(1 - \xi_3(\gamma)^2\right)
	\end{align*}
	
\noindent has rational solutions whenever $|\xi_3(\gamma)| \leq 1$. However, we know that
	\begin{align*}
	\xi_3(\gamma) = 1 - 2^\alpha\frac{\nrm(j)}{P^2}k
	\end{align*}
	
\noindent for some integer $k$, so
	\begin{align*}
	\nrm(j)\left(1 - \xi_3(\gamma)\right)^2 &= \nrm(j)\frac{2^\alpha \nrm(j)k(2P^2 - 2^\alpha\nrm(j)k)}{P^4} \\
	&= \frac{\nrm(j)^2}{P^4}\left(2^\alpha k(2P^2 - 2^\alpha\nrm(j) k)\right),
	\end{align*}
	
\noindent which concludes the proof.
\end{proof}

\section{Orders with Euclidean Intersection:}\label{DefinitionOfEuclideanIntersection}

Corollary \ref{DescribeWhenTangential} reduces determining whether intersections in $\mathcal{S}_{\OO,j}$ are tangential to a straightforward computation. This leaves the problem of determining when the packing $\mathcal{S}_{\OO,j}$ is tangency-connected---for this, we shall need to determine for what $\gamma \in SL^\ddagger(2,\OO)$ the image $\gamma \hat{S}_j$ intersects $\hat{S}_j$. By Lemma \ref{BasicImplication}, we know that tangencies occur at rational points of $\hat{S}_j$. However, for arbitrary groups $SL^\ddagger(2,\OO)$, the action on rational points of $\hat{S}_j$ is not transitive; as a result, it is unclear how to determine the tangency-structure in general. In fact, the intersection of $SL^\ddagger(2,\OO)$ with $SL(2,\CC)$ is generally going to be difficult to describe. To simplify the problem, we instead consider the special case where the intersection of $\OO$ with $\CC$ is a Euclidean ring.

\begin{definition}
Let $H = \left(\frac{-m,-n}{\QQ}\right)$ for some square-free positive integers $m,n$ such that the ring of integers of $\QQ(i)$ is a Euclidean ring---equivalently, when $m = 1,2,3,7,11$. We equip $H$ with the usual involution $\ddagger$. We say that a maximal $\ddagger$-order $\OO$ has \emph{Euclidean intersection} if $\OO \cap \QQ(i)$ is the ring of integers of $\QQ(i)$.
\end{definition}
	
\noindent We shall show that for orders with Euclidean intersection, the converse to Lemma \ref{BasicImplication} is true. We start by giving a classification of all such orders.

\begin{theorem}\label{All Orders with Euclidean Intersection}
Let $H = \left(\frac{-m,-n}{\QQ}\right)$ and let $\OO$ be a maximal $\ddagger$-order with Euclidean intersection. Then $\OO$ is one of the orders given in Table \ref{Orders with Euclidean intersection}, with the given image of $\xi_3(SL^\ddagger(2,\OO))$.
\end{theorem}

	\begin{table}
	\begin{align*}
	\begin{array}{l|lll}
	m & \OO & \text{Conditions} & \xi_3(SL^\ddagger(2,\OO) \\ \hline
	m = 1 &	\ZZ \oplus \ZZ i \oplus \ZZ j \oplus \frac{1 + i + j + ij}{2} & \text{if } n \equiv 1 \mod 4 & 1 + n\ZZ \\
			& \ZZ \oplus \ZZ i \oplus \ZZ \frac{1 + i + j}{2} \oplus \ZZ \frac{j + ij}{2} & \text{if } n\equiv 2 \mod 4 & 1 + \frac{n}{2}\ZZ \\
			& \begin{cases} \ZZ \oplus \ZZ i \oplus \ZZ \frac{1 + j}{2} \oplus \ZZ \frac{i + ij}{2} \\ \ZZ \oplus \ZZ i \oplus \ZZ \frac{i + j}{2} \oplus \ZZ \frac{1 + ij}{2} \end{cases} & \text{if } n \equiv -1 \mod 4 & 1 + \frac{n}{2}\ZZ
		\\ \hline	
	m = 2 & \ZZ \oplus \ZZ i \oplus \ZZ \frac{1 + i + j}{2} \oplus \ZZ \frac{i + ij}{2} & \text{if } n \equiv 1 \mod 4 & 1 + \frac{n}{2}\ZZ\\
			& \ZZ \oplus \ZZ i \oplus \ZZ \frac{1 + j}{2} \oplus \ZZ \frac{i + ij}{2} & \text{if } n \equiv -1 \mod 4 & 1 + \frac{n}{2}\ZZ \\
			& \ZZ \oplus \ZZ i \oplus \ZZ \frac{i + j}{2} \oplus \ZZ \frac{2 + 2j + ij}{4} & \text{if } n/2 \equiv 1 \mod 8 & 1 + \frac{n}{4}\ZZ \\
			& \ZZ \oplus \ZZ i \oplus \ZZ \frac{i + j}{2} \oplus \ZZ \frac{2 + ij}{4} & \text{if } n/2 \equiv 3 \mod 8 & 1 + \frac{n}{4}\ZZ \\
			& \begin{cases} \ZZ \oplus \ZZ i \oplus \ZZ \frac{2 + i + j}{4} \oplus \ZZ \frac{i - j + ij}{4} \\ \ZZ \oplus \ZZ i \oplus \ZZ \frac{2 + i - j}{4} \oplus \ZZ \frac{i + j + ij}{4} \end{cases} & \text{if } n/2 \equiv -3 \mod 8 & 1 + \frac{n}{8}\ZZ \\
			& \begin{cases} \ZZ \oplus \ZZ i \oplus \ZZ \frac{i + j}{4} \oplus \ZZ \frac{2 + ij}{4} \\ \ZZ \oplus \ZZ i \oplus \ZZ \frac{i - j}{4} \oplus \ZZ \frac{2 + ij}{4} \end{cases} & \text{if } n/2 \equiv -1 \mod 8 & 1 + \frac{n}{8}\ZZ
		\\ \hline
	m = 3 & \ZZ \oplus \ZZ \frac{1 + i}{2} \oplus \ZZ j \oplus \ZZ \frac{j + ij}{2} & \text{if } 3\nmid n & 1 + 2n\ZZ \\
			& \ZZ \oplus \ZZ \frac{1 + i}{2} \oplus \ZZ j \oplus \ZZ \frac{3j + ij}{6} & \text{if } 3|n \text{ and } n/3 = 1 \mod 3 & 1 + \frac{2n}{3}\ZZ \\
			& \begin{cases} \ZZ \oplus \ZZ \frac{1 + i}{2} \oplus \ZZ \frac{i + j}{3} \oplus \ZZ \frac{3j + ij}{6} \\ \ZZ \oplus \ZZ \frac{1 + i}{2} \oplus \ZZ \frac{i - j}{3} \oplus \ZZ \frac{3j + ij}{6} \end{cases} & \text{if } 3|n \text{ and } n/3 = -1 \mod 3 & 1 + \frac{2n}{9}\ZZ \\ \hline
	m = 7 & \ZZ \oplus \ZZ \frac{1 + i}{2} \oplus \ZZ j \oplus \ZZ \frac{j + ij}{2} & \text{if } 7\nmid n & 1 + 2n\ZZ \\
			& \ZZ \oplus \ZZ \frac{1 + i}{2} \oplus \ZZ j \oplus \ZZ \frac{7j + ij}{14} & \text{if } 7|n & 1 + \frac{2n}{7}\ZZ \\ \hline
	m = 11 & \ZZ \oplus \ZZ \frac{1 + i}{2} \oplus \ZZ j \oplus \ZZ \frac{j + ij}{2} & \text{if } 11\nmid n & 1 + 2n\ZZ \\
			& \ZZ \oplus \ZZ \frac{1 + i}{2} \oplus \ZZ j \oplus \ZZ \frac{11j + ij}{22} & \text{if } 11|n \text{ and } \left(\frac{n/11}{11}\right) = 1 & 1 + \frac{2n}{11}\ZZ \\
			& \begin{cases} \ZZ \oplus \ZZ \frac{1 + i}{2} \oplus \ZZ \frac{3i + j}{11} \oplus \ZZ \frac{11j + ij}{22} \\ \ZZ \oplus \ZZ \frac{1 + i}{2} \oplus \ZZ \frac{3i - j}{11} \oplus \ZZ \frac{11j + ij}{22} \end{cases} & \text{if } 11|n \text{ and } n/11 \equiv 2 \mod 11 & 1 + \frac{2n}{11^2}\ZZ \\
			& \begin{cases} \ZZ \oplus \ZZ \frac{1 + i}{2} \oplus \ZZ \frac{4i + j}{11} \oplus \ZZ \frac{11j + ij}{22} \\ \ZZ \oplus \ZZ \frac{1 + i}{2} \oplus \ZZ \frac{4i - j}{11} \oplus \ZZ \frac{11j + ij}{22} \end{cases} & \text{if } 11|n \text{ and } n/11 \equiv -5 \mod 11 & 1 + \frac{2n}{11^2}\ZZ \\
			& \begin{cases} \ZZ \oplus \ZZ \frac{1 + i}{2} \oplus \ZZ \frac{2i + j}{11} \oplus \ZZ \frac{11j + ij}{22} \\ \ZZ \oplus \ZZ \frac{1 + i}{2} \oplus \ZZ \frac{2i - j}{11} \oplus \ZZ \frac{11j + ij}{22} \end{cases} & \text{if } 11|n \text{ and } n/11 \equiv -4 \mod 11 & 1 + \frac{2n}{11^2}\ZZ \\
			& \begin{cases} \ZZ \oplus \ZZ \frac{1 + i}{2} \oplus \ZZ \frac{5i + j}{11} \oplus \ZZ \frac{11j + ij}{22} \\ \ZZ \oplus \ZZ \frac{1 + i}{2} \oplus \ZZ \frac{5i - j}{11} \oplus \ZZ \frac{11j + ij}{22} \end{cases} & \text{if } 11|n \text{ and } n/11 \equiv -3 \mod 11 & 1 + \frac{2n}{11^2}\ZZ \\
			& \begin{cases} \ZZ \oplus \ZZ \frac{1 + i}{2} \oplus \ZZ \frac{i + j}{11} \oplus \ZZ \frac{11j + ij}{22} \\ \ZZ \oplus \ZZ \frac{1 + i}{2} \oplus \ZZ \frac{i - j}{11} \oplus \ZZ \frac{11j + ij}{22} \end{cases} & \text{if } 11|n \text{ and } n/11 \equiv -1 \mod 11 & 1 + \frac{2n}{11^2}\ZZ
	\end{array}
	\end{align*}
	\caption{All orders with Euclidean intersection.}
	\label{Orders with Euclidean intersection}
	\end{table}

\begin{proof}
For $m = 1$ and $m = 2$, there are no odd primes dividing $m$, and therefore by Lemma \ref{MostPrimesAreUnnecessary} there is only one maximal $\ddagger$-order in $H_p$ for every odd prime $p$. Consequently, maximal $\ddagger$-orders in $H$ are wholly determined by the orders specified in Lemma \ref{TwoIsNotThatBad}. For all other $m$, there is a unique maximal $\ddagger$-order in $H_2$ containing $(1 + i)/2$, specifically
	\begin{align*}
	\ZZ \oplus \ZZ \frac{1 + i}{2} \oplus \ZZ j \oplus \ZZ \frac{j + ij}{2}.
	\end{align*}
	
\noindent It therefore suffices to determine the behavior over $\QQ_m$. If $m = 7$, by a computation with the Hilbert symbol shows that either $7$ ramifies or $-\disc(\ddagger) \neq \left(\QQ_7^\times\right)^2$. Consequently, by Lemma \ref{MostPrimesAreUnnecessary}, there is but one maximal $\ddagger$-order in $H_p$ for every prime $p$. We conclude that the maximal $\ddagger$-orders given in the table are the only ones possible. If $m = 3,11$, there is again just one maximal $\ddagger$-order, unless $m|n$ and $n/m$ is not a square modulo $m$, in which case $m$ does not ramify, $-\disc(\ddagger) = \left(\QQ_m^\times\right)^2$, and consequently, there are two distinct maximal $\ddagger$-orders. Since $n/m$ is not a square modulo $m$, $-n/m$ is, which implies that there exists an integer $t$ such that $n/m + t^2 \equiv 0 \mod m$, and therefore the orders given in the table are well-defined. Finally, we compute the sets $\xi_3(SL^\ddagger(2,\OO))$ by appealing to Corollary \ref{DescribeIntersectionSets}.
\end{proof}

Orders with Euclidean intersection are very well-behaved---for example, for almost all of them, intersections in $\mathcal{S}_{\OO,j}$ are rational, and barring a small number of exceptions, they are in fact tangential.

\begin{theorem}\label{WhenRationalAndTangential}
Let $H = \left(\frac{-m,-n}{\QQ}\right)$ be a rational, definite quaternion algebra with a maximal $\ddagger$-order $\OO$ with Euclidean intersection. Then intersections in $\mathcal{S}_{\OO,j}$ are rational unless $m,n$ are on the following list.
	\begin{align*}
	\begin{array}{l|l}
	m & n \\ \hline
	3 & 6 \\
	11 & 22, 66, 77, 110
	\end{array}
	\end{align*}
	
\noindent Furthermore, intersections are tangential unless $m,n$ are on the following list.

	\begin{minipage}{.45\textwidth}
	\begin{align*}
	\begin{array}{l|l}
	m & n \\ \hline
	1 & 1,2,3 \\
	2 & 1, 2, 3, 6, 10, 14
	\end{array}
	\end{align*}
	\end{minipage}%
	\begin{minipage}{.45\textwidth}
	\begin{align*}
	\begin{array}{l|l}
	3 & 6 \\
	11 & 22, 66, 77, 110.
	\end{array}
	\end{align*}
	\end{minipage}
\end{theorem}

\begin{proof}
It is a consequence of Corollary \ref{DescribeWhenTangential} that unless $m,n$ are on the given list then all intersections are tangential. For the remaining finite cases, we apply Corollary \ref{DescribeWhenRational}.
\end{proof}	

Furthermore, we can give a precise description of how two spheres intersect in $\mathcal{S}_{\OO,j}$ if $\OO$ has Euclidean intersection. We begin by looking at spheres intersecting $\hat{S}_j$ at $\infty$.

\begin{lemma}\label{IntersectionsWithPlane}
Let $H = \left(\frac{-m,-n}{\QQ}\right)$ be a definite quaternion algebra with a maximal $\ddagger$-order $\OO$ with Euclidean intersection. Suppose $S \in \mathcal{S}_{\OO,j}$ and $\infty \in S \cap \hat{S}_j$. Then,
	\begin{align*}
	S = \begin{pmatrix} u & 0 \\ 0 & \left(u^\ddagger\right)^{-1} \end{pmatrix} \begin{pmatrix} 1 & \tau \\ 0 & 1 \end{pmatrix}(\hat{S}_j),
	\end{align*}
	
\noindent for some $u \in \OO^\times$, $\tau \in \OO \cap H^+$.
\end{lemma}

\begin{proof}
Choose $\gamma \in SL^\ddagger(2,\OO)$ such that $S = \gamma \hat{S}_j$. The point $p = \gamma^{-1}(\infty)$ must be a rational point of $\hat{S}_j$. Since $\OO \cap \QQ(i)$ is a Euclidean ring, the action of $SL\left(2,\OO \cap \QQ(i)\right)$ is transitive on the rational points of $\hat{S}_j$. Therefore, we can find an element $\gamma' \in SL^\ddagger\left(2,\OO \cap \QQ(i)\right)$ such that $\gamma'(p) = \infty$. We have
	\begin{align*}
	S &= \gamma\gamma'\left(\hat{S}_j\right)
	\end{align*}
	
\noindent and $\gamma\gamma'(\infty) = \infty$, hence
	\begin{align*}
	\gamma\gamma' &= \begin{pmatrix} u & * \\ 0 & \left(u^\ddagger\right)^{-1} \end{pmatrix} \\
	&=\begin{pmatrix} u & 0 \\ 0 & \left(u^\ddagger\right)^{-1} \end{pmatrix}\begin{pmatrix} 1 & \tau \\ 0 & 1 \end{pmatrix},
	\end{align*}
	
\noindent where $u \in \OO^\times$, $\tau \in \OO \cap H^+$.
\end{proof}

\begin{theorem}\label{IntersectionsOfSpheres}
Let $H = \left(\frac{-m,-n}{\QQ}\right)$ be a definite quaternion algebra with maximal $\ddagger$-order $\OO$ with Euclidean intersection. Choose any sphere $S \in \mathcal{S}_{\OO,j}$ and any rational point $z \in S$. There exist $a,b,c,d \in \OO$ such that $ac^{-1} = z$,
	\begin{align*}
	\begin{pmatrix} a & b \\ c & d \end{pmatrix} \in SL^\ddagger(2,\OO),
	\end{align*}
	
\noindent and
	\begin{align*}
	S = \begin{pmatrix} a & b \\ c & d \end{pmatrix}(\hat{S}_j).
	\end{align*}
	
\noindent Furthermore, the spheres intersecting $S$ at $z$ are given by
	\begin{align*}
	\begin{pmatrix} a & b \\ c & d \end{pmatrix}\begin{pmatrix} u & 0 \\ 0 & {u^\ddagger}^{-1} \end{pmatrix} \begin{pmatrix} 1 & \tau \\ 0 & 1 \end{pmatrix}\left(\hat{S}_j\right)
	\end{align*}
	
\noindent for $u \in \OO^\times$, $\tau \in \OO \cap H^+$.
\end{theorem}

\begin{proof}
We can assume without loss of generality that $S = \hat{S}_j$ since we can always apply a transformation $\gamma \in SL^\ddagger(2,\OO)$ to move $S$ back to $\hat{S}_j$. Choose $\gamma \in SL^\ddagger\left(2,\OO \cap \QQ(i)\right)$ such that $\gamma(\infty) = z$. Then $\gamma(\hat{S}_j) = \hat{S}_j$, and so
	\begin{align*}
	\gamma = \begin{pmatrix} a & b \\ c & d \end{pmatrix}
	\end{align*}
	
\noindent has the desired properties. Furthermore, if $S'$ is a sphere intersecting $\hat{S}_j$ at $z$, then $\gamma^{-1}(S')$ is a sphere intersecting $\hat{S}_j$ at $\infty$. However, by Lemma \ref{IntersectionsWithPlane} we must have
	\begin{align*}
	\gamma^{-1}(S') = \begin{pmatrix} u & 0 \\ 0 & {u^\ddagger}^{-1} \end{pmatrix} \begin{pmatrix} 1 & \tau \\ 0 & 1 \end{pmatrix}\left(\hat{S}_j\right),
	\end{align*}
	
\noindent whence
	\begin{align*}
	S' = \begin{pmatrix} a & b \\ c & d \end{pmatrix}\begin{pmatrix} u & 0 \\ 0 & {u^\ddagger}^{-1} \end{pmatrix} \begin{pmatrix} 1 & \tau \\ 0 & 1 \end{pmatrix}\left(\hat{S}_j\right).
	\end{align*}
	
\noindent Conversely, it is easy to see that any sphere of the given form must intersect $\hat{S}_j$ at $z$.
\end{proof}

Theorem \ref{IntersectionsOfSpheres} directly implies the converse to Lemma \ref{BasicImplication}, as promised.

\begin{corollary}\label{ImportantConverse}
Let $\OO$ be a maximal $\ddagger$-order with Euclidean intersection. If $\OO^\times$ preserves $\hat{S}_j$ and spheres in $\mathcal{S}_{\OO,j}$ intersect rationally, then all intersections in $\mathcal{S}_{\OO,j}$ are tangential.
\end{corollary}

\section{The Group $E^\ddagger(2,\OO)$ and Covers of $\OO \cap H^+$:}

Having found a collection of orders that admit only tangential intersections, we turn to the requirement that $\mathcal{S}_{\OO,j}$ is tangency-connected. To this end, we define $E^\ddagger(2,\OO)$ to be the group generated by upper and lower triangular matrices in $SL^\ddagger(2,\OO)$. The group $E^\ddagger(2,\OO)$ is an analog of the group $E(2,\OO)$ generated by elementary matrices when $SL(2,\OO)$ is a Bianchi group. We also define the intersection graph of $\mathcal{S}_{\OO,j}$ as the graph where the vertices are spheres in $\mathcal{S}_{\OO,j}$, and any edges connect spheres that intersect.

\begin{theorem}\label{IntersectionGraph}
Let $\OO$ be a maximal $\ddagger$-order that has Euclidean intersection and such that $\mathcal{S}_{\OO,j}$ admits only rational intersections. Then the connected components of the intersection graph of $\mathcal{S}_{\OO,j}$ are in bijection with the coset space $SL^\ddagger(2,\OO)/E^\ddagger(2,\OO)$.
\end{theorem}

\begin{proof}
Note that given a sphere $S \in \mathcal{S}_{\OO,j}$, the element $\gamma \in SL^\ddagger(2,\OO)$ such that $S = \gamma(\hat{S}_j)$ is well-defined up to multiplication on the right by an element of $E\left(2,\OO \cap \QQ(i)\right) = SL\left(2,\OO \cap \QQ(i)\right)$. Therefore, by Theorem \ref{IntersectionsOfSpheres}, we know that if two vertices $S_1, S_2$ are connected by an edge and $S_1 = \gamma_1(\hat{S}_j)$, $S_2 = \gamma_2(\hat{S}_j)$, then there is an element $\gamma \in E^\ddagger(2,\OO)$ such that $\gamma_1 = \gamma_2 \gamma$. Therefore, if two spheres $S_1 = \gamma_1(\hat{S}_j)$, $S_2 = \gamma_2(\hat{S}_j)$ are in the same connected component of the intersection graph, this implies that $\gamma_1, \gamma_2$ are in the same coset of $SL^\ddagger(2,\OO)/E^\ddagger(2,\OO)$. On the other hand, given $\gamma \in E^\ddagger(2,\OO)$, it is easy to see that $\gamma(\hat{S}_j)$ is in the same connected component as $\hat{S}_{j}$, since all of the generators of $E^\ddagger(2,\OO)$ send $\hat{S}_j$ to a sphere that intersects $\hat{S}_j$. Ergo, we have a well-defined map
	\begin{align*}
	SL^\ddagger(2,\OO)/E^\ddagger(2,\OO) &\rightarrow \left\{\text{connected components of } \mathcal{S}_{\OO,j}\right\} \\
	\gamma E^\ddagger(2,\OO) &\mapsto \gamma E^\ddagger(2,\OO)(\hat{S}_{j_\RR}),
	\end{align*}
	
\noindent and this map is the desired bijection.
\end{proof}

Thus, whether or not a sphere packing is tangency-connected is wholly determined by whether or not $E^\ddagger(2,\OO) = SL^\ddagger(2,\OO)$. We shall see that this depends on whether or not unit balls centered on points in $\OO \cap H^+$ cover $H^+$---we shall say that $\OO$ \emph{covers} $\RR^3$ \emph{by unit balls}. If $\OO$ covers $\RR^3$ by unit balls, we have a division lemma.

\begin{lemma}\label{DivisionLemma}
Suppose $\OO$ is a $\ddagger$-order that covers $\RR^3$ by unit balls. Let $a,b \in \OO$ such that $\overline{b}a \in \OO \cap H^+$ and $b \neq 0$. Then there exists $q,r \in \OO \cap H^+$ such that $\nrm(r) < \nrm(b)$ and $a = bq + r$.
\end{lemma}

\begin{proof}
We have $b^{-1}a \in H^+$ and therefore we can find $q \in \OO \cap H^+$ such that $N\left(b^{-1}a - q\right) < 1$. But this means that we can rearrange
	\begin{align*}
	b^{-1}a &= q + (b^{-1}a - q)
	\end{align*}
	
\noindent to give
	\begin{align*}
	a &= bq + (a - bq) \\
	&= bq + r,
	\end{align*}
	
\noindent where $r = a - bq$, and so $\nrm(r) < \nrm(b)$.
\end{proof}

Using this lemma, we define an analog of the Euclidean algorithm.

\begin{lemma}\label{EuclideanAlgorithm}
Suppose $\OO$ is a $\ddagger$-order that covers $\RR^3$ by unit balls. Let $a,b \in \OO$ such that $\overline{b}a \in \OO \cap H^+$ and $b \neq 0$. Consider any sequence defined recursively by:

\begin{equation*}
	\begin{aligned}
	r_0 &= a \\
	r_1 &= b \\
	r_2 &= a - b q_1 \\
	&\vdots \\
	r_l &= r_{l - 2} - r_{l - 1}q_{l - 1},
	\end{aligned}\qquad
	\begin{aligned}
	s_0 &= 1 \\
	s_1 &= 0 \\
	s_2 &= s_0 - s_1 q_1 \\
	&\vdots \\
	s_l &= s_{l - 2} - s_{l - 1}q_{l - 1} 
	\end{aligned}\qquad
	\begin{aligned}
	t_0 &= 0 \\
	t_1 &= 1 \\
	t_2 &= t_0 - t_1 q_1 \\
	&\vdots \\
	t_l &= t_{l - 2} - t_{l - 1} q_{l - 1}
	\end{aligned}
\end{equation*}
	
\noindent where $q_l, r_l \in \OO \cap H^+$ and $q_l$ is chosen such that $nrm(b) > nrm(r_0) > nrm(r_1) > \ldots$. Then:

	\begin{enumerate}
		\item $r_l = 0$ for sufficiently large $l$.
		\item If $k$ is the smallest value for which $r_{k + 1} = 0$, then $r_k$ is a left gcd of $a$ and $b$.
		\item $a s_k + b t_k = r_k$, and $s_k \overline{t_k} \in \OO \cap H^+$.
		\item For all integers $l$ such that $k + 1 \geq l \geq 0$,
			\begin{align*}
			\begin{pmatrix} s_{l - 1} & (-1)^{l - 1} s_l \\ t_{l - 1} & (-1)^{l - 1} t_l \end{pmatrix} &\in E^\ddagger(2,\OO).
			\end{align*}
	\end{enumerate}
\end{lemma}

\begin{proof}
The first part of the claim follows immediately from the fact that there are only finitely many integers between $\nrm(b)$ and $0$. For the second part, we first show that $r_k$ is a left divisor of $a$ and $b$. We prove this inductively, by showing that $r_k$ divides $r_{k'}$ for all $k' < k$.

Note that since $r_{k + 1} = 0$, we must have $r_{k - 1} = r_k q_{k + 1}$ and of course $r_k = r_k$, which proves the base cases. Now, assume that $r_k$ is a left divisor of $r_{k - l}$ and $r_{k - l + 1}$; we must show $r_k$ is a left divisor of $r_{k - l - 1}$. This is immediate, since
	\begin{align*}
	r_{k - l - 1} &= r_{k - l}q_{k - l} + r_{k - l + 1} \\
	&= r_k \left(r_{k - l}' q_{k - l} + r_{k - l + 1}'\right).
	\end{align*}
	
\noindent Next, we show that if $r$ is a left divisor of $a$ and $b$, then $r$ is a left divisor of $r_k$. We also prove this inductively, by showing that $r$ is a left divisor of $r_k$ for all $k$. The base cases of $a$ and $b$ is known by assumption, so it suffices to assume that $r$ is a left divisor of $r_{l - 1}$ and $r_{l - 2}$ and to prove that it is a left divisor of $r_l$. This is also immediate, since
	\begin{align*}
	r_l &= r_{l - 2} - r_{l - 1} q_{l - 1} \\
	r_l &= r\left(r_{l - 2}' - r_{l - 1}'q_{l - 1}\right).
	\end{align*}
	
\noindent To prove that $as_k + b t_k = r_k$, we prove $as_l + b t_l = r_l$ for all $l$ by induction. The base cases are easy to check
	\begin{align*}
	s_0 a + t_0 b &= a = r_0 \\
	s_1 a + t_1 b &= b = r_1
	\end{align*}
	
\noindent So, we assume that $as_{l - 1} + bt_{l - 1} = r_{l - 1}$, and $a s_{l - 2} + b t_{l - 2} = r_{l - 2}$ and show that claim holds for $l$. Indeed
	\begin{align*}
	a s_l + b t_l &= a\left(s_{l - 2} - s_{l - 1}q_{l - 1}\right) + b\left(t_{l - 2} - t_{l - 1}q_{l - 1}\right) \\
	&= \left(as_{l - 2} + bt_{l - 2}\right) - \left(as_{l - 1} + bt_{l - 1}\right)q_{l - 1} \\
	&= r_{l - 2} - r_{l - 1}q_{l - 1} \\
	&= r_l.
	\end{align*}
	
\noindent Next, note that
	\begin{align*}
	\begin{pmatrix} s_{l - 1} &  -s_l \\ t_{l - 1} & -t_l \end{pmatrix}\begin{pmatrix} 0 & 1 \\ - 1 & q_l \end{pmatrix} = \begin{pmatrix} s_l & s_{l + 1} \\ t_l & t_{l + 1} \end{pmatrix} \\
	\begin{pmatrix} s_{l - 1} &  s_l \\ t_{l - 1} & t_l \end{pmatrix}\begin{pmatrix} 0 & -1 \\ 1 & q_l \end{pmatrix} = \begin{pmatrix} s_l & -s_{l + 1} \\ t_l & -t_{l + 1} \end{pmatrix}.
	\end{align*}
	
\noindent Since 
	\begin{align*}
	\begin{pmatrix} s_0 & s_1 \\ t_0 & t_1 \end{pmatrix} = \begin{pmatrix} 1 & 0 \\ 0 & 1 \end{pmatrix} &\in E^\ddagger(2,\OO)
	\end{align*}
	
\noindent we must have
	\begin{align*}
	\begin{pmatrix} s_{l - 1} & (-1)^{l - 1} s_l \\ t_{l - 1} & (-1)^{l - 1} t_l \end{pmatrix} &\in E^\ddagger(2,\OO),
	\end{align*}
	
\noindent for all $l$. It is an immediate consequence that $s_k \overline{t_k} \in \OO \cap H^+$.
\end{proof}

Much like the existence of the standard Euclidean algorithm implies that $SL(2,\ZZ)$ is generated by elementary matrices, the existence of the algorithm defined by Lemma \ref{EuclideanAlgorithm} implies that $SL^\ddagger(2,\OO) = E^\ddagger(2,\OO)$.

\begin{lemma}\label{ConsequenceOfEuclideanAlgorithm}
If $\OO$ is a $\ddagger$-order that covers $\RR^3$ by unit balls, then $SL^\ddagger(2,\OO) = E^\ddagger(2,\OO)$.
\end{lemma}

\begin{proof}
Choose any element
	\begin{align*}
	\gamma = \begin{pmatrix} a & b \\ c & d \end{pmatrix} \in SL^\ddagger(2,\OO).
	\end{align*}
	
\noindent If $b = 0$, then it follows that $\gamma \in E^\ddagger(2,\OO)$. So, assume $b \neq 0$, and consider $\overline{b}a$. Since
	\begin{align*}
	\gamma^{-1} &= \pm\begin{pmatrix} d^\ddagger & -b^\ddagger \\ -c^\ddagger & a^\ddagger \end{pmatrix},
	\end{align*}
	
\noindent we must have
	\begin{align*}
	-b^\ddagger\left(a^{-1}\right)^{\ddagger} &\in H^+
	\end{align*}
	
\noindent and so $\overline{b}a \in \OO \cap H^+$. We can therefore apply the algorithm described in Lemma \ref{EuclideanAlgorithm}. Since $a d^\ddagger - b c^\ddagger = 1$, we see that the greatest common divisor $g$ of $a,b$ must be a unit in $\OO$. Therefore
	\begin{align*}
	\begin{pmatrix} a & b \\ -{g^{-1}}^\ddagger t_k^\ddagger & {g^{-1}}^\ddagger s_k^\ddagger \end{pmatrix} \in SL^\ddagger(2,\OO)
	\end{align*}

\noindent and since
	\begin{align*}
	\begin{pmatrix} 1 & 0 \\ z & 1 \end{pmatrix} \begin{pmatrix} a & b \\ -{g^{-1}}^\ddagger t_k^\ddagger & {g^{-1}}^\ddagger s_k^\ddagger \end{pmatrix} = \begin{pmatrix} a & b \\ c & d \end{pmatrix}
	\end{align*}
	
\noindent for some $z \in \OO \cap H^+$, it suffices to show that this new matrix is in $E^\ddagger(2,\OO)$. We can simplify this even further by noting that
	\begin{align*}
	\begin{pmatrix} a & b \\ -{g^{-1}}^\ddagger t_k^\ddagger & {g^{-1}}^\ddagger s_k^\ddagger \end{pmatrix}^{-1} &= \begin{pmatrix} s_k g^{-1} & -b^\ddagger \\ t_k g^{-1} & a^\ddagger\end{pmatrix} \\
	&= \begin{pmatrix} s_k & -b^\ddagger {g^{-1}}^\ddagger \\ t_k & a^\ddagger {g^{-1}}^\ddagger \end{pmatrix}\begin{pmatrix} g^{-1} & 0 \\ 0 & g^\ddagger \end{pmatrix},
	\end{align*}

\noindent so the problem is equivalent to showing that
	\begin{align*}
	\begin{pmatrix} s_k & -b^\ddagger {g^{-1}}^\ddagger \\ t_k & a^\ddagger {g^{-1}}^\ddagger \end{pmatrix} \in E^\ddagger(2,\OO).
	\end{align*}
	
\noindent But, since
	\begin{align*}
	\begin{pmatrix} s_k & -b^\ddagger {g^{-1}}^\ddagger \\ t_k & a^\ddagger {g^{-1}}^\ddagger \end{pmatrix} &= \begin{pmatrix} s_k & (-1)^k s_{k + 1} \\ t_k & (-1)^k t_{k + 1} \end{pmatrix}\begin{pmatrix} 1 & z \\ 0 & 1 \end{pmatrix}
	\end{align*}

\noindent for some $z \in \OO \cap H^+$, we are done.
\end{proof}

\begin{theorem}\label{UnitBallsTheorem}
If $\OO \subset \left(\frac{-m,-n}{\QQ}\right)$ is a maximal $\ddagger$-order with Euclidean intersection, then it covers $\RR^3$ by unit balls if and only if $(m,n)$ is on the following table.
	\begin{align*}
	\begin{array}{l|l}
	m & n \\ \hline
	1 & 1,2,3,6,7,10 \\
	2 & 1,2,3,5,6,10,14,26 \\
	3 & 1,2,6,15 \\
	7 & 1 \\
	11 & 22,66,77,110,143
	\end{array}
	\end{align*}
\end{theorem}

	\begin{table}
	\arraycolsep=1.4pt\def\arraystretch{0.5}
	\begin{align*}
	\begin{array}{l|ll}
	m & \OO \cap H^+ & G \\ \hline
	m = 1 &	\ZZ \oplus \ZZ i \oplus \ZZ j & \sqrt{\frac{n}{n - 2}}\left(2,2,1 + i + j\right) \\
				& \ZZ \oplus \ZZ i \oplus \ZZ \frac{1 + i + j}{2} & \sqrt{\frac{n^2}{n^2 - 12n + 4}}\left(4,4, 2 + 2i + \frac{n - 2}{n}j\right) \\
				& \ZZ \oplus \ZZ i \oplus \ZZ \frac{1 + j}{2} & \sqrt{\frac{n^2}{n^2 - 10n + 1}}\left(4,4, 2 + 2i + \frac{n - 1}{n}j\right) \\
				& \ZZ \oplus \ZZ i \oplus \ZZ \frac{i + j}{2} & \sqrt{\frac{n^2}{n^2 - 10n + 1}}\left(4,4, 2 + 2i + \frac{n - 1}{n}j\right) \\ \hline
	m = 2 & \ZZ \oplus \ZZ i \oplus \ZZ \frac{1 + i + j}{2} & \sqrt{\frac{n^2}{(n - 1)(n - 9)}}\left(4,4, 2 + 2i + \frac{n - 3}{n}j\right) \\
				& \ZZ \oplus \ZZ i \oplus \ZZ \frac{1 + j}{2} & \sqrt{\frac{n^2}{n^2 - 6n + 1}}\left(4,4, 2 + 2i + \frac{n - 1}{n}j\right) \\
				& \ZZ \oplus \ZZ i \oplus \ZZ \frac{i + j}{2} & \sqrt{\frac{n^2}{n^2 - 8n + 4}}\left(4,4, 2 + 2i + \frac{n - 2}{n}j\right) \\
				& \ZZ \oplus \ZZ i \oplus \ZZ \frac{2 + i + j}{4} & \sqrt{\frac{n^2}{n^2 - 36n + 100}}\left(8,8, 4 + 4i + \frac{n - 10}{n}j\right) \\
				& \ZZ \oplus \ZZ i \oplus \ZZ \frac{i + j}{4} & \sqrt{\frac{n^2}{n^2 - 28n + 36}}\left(8,8, 4 + 4i + \frac{n - 6}{n}j\right) \\
	m = 3 & \ZZ \oplus \ZZ \frac{1 + i}{2} \oplus \ZZ j & \sqrt{\frac{3n}{3n - 8}}\left(2,2, 1 + \frac{i}{3} + j\right) \\
				& \ZZ \oplus \ZZ \frac{1 + i}{2} \oplus \ZZ \frac{i + j}{3} & \sqrt{\frac{n}{n - 24}}\left(6,6, 3 + i + j\right) \\ \hline
	m = 7 & \ZZ \oplus \ZZ \frac{1 + i}{2} \oplus \ZZ j & \sqrt{\frac{7n}{7n - 12}}\left(2,2, 1 + \frac{3i}{7} + j\right) \\ \hline
	m = 11 & \ZZ \oplus \ZZ \frac{1 + i}{2} \oplus \ZZ j & \sqrt{\frac{11n}{11n - 8}}\left(2,2, 1 + \frac{5i}{11} + j\right) \\
				& \ZZ \oplus \ZZ \frac{1 + i}{2} \oplus \ZZ \frac{3i + j}{11} & \sqrt{\frac{n^2}{(n - 22)(n - 198)}}\left(22,22, 11 + 5i + \frac{n - 66}{n}j\right) \\
				& \ZZ \oplus \ZZ \frac{1 + i}{2} \oplus \ZZ \frac{4i + j}{11} & \sqrt{\frac{n^2}{n^2 - 176n + 1936}}\left(22,22, 11 + 5i + \frac{n - 44}{n}j\right) \\
				& \ZZ \oplus \ZZ \frac{1 + i}{2} \oplus \ZZ \frac{2i + j}{11} & \sqrt{\frac{n^2}{(n - 22)(n - 198)}}\left(22,22, 11 + 5i + \frac{n - 66}{n}j\right) \\
				& \ZZ \oplus \ZZ \frac{1 + i}{2} \oplus \ZZ \frac{5i + j}{11} & \sqrt{\frac{n}{n - 88}}\left(22,22, 11 + 5i + j\right) \\
				& \ZZ \oplus \ZZ \frac{1 + i}{2} \oplus \ZZ \frac{i + j}{11} & \sqrt{\frac{n^2}{n^2 - 176n + 1936}}\left(22,22, 11 + 5i + \frac{n - 44}{n}j\right)
	\end{array}
	\end{align*}
	\caption{Spheres orthogonal to unit spheres centered at points of $\OO \cap H^+$.}
	\label{Ghost Spheres Table}
	\end{table}

\begin{proof}
Each of the orders with Euclidean intersection covers $\RR^3$ by unit balls if and only if there does not exist a sphere orthogonal to unit spheres centered at the lattice points of $\OO \cap H^+$. These spheres are compiled in Table \ref{Ghost Spheres Table}, with one representative for every pair $\OO,j\OO j^{-1}$. Note that each of these spheres exists if and only if the expression in the square root is positive---the values of $(m,n)$ given in the statement of the theorem are precisely those for which the inequality holds true.
\end{proof}

\section{The $SL^\ddagger(2,\OO) \neq E^\ddagger(2,\OO)$ Case:}

We need to determine when $SL^\ddagger(2,\OO) \neq E^\ddagger(2,\OO)$. We begin by giving a sufficient condition that applies for all maximal $\ddagger$-orders.

\begin{theorem}\label{GeneralNonEqualCondition}
Let $H$ be a definite rational quaternion algebra with involution $\ddagger$ and $\OO$ a maximal $\ddagger$-order. Let $\OO'$ be the $\ddagger$-order generated by the elements of $\OO \cap H^+$ and $\OO^\times$. If $\OO \neq \OO'$, then $E^\ddagger(2,\OO) \neq SL^\ddagger(2,\OO)$.
\end{theorem}

\begin{proof}
Choose any $u \in \OO \backslash \OO'$. Let $N = \nrm(u)$. Then
	\begin{align*}
	\begin{pmatrix} u & 0 \\ 0 & \overline{u}^\ddagger \end{pmatrix} \in \begin{cases} SL^\ddagger\left(2,\OO\right) & \text{if } N = 1 \\ SL^\ddagger\left(2,\OO \otimes_\ZZ \ZZ/(N - 1)\ZZ\right) & \text{otherwise.} \end{cases}
	\end{align*}
	
\noindent In either case, however, by appealing to the fact that
	\begin{align*}
	SL^\ddagger\left(2,\OO\right) \rightarrow SL^\ddagger\left(2,\OO \otimes_\ZZ \ZZ/(N - 1)\ZZ\right)
	\end{align*}
	
\noindent is surjective, we conclude that there exists an element
	\begin{align*}
	\gamma = \begin{pmatrix} a & b \\ c & d \end{pmatrix} \in SL^\ddagger(2,\OO)
	\end{align*}
	
\noindent with $a \in \OO \backslash \OO'$. However, clearly $E^\ddagger(2,\OO) \subset SL^\ddagger(2,\OO')$, so $\gamma \in SL^\ddagger(2,\OO) \backslash E^\ddagger(2,\OO)$.
\end{proof}

	\begin{table}
	\begin{align*}
	\begin{array}{l|ll}
	m & \OO & \text{Conditions} \\ \hline
	m = 2 & \ZZ \oplus \ZZ i \oplus \ZZ \frac{i + j}{2} \oplus \ZZ \frac{2 + 2j + ij}{4} & \text{if } n/2 \equiv 1 \mod 8, \ n > 2 \\
			& \ZZ \oplus \ZZ i \oplus \ZZ \frac{i + j}{2} \oplus \ZZ \frac{2 + ij}{4} & \text{if } n/2 \equiv 3 \mod 8, \ n > 6
		\\ \hline
	m = 7	& \ZZ \oplus \ZZ \frac{1 + i}{2} \oplus \ZZ j \oplus \ZZ \frac{7j + ij}{14} & \text{if } 7|n, \ n > 7
		\\ \hline
	m \equiv 3 \mod 8 & \ZZ \oplus \ZZ \frac{1 + i}{2} \oplus \ZZ j \oplus \ZZ \frac{mj + ij}{2m} & \text{if } m|n \text{ and } \left(\frac{n/m}{m}\right) = 1, \ n > m
	\end{array}
	\end{align*}
	\caption{Orders $\OO$ that are not generated by $\OO^\times$ and $\OO\cap H^+$.}
	\label{NonEqualOrderTable}
	\end{table}

\begin{corollary}\label{NonEqualOrders}
For all of the orders in Table \ref{NonEqualOrderTable}, $SL^\ddagger(2,\OO) \neq E^\ddagger(2,\OO)$.
\end{corollary}
	
\begin{proof}
It is easy to check by computing the norm forms that for all but finitely many $n$, $\OO^\times \cap H^+ = \OO^\times$. This reduces showing that the orders in the table satisfy this property to a finite computation. However, if $\OO^\times \cap H^+ = \OO^\times$, then by Theorem \ref{GeneralNonEqualCondition} it suffices to note that none of the above orders are generated by elements in $\OO \cap H^+$.
\end{proof}

Apart from this criterion that allows us to establish that $\mathcal{S}_{\OO,j}$ is tangency-disconnected, we also have a means of proving that it is topologically disconnected, using the concept of a "ghost sphere."

\begin{definition}
Let $\OO$ be a maximal $\ddagger$-order of a definite quaternion algebra over $\QQ$. A \emph{ghost sphere} is a sphere in $\RR^3$ that does not intersect any sphere in $\mathcal{S}_{\OO,j}$.
\end{definition}

The term ``ghost sphere" comes from Stange's terminology for circle packings \cite{Stange2017}. Since the ghost sphere does not belong to the packing, it defines a boundary splitting the packing into two disconnected components. The same is true of all of the translates of the ghost sphere, and so the existence of a ghost sphere automatically proves that $\mathcal{S}_{\OO,j}$ splits into infinitely many disconnected components.

\begin{lemma}\label{GhostSphereLemma}
Let $\OO$ be a maximal $\ddagger$-order of a rational quaternion algebra $H = \left(\frac{-m,-n}{\QQ}\right)$, such that $\OO$ has Euclidean intersection. Suppose that $\OO$ is not one of the orders listed in Table \ref{NonEqualOrderTable}, nor is
	\begin{align*}
	\OO = \ZZ \oplus \ZZ \frac{1 + i}{2} \oplus \ZZ \frac{ti + j}{11} \oplus \ZZ \frac{11j + ij}{22} \subset \left(\frac{-11,-n}{\QQ}\right).
	\end{align*}
	
\noindent If $\OO$ is does not cover $\RR^3$ by unit balls, then $\mathcal{S}_{\OO,j}$ contains a ghost sphere.
\end{lemma}

\begin{proof}
If $\OO$ does not cover $\RR^3$ by unit balls, then there exists a sphere that is orthogonal to unit spheres centered on points in $\OO \cap H^+$, as defined in Table \ref{Ghost Spheres Table}. Under the given conditions, we shall prove that this sphere is a ghost sphere. Note that, in general, this sphere is defined by the fact that if
	\begin{align*}
	\OO = \ZZ \oplus \ZZ e_1 \oplus \ZZ e_2 \oplus \ZZ e_3,
	\end{align*}
	
\noindent and the inversive coordinates of this sphere are given by $\hat{x} = (\hat{\kappa}, \hat{\kappa}', \hat{\xi})$ then
	\begin{align*}
	b_{H,j}\left(\hat{x},(1,\alpha,e_1)\right) = b_{H,j}\left(\hat{x},(1,\beta,e_2)\right),
	\end{align*}
	
\noindent where $\alpha,\beta$ are integers such that
	\begin{align*}
	q_{H,j}\left((1,\alpha,e_1)\right) = q_{H,j}\left((1,\beta,e_2)\right) = 1.
	\end{align*}
	
\noindent From this, we conclude that
	\begin{align*}
	b_{H,j}\left(\hat{x},(\kappa,\kappa',a + b e_1 + c e_2)\right) &= b_{H,j}\left(\hat{x},(\kappa + \kappa' - a - b\alpha - c\beta)\right) \\
	&= -\frac{\hat{\kappa}}{2}\left(\kappa + \kappa' - a - b(\alpha + 1) - c(\beta + 1)\right).
	\end{align*}
	
\noindent Our goal is to use this to prove that for all $\gamma \in SL^\ddagger(2,\OO)$,
	\begin{align*}
	b_{H,j}\left(\hat{x},\inv(\gamma)\right) \notin [-n,n],
	\end{align*}
	
\noindent which is exactly equivalent to the sphere corresponding to the inversive coordinates $\hat{x}$ being a ghost sphere. From the above, we have shown that
	\begin{align*}
	\gamma \mapsto \frac{2}{\hat{\kappa}}b_{H,j}\left(\hat{x},\inv(\gamma)\right)
	\end{align*}
	
\noindent is a polynomial map---indeed, this map is linear, with integer coefficients. Consequently, we can completely determine its image by considering the localizations at various primes $p$. First, note that if $p$ is odd and it is not the case that $p|\nrm(j)$, $p\nmid \disc(H)$, and $-\disc(\ddagger) = \left(\QQ_p^\times\right)^2$, then we know there is a unique maximal $\ddagger$-order in $H_p$ containing $j$, and that is
	\begin{align*}
	\OO_p = \ZZ_p \oplus \ZZ_p i \oplus \ZZ_p j \oplus \ZZ_p \frac{ij}{p^k},
	\end{align*}
	
\noindent where $k$ is either $0$ or $1$, with the latter occurring precisely when $p|\nrm(i),\nrm(j)$. In either case, one checks that if $p|\nrm(j)$, then
	\begin{align*}
	\tilde{\kappa}_j'(\gamma) &= 2\re\left(a\overline{b}j\right) \in p\ZZ_p \\
	\tilde{\kappa}_j(\gamma) &= 2\re\left(c\overline{d}j\right) \in p\ZZ_p
	\end{align*}
	
\noindent and
	\begin{align*}
	\tilde{\xi}_j(\gamma) &= j + aj\overline{d} - bj\overline{c} - (ad^\ddagger - bc^\ddagger)j \\
	&= j + a\left((j\overline{d}j^{-1}\right)j - b\left((j\overline{c}j^{-1}\right)j \\
	&\in p\ZZ_p \oplus p \ZZ_p i \oplus \left(1 + p\ZZ_p\right)j.
	\end{align*}
	
\noindent We also know that $\tilde{\xi}_j(\gamma) = a + be_1 + ce_2$, and from this we compute that
	\begin{align*}
	c &\in \frac{1}{\pi_j(e_2)} + p \ZZ_p \\
	b &\in -\frac{\pi_i(e_2)}{\pi_i(e_1)\pi_j(e_2)} + p\ZZ_p \\
	a &\in \frac{\pi_i(e_1)\re(e_2) - \pi_i(e_2)\re(e_1)}{\pi_i(e_1)\pi_j(e_2)} + p\ZZ_p.
	\end{align*}
	
\noindent Taking all of this together, we have that
	\begin{align*}
	\frac{2}{\hat{\kappa}}b_{H,j}\left(\hat{x},\inv(\gamma)\right) &\in \frac{\pi_i(e_1)\re(e_2) - \pi_i(e_2)\re(e_1)}{\pi_i(e_1)\pi_j(e_2)} - \frac{\pi_i(e_2)(\alpha + 1)}{\pi_i(e_1)\pi_j(e_2)} + \frac{\beta + 1}{\pi_j(e_2)} + p\ZZ_p \\
	&\in \frac{\pi_i(e_1)(\re(e_2) + \beta + 1) - \pi_i(e_2)(\re(e_1) + \alpha + 1)}{\pi_i(e_1)\pi_j(e_2)} + p\ZZ_p.
	\end{align*}
	
\noindent For all other odd primes $p$, it shall suffice to note that
	\begin{align*}
	\frac{2}{\hat{\kappa}}b_{H,j}\left(\hat{x},\inv(\gamma)\right) \in \ZZ_p.
	\end{align*}
	
\noindent This leaves only $p = 2$, which we handle separately for all the various orders under consideration. Let us consider the case $m = 3$, $3|n$, $n \equiv -1 \mod 3$, where
	\begin{align*}
	\OO = \ZZ \oplus \ZZ \frac{1 + i}{2} \oplus \ZZ \frac{i + j}{3} \oplus \ZZ \frac{3j + ij}{3}.
	\end{align*}
	
\noindent First, we note that from the above discussion, we have
	\begin{align*}
	\frac{2}{\hat{\kappa}}b_{H,j}\left(\hat{x},\inv(\gamma)\right) = \frac{n - 6}{3} + \frac{n_2}{3}\ZZ,
	\end{align*}
	
\noindent where $n = 2^l n_2$ for some integers $l,n_2$, where $n_2$ is odd. One checks by computing $SL^\ddagger(2,\OO\otimes_{\ZZ} \ZZ/2^l\ZZ)$ that for all $\gamma \in SL^\ddagger(2,\OO)$,
	\begin{align*}
	\inv(\gamma) \in \left(2^{l + 1}\ZZ_{2^{l + 1}}, 2^{l + 1}\ZZ_{2^{l + 1}}, 2^{l + 1}\ZZ_{2^{l + 1}} \oplus 2^{l + 1}\ZZ_{2^{l + 1}} i \oplus (1 + 2^{l + 1}\ZZ_{2^{l + 1}}) j\right).
	\end{align*}
	
\noindent From this, we deduce that if $\tilde{\xi}_j(\gamma) = a + be_1 + ce_2$, then
	\begin{align*}
	c &= 1 \mod 2^{l + 1} \\
	b &= -2 \mod 2^{l + 1} \\
	a &= 1 \mod 2^{l + 1},
	\end{align*}
	
\noindent and therefore
	\begin{align*}
	\frac{2}{\hat{\kappa}}b_{H,j}\left(\hat{x},\inv(\gamma)\right) &= -(\kappa + \kappa') + a + b + \frac{n + 3}{9}c \\
	&= n + 2 \mod 2^{l + 1}.
	\end{align*}
	
\noindent However,
	\begin{align*}
	\frac{n - 6}{3} &= -n - 2 \mod 2^{l + 1} \\
	&= n + 2 \mod 2^{l + 1},
	\end{align*}
	
\noindent from which we conclude that
	\begin{align*}
	\frac{2}{\hat{\kappa}}b_{H,j}\left(\hat{x},\inv(\gamma)\right) = \frac{n - 6}{3} + \frac{2n}{3}\ZZ.
	\end{align*}
	
\noindent This implies that
	\begin{align*}
	b_{H,j}\left(\hat{x},\inv(\gamma)\right) &= \sqrt{\frac{n}{n - 24}}\left(\frac{n - 6}{3} + 2n\mathbb{Z}\right),
	\end{align*}
	
\noindent which has non-empty intersection with $[-n,n]$ only if
	\begin{align*}
	\sqrt{\frac{n}{n - 24}}\frac{n - 6}{3n} < 1.
	\end{align*}
	
\noindent One checks this never occurs, ergo the sphere with coordinates $\hat{x}$ is a ghost sphere. The arguments for the other cases are similar---we compile the results of the computations of the image of $b_{H,j}\left(\hat{x},\inv(\gamma)\right)$ in Table \ref{GhostSphereNonIntersections}, with one representative for each pair of orders $\OO, j\OO j^{-1}$ under consideration.
\end{proof}

	\begin{table}
	\begin{align*}
	\begin{array}{l|ll}
	m & \OO \cap H^+ & \text{Image} \\ \hline
	m = 1 &	\ZZ \oplus \ZZ i \oplus \ZZ j & \sqrt{\frac{n}{n - 2}}\left(n + 2n\ZZ\right) \\
				& \ZZ \oplus \ZZ i \oplus \ZZ \frac{1 + i + j}{2} & \sqrt{\frac{n^2}{n^2 - 12n + 4}}\left(2 - n + 2n\ZZ\right) \\
				& \ZZ \oplus \ZZ i \oplus \ZZ \frac{1 + j}{2} & \sqrt{\frac{n^2}{n^2 - 10n + 1}}\left(3 - n + 2n\ZZ\right) \\
				& \ZZ \oplus \ZZ i \oplus \ZZ \frac{i + j}{2} & \sqrt{\frac{n^2}{n^2 - 10n + 1}}\left(n - 1 + 2n\ZZ\right) \\ \hline
	m = 2 & \ZZ \oplus \ZZ i \oplus \ZZ \frac{1 + i + j}{2} & \sqrt{\frac{n^2}{(n - 1)(n - 9)}}\left(1 - n + 2n\ZZ\right) \\
				& \ZZ \oplus \ZZ i \oplus \ZZ \frac{1 + j}{2} & \sqrt{\frac{n^2}{n^2 - 6n + 1}}\left(3 - n + 2n\ZZ\right) \\
				& \ZZ \oplus \ZZ i \oplus \ZZ \frac{2 + i + j}{4} & \sqrt{\frac{n^2}{n^2 - 36n + 100}}\left(6 - n + 2n\ZZ\right) \\
				& \ZZ \oplus \ZZ i \oplus \ZZ \frac{i + j}{4} & \sqrt{\frac{n^2}{n^2 - 28n + 36}}\left(n - 6 + 2n\ZZ\right) \\
	m = 3 & \ZZ \oplus \ZZ \frac{1 + i}{2} \oplus \ZZ j & \sqrt{\frac{3n}{3n - 8}}\left(n + 2n\ZZ\right) \\
				& \ZZ \oplus \ZZ \frac{1 + i}{2} \oplus \ZZ \frac{i + j}{3} & \sqrt{\frac{n}{n - 24}}\left(\frac{n - 6}{3} + 2n\mathbb{Z}\right) \\ \hline
	m = 7 & \ZZ \oplus \ZZ \frac{1 + i}{2} \oplus \ZZ j & \sqrt{\frac{7n}{7n - 12}}\left(n + 2n\ZZ\right) \\ \hline
	m = 11 & \ZZ \oplus \ZZ \frac{1 + i}{2} \oplus \ZZ j & \sqrt{\frac{11n}{11n - 8}}\left(n + 2n\ZZ\right)
	\end{array}
	\end{align*}
	\caption{Images of $b_{H,j}\left(\hat{x},\inv(\gamma)\right)$ for $\gamma \in SL^\ddagger(2,\OO)$.}
	\label{GhostSphereNonIntersections}
	\end{table}

An obvious question to ask is why the proof of Lemma \ref{GhostSphereLemma} cannot be extended to other maximal $\ddagger$-orders. The surprising answer is that even for maximal $\ddagger$-orders with Euclidean intersection, the sphere orthogonal to the unit spheres centered at points in $\OO \cap H^+$ is not necessarily a ghost sphere. For instance, for the quaternion algebra $H = \left(\frac{-7,-7}{\QQ}\right)$ with order
	\begin{align*}
	\OO = \ZZ \oplus \ZZ \frac{1 + i}{2} \oplus \ZZ j \oplus \ZZ \frac{7j + ij}{14}
	\end{align*}
	
\noindent the sphere with center $\left(\frac{1}{2},\frac{3}{2\sqrt{7}},\frac{\sqrt{7}}{2}\right)$ and radius $\frac{\sqrt{37/7}}{2}$ is such a sphere, but it intersects $\gamma(\hat{S}_j)$, where
	\begin{align*}
	\gamma = \begin{pmatrix}
 1 + \frac{j}{2} + \frac{ij}{14} & 1-\frac{j}{2}+\frac{ij}{14}\\
 1-\frac{j}{2}+\frac{ij}{14} & -1-\frac{j}{2}-\frac{ij}{14}
\end{pmatrix} \in SL^\ddagger(2,\OO).
	\end{align*}
	
\noindent However, between Corollary \ref{NonEqualOrders} and Lemma \ref{GhostSphereLemma}, we have proved that if $\OO$ does not cover $\RR^3$ by unit balls, then $SL^\ddagger(2,\OO) \neq E^\ddagger(2,\OO)$ for all maximal $\ddagger$-orders with Euclidean intersection, with the sole exception of the case where
	\begin{align*}
	\OO = \ZZ \oplus \ZZ \frac{1 + i}{2} \oplus \ZZ \frac{ti + j}{11} \oplus \ZZ \frac{11j + ij}{22} \subset \left(\frac{-11,-n}{\QQ}\right),
	\end{align*}
	
\noindent where $11|n$ and $n/11$ is a non-square modulo $11$. We handle this case separately.

\section{Special Unimodular Pairs:}

Throughout this section, we shall take $\OO$ to be a maximal $\ddagger$-order
	\begin{align*}
	\OO = \ZZ \oplus \ZZ \frac{1 + i}{2} \oplus \ZZ \frac{ti + j}{11} \oplus \ZZ \frac{11j + ij}{22} \subset \left(\frac{-11, -n}{\QQ}\right),
	\end{align*}
	
\noindent where $n$ is a positive integer such that $11|n$ and $n/11$ is not a square modulo $11$, and $\OO$ does not cover $\RR^3$ by unit balls, unless otherwise specified. Our goal is to prove that $SL^\ddagger(2,\OO) \neq E^\ddagger(2,\OO)$ in this case, which we do by adapting Nica's proof that $SL(2,\OO) \neq E(2,\OO)$ if $\OO$ is an imaginary quadratic ring \cite{Nica}.

\begin{definition}
We call a pair $(a,b) \in \OO^2$ \emph{unimodular} if there exists a pair $(c,d) \in \OO^2$ such that
	\begin{align*}
	\begin{pmatrix} a & b \\ c & d \end{pmatrix} \in SL^\ddagger(2,\OO).
	\end{align*}
	
\noindent We denote the entire space of unimodular pairs by $\mathcal{U}(2,\OO)$. We call a unimodular pair $(a,b)$ \emph{special} if
	\begin{align*}
	\nrm(a) = \nrm(b) < \nrm(a \pm b)
	\end{align*}
	
\noindent and $a,b \in \OO \cap \QQ(j)$.
\end{definition}

\noindent The definitions here are essentially the same as Nica's, mutatis mutandis. Notice that $SL^\ddagger(2,\OO)$ has a right action on $\mathcal{U}(2,\OO)$, and since
	\begin{align*}
	\left\{\begin{pmatrix} 1 & x \\ 0 & 1 \end{pmatrix} \middle| x \in \OO \cap H^+\right\}/SL^\ddagger(2,\OO) = \mathcal{U}(2,\OO),
	\end{align*}
	
\noindent if $\mathcal{U}(2,\OO)$ splits into infinitely many orbits under the action of $E^\ddagger(2,\OO)$, then $E^\ddagger(2,\OO)$ is an infinite index subgroup of $SL^\ddagger(2,\OO)$. We shall prove that this is the case by showing that there exist special pairs that belong to different orbits under the action of $E^\ddagger(2,\OO)$. To this end, we begin with a pair of lemmas.

\begin{lemma}\label{Modification of Cohn}
Every element $\gamma \in E^\ddagger(2,\OO)$ can be written in the form
	\begin{align*}
	\gamma = \pm \begin{pmatrix} x_1 & 1 \\ -1 & 0 \end{pmatrix}\begin{pmatrix} x_2 & 1 \\ -1 & 0 \end{pmatrix}\ldots \begin{pmatrix} x_n & 1 \\ -1 & 0 \end{pmatrix},
	\end{align*}
	
\noindent where $x_1, x_2, \ldots x_n \in \OO \cap H^+$.
\end{lemma}

\begin{lemma}\label{Small Norms are Rare}
For all $x \in \OO$, $\nrm(x) > 2$ if $x \neq 0,1$, and $\nrm(x) = 3$ if and only if $x = \pm (1 \pm i)/2$.
\end{lemma}

Lemma \ref{Modification of Cohn} was already shown to be true with $\OO$ replaced by an imaginary quadratic ring, and $E^\ddagger(2,\OO)$ replaced by $E(2,\OO)$ by Cohn \cite{Cohn1966}---the proof in our case is essentially the same, and so we omit it in the interest of brevity. The proof of Lemma \ref{Small Norms are Rare} follows via a straightforward computation of the norm form, which we leave as an exercise to the reader. As a consequence of the latter lemma, we have a restriction on how the norms of pairs in $\mathcal{U}(2,\OO)$ can change as a result of the action by $E^\ddagger(2,\OO)$.

\begin{lemma}\label{Non-shrinking norms}
Let $(a,b) \in \mathcal{U}(2,\OO)$, and consider
	\begin{align*}
	(a',b') = (a,b)\begin{pmatrix} x & 1 \\ -1 & 0 \end{pmatrix},
	\end{align*}
	
\noindent where $x \in \OO \cap H^+$. If $x \neq 0,\pm 1$ and $\nrm(a) > \nrm(b)$, then $\nrm(a') > \nrm(b')$. Additionally, if $x \neq 0$ and $(a,b)$ is special, then $\nrm(a') > \nrm(b')$.
\end{lemma}

\begin{proof}
Note that $(a',b') = (ax - b, a)$, hence
	\begin{align*}
	\nrm(a') = \nrm(ax - b) &\geq \nrm(x)\nrm(a) - \nrm(b).
	\end{align*}
	
\noindent If $x \neq 0,\pm 1$, then by Lemma \ref{Small Norms are Rare}, if $\nrm(a) > \nrm(b)$, then
	\begin{align*}
	\nrm(a') > 2\nrm(a) - \nrm(a) = \nrm(a) = \nrm(b').
	\end{align*}
	
\noindent On the other hand, if $(a,b)$ is a special pair, if $x \neq 0, \pm 1, \pm (1 \pm i)/2$, then
	\begin{align*}
	\nrm(a') \geq 3\nrm(a) - \nrm(a) = 2\nrm(a) > \nrm(b').
	\end{align*}
	
\noindent If $x = \pm 1$, then we note that
	\begin{align*}
	\nrm(a') = \nrm(a \pm b) > \nrm(a) = \nrm(b'),
	\end{align*}
	
\noindent by definition of the special pair. This leaves the case when $x = \pm (1 \pm i)/2$. Suppose that $a = a_0 + a_1 j$, $b = b_0 + b_1 j$, where $a_0,a_1,b_0,b_1 \in \QQ$. Then one directly computes that if $x = \pm (1 \pm i)/2$,
	\begin{align*}
	\nrm(a') - \nrm(b') &= \nrm(ax - b) - \nrm(a) \\
	&= \frac{7 a_0^2}{4} + \left(\frac{a_0}{2} \pm b_0\right)^2 + \frac{77 a_1^2 n}{4} + 11 n \left(\frac{a_1}{2} \pm b_1\right)^2 > 0,
	\end{align*}
	
\noindent and so we conclude that $\nrm(a') > \nrm(b')$ in this case as well.
\end{proof}

As a corollary, we obtain our desired result.

\begin{theorem}\label{Eleven Is The Worst}
$E^\ddagger(2,\OO)$ is an infinite index subgroup of $SL^\ddagger(2,\OO)$.
\end{theorem}

\begin{proof}
By Lemmas \ref{Modification of Cohn} and \ref{Non-shrinking norms}, if $(a,b), (c,d)$ are both special and in the same orbit, then it must be that
	\begin{align*}
	(c,d) = \pm (a,b) \begin{pmatrix} 0 & 1 \\ -1 & 0 \end{pmatrix}^k
	\end{align*}
	
\noindent where $k = 0,1$---or, in other words, $(c,d) = (a,b), (-a,-b), (b,-a), (-b,a)$. It suffices to show that there are infinitely many special pairs to conclude that infinitely many special pairs belong to different orbits under the action of $E^\ddagger(2,\OO)$, proving that $E^\ddagger(2,\OO)$ is an infinite index subgroup of $SL^\ddagger(2,\OO)$. Since this is equivalent to finding special pairs in the imaginary quadratic ring $\OO \cap \QQ(j)$, we know that this has already been proved by Nica \cite{Nica}.
\end{proof}

This result finally allows us to characterize all of the orders for which $SL^\ddagger(2,\OO) = E^\ddagger(2,\OO)$.

\begin{theorem}\label{Slightness in R3}
Let $H$ be a rational quaternion algebra generated by $i,j$. Let $\OO$ be a maximal $\ddagger$-order of $H$ with Euclidean intersection. If $\OO$ covers $\RR^3$ by unit balls, then $SL^\ddagger(2,\OO) = E^\ddagger(2,\OO)$. Otherwise, $E^\ddagger(2,\OO)$ is an infinite index subgroup of $SL^\ddagger(2,\OO)$. The order $\OO$ covers $\RR^3$ by unit balls if and only if $m = \nrm(i), n = \nrm(j)$ are on the following list.
	\begin{align*}
	\begin{array}{l|l}
	m & n \\ \hline
	1 & 1,2,3,6,7,10 \\
	2 & 1,2,3,5,6,10,14,26 \\
	3 & 1,2,6,15 \\
	7 & 1 \\
	11 & 22,66,77,110,143.
	\end{array}
	\end{align*}
\end{theorem}

Theorem \ref{Slightness in R3} should be seen an analog of the result of Cohn \cite{Cohn1966} that if $\OO$ is the ring of integers of an imaginary quadratic field, then either $\OO$ is a Euclidean domain and $E(2,\OO) = SL(2,\OO)$ or otherwise $E(2,\OO)$ is an infinite index subgroup of $SL(2,\OO)$. As a side note, this is a surprising result, since it is true that $SL(2,\OO) = E(2,\OO)$ if $\OO$ is the ring of integers of any algebraic number field that is not imaginary quadratic; the imaginary quadratic case seems to be special, and was given an elementary proof by Nica \cite{Nica}, with a small corrigendum made in \cite{Sheydvasser2016}.

\section{Classification of Superpackings:}

We are finally ready to give a proof of Theorem \ref{SuperpackingClassification}.

\begin{proof}[Proof of Theorem \ref{SuperpackingClassification}]
If $H$ is generated by $i,j$, we shall write $m = \nrm(i)$, $n = \nrm(j)$. By Theorems \ref{WhenRationalAndTangential}, \ref{IntersectionGraph}, and \ref{Slightness in R3}, we know that $\mathcal{S}_{\OO,j}$ has only rational intersections and is tangency-connected if and only $(m,n)$ are on the following list.

\begin{minipage}{.45\textwidth}
	\begin{align*}
	\begin{array}{l|l}
	m & n \\ \hline
	1 & 6,7,10 \\
	2 & 5,26
	\end{array}
	\end{align*}
\end{minipage}%
\begin{minipage}{.45\textwidth}
	\begin{align*}
	\begin{array}{l|l}
	3 & 1,2,15 \\
	7 & 1 \\
	11 & 143.
	\end{array}
	\end{align*}
\end{minipage}

\noindent By Lemma \ref{InversiveRational}, $\mathcal{S}_{\OO,j}$ is integral for any choice of $H$ and $\OO$. Since we know that $\OO$ covers $\RR^3$ by unit balls, we know that the extended Euclidean algorithm described in Lemma \ref{EuclideanAlgorithm} applies---consequently, at any point in $H^+$ that can be written as $b^{-1}a$ with $a,b \in \OO$, we can find a sphere in $\mathcal{S}_{\OO,j}$ passing through that point. Such points are dense in $\RR^3$, hence $\mathcal{S}_{\OO,j}$ is dense in $\RR^3$. Finally, using the explicit forms of $\OO$ compiled in Theorem \ref{All Orders with Euclidean Intersection}, we conclude that the only isomorphism classes satisfying the desired conditions are the ones listed. It remains to determine which of the sets $\mathcal{S}_{\OO,j}$ are conformally equivalent

We note that the two collections with $(m,n) = (1,7)$ are conformally equivalent, since if
	\begin{align*}
    \OO_1 &= \ZZ \oplus \ZZ i \oplus \ZZ \frac{1 + j}{2} \oplus \ZZ \frac{i + ij}{2} \\
    \OO_2 &= \ZZ \oplus \ZZ i \oplus \ZZ \frac{i + j}{2} \oplus \ZZ \frac{1 + ij}{2}
    \end{align*}
    
\noindent then
	\begin{align*}
    \begin{pmatrix} \frac{1 + i}{\sqrt{2}} & 0 \\ 0 & -\frac{1 - i}{\sqrt{2}} \end{pmatrix}SL^\ddagger(2,\OO_1)\begin{pmatrix} \frac{1 + i}{\sqrt{2}} & 0 \\ 0 & -\frac{1 - i}{\sqrt{2}} \end{pmatrix}^{-1} = SL^\ddagger(2,\OO_2),
    \end{align*}
    
\noindent and therefore
	\begin{align*}
    \begin{pmatrix} \frac{1 + i}{\sqrt{2}} & 0 \\ 0 & -\frac{1 - i}{\sqrt{2}} \end{pmatrix}SL^\ddagger(2,\OO_1)(\hat{S}_j) = SL^\ddagger(2,\OO_2)(\hat{S}_j).
    \end{align*}
		
\noindent Similarly, if $\OO' = j\OO j^{-1}$, then $\mathcal{S}_{\OO,j}$ and $\mathcal{S}_{\OO',j}$ are conformally equivalent. On the other hand, note that
	\begin{align*}
	\left\{b_{H,j}(\text{inv}(S_1),\text{inv}(S_2))\middle|S_1, S_2 \in \mathcal{S}_{\OO,j}\right\}  &= \left\{b_{H,j}\left(\text{inv}(-\hat{S_j}),\text{inv}(S_2)\right)\middle|S_1, S_2 \in \mathcal{S}_{\OO,j}\right\} \\
    &= \xi_3\left(SL^\ddagger(2,\OO)\right).
	\end{align*}
    
\noindent This set is clearly invariant under conformal transformations; furthermore, we already computed it in Theorem \ref{All Orders with Euclidean Intersection}, as summarized below.
	
	\begin{minipage}{.45\textwidth}
	\begin{align*}
    \begin{array}{l|l}
    H & \xi_3\left(SL^\ddagger(2,\OO)\right) \\ \hline
		\left(\frac{-1,-6}{\QQ}\right) & 1 + 3\ZZ \\
    \left(\frac{-1,-7}{\QQ}\right) & 1 + \frac{7}{2}\ZZ \\
    \left(\frac{-1,-10}{\QQ}\right) & 1 + 5\ZZ \\
    \left(\frac{-2,-5}{\QQ}\right) & 1 + \frac{5}{2}\ZZ \\
		\left(\frac{-2,-26}{\QQ}\right) & 1 + \frac{13}{4}\ZZ
		\end{array}
	\end{align*}
	\end{minipage} %
	\begin{minipage}{.45\textwidth}
		\begin{align*}
		\begin{array}{l|l}
    \left(\frac{-3,-1}{\QQ}\right) & 1 + 2\ZZ \\
    \left(\frac{-3,-2}{\QQ}\right) & 1 + 4\ZZ \\
		\left(\frac{-3,-15}{\QQ}\right) & 1 + \frac{10}{3}\ZZ \\
    \left(\frac{-7,-1}{\QQ}\right) & 1 + 2\ZZ \\
		\left(\frac{-11,-143}{\QQ}\right) & 1 + \frac{26}{11}\ZZ.
    \end{array}
    \end{align*}
	\end{minipage}
    
\noindent This proves that the only two additional sphere collections that can possibly be conformally equivalent are the collections with $(m,n) = (3,1)$ and $(m,n) = (7,1)$. However, it shall be proved in Lemma \ref{BMGraph31} that the sphere collection with $(m,n) = (3,1)$ is the superpacking of the Soddy sphere packing, which is a Boyd-Maxwell packing. In contrast, we shall prove in Theorem \ref{NotBMGraph} that the collection with $(m,n) = (7,1)$ is the superpacking of an integral crystallographic packing that is not Boyd-Maxwell.
\end{proof}

\section{$(\OO,j)$-Apollonian Packings:}\label{ApollonianPackings}

It remains to show that the sphere collections $\hat{\mathcal{S}}_{\OO,j}$ enumerated in Theorem \ref{SuperpackingClassification} are superpackings of corresponding integral crystallographic packings. Our approach is to introduce the notion of $(\OO,j)$-Apollonian packings; this terminology is taken from Stange \cite{StangeFuture}, where she uses a similar construction for the circle packings introduced in \cite{Stange2017}. We begin with a definition.

	\begin{definition}
	A collection of spheres $C$ \emph{straddles} a sphere $S$ if it intersects both the interior and exterior of $S$. We define an $(\OO,j)$-\emph{Apollonian packing} to be a maximal tangency-connected subset $C$ of spheres in $\hat{\mathcal{S}}_{\OO,j}$ such that $C$ does not straddle any spheres in $\hat{\mathcal{S}}_{\OO,j}$. We say two spheres $S_1, S_2 \in \hat{\mathcal{S}}_{\OO,j}$ are \emph{immediately tangent} if they don't straddle any spheres in $\hat{\mathcal{S}}_{\OO,j}$.
	\end{definition}
	
\noindent It is evident that $\hat{\mathcal{S}}_{\OO,j}$ is the disjoint union of all $(\OO,j)$-Apollonian packings, since every sphere in $\hat{\mathcal{S}}_{\OO,j}$ lies in some tangency-connected component, and if two such components intersect, then they must be equal. Furthermore, any two $(\OO,j)$-Apollonian packings are necessarily the same up to a M\"{o}bius transformation.
	
We shall write $\mathcal{P}_{\OO,j}$ to denote the unique $(\OO,j)$-Apollonian packing containing $\hat{S}_{j}$. Since $\mathcal{P}_{\OO,j}$ is a sub-packing of $\hat{\mathcal{S}}_{\OO,j}$, it is integral and spheres can only intersect tangentially---indeed, due to its construction, they can only intersect externally, which is to say that the interiors do not intersect. Furthermore, since $\hat{\mathcal{S}}_{\OO,j}$ is tangency-connected and dense in $\RR^3$, $\mathcal{P}_{\OO,j}$ must be dense in $\RR^3$. We wish to show that $\mathcal{P}_{\OO,j}$ is in fact a crystallographic packing. By the above discussion, it shall suffice to find a discrete subgroup of $\text{Isom}(\HH^4)$ admitting a convex fundamental polyhedron with finitely many sides, and whose limit set is the closure of $\mathcal{P}_{\OO,j}$. Let $\varpi, \tau \in \OO$ such that $\OO \cap H^+ = \ZZ \oplus \ZZ \varpi \oplus \ZZ \tau$, as in Table \ref{Orders with Euclidean intersection}. Define
	\begin{align*}
	\begin{split}
	\varphi_1(z) &= -\overline{z} \\
	\varphi_2(z) &= \frac{1}{\nrm(i)}i \overline{z} i \\
	\varphi_3(z) &= 2 - \overline{z} \\
	\varphi_4(z) &= \begin{cases} 2i + \varphi_2(z) & \text{if } \varpi = i \\ i + \varphi_2(z) & \text{if } \varpi = \frac{1 + i}{2} \end{cases} \\
	\varphi_5(z) &= \overline{z}^{-1}
		\end{split}
	\begin{split}
	\varphi_6(z) &= \varphi_5(z - 1) + 1 \\
	\varphi_7(z) &= \varphi_5(z - \varpi) + \varpi \\
	\varphi_8(z) &= \varphi_5(z - 1 - \varpi) + 1 + \varpi \\
	\varphi_9(z) &= \varphi_5(z - \tau) + \tau \\
	\varphi_{10}(z) &= \varphi_5(z - 1 - \tau) + 1 + \tau \\
	\varphi_{11}(z) &= \varphi_5(z - \varpi - \tau) + \varpi + \tau \\
	\varphi_{12}(z) &= \varphi_5(z - 1 - \varpi - \tau) + 1 + \varpi + \tau.
	\end{split}
	\end{align*}
	
\noindent The maps $\varphi_1, \varphi_2, \varphi_3,\varphi_4$ are reflections through the planes $x = 0$, $y = 0$, $x = 1$, $y = im(\varpi)$ respectively. The map $\varphi_5$ is a reflection through the unit sphere centered at $0$. The maps $\varphi_i$ for $i = 6, \ldots 12$ are reflections through translations of this sphere. Let $\Gamma$ be the group generated by the maps $\varphi_i$.
	
	\begin{lemma}\label{FundamentalPolyhedron}
	$\Gamma$ is a discrete subgroup of $\text{Isom}(\HH^4)$ admitting a convex fundamental polyhedron with finitely many sides.
	\end{lemma}
	
	\begin{proof}
	Note that $\Gamma$ can be embedded inside the smallest subgroup containing $\varphi_1, \varphi_2$ and $PSL^\ddagger(2,\OO)$---which, since these commute, is nothing more than
		\begin{align*}
		\begin{cases} PSL^\ddagger(2,\OO) \oplus \langle \varphi_1 \rangle & \text{if } m = 1 \\ PSL^\ddagger(2,\OO) \oplus \langle \varphi_1 \rangle \oplus \langle \varphi_2 \rangle & \text{otherwise} \end{cases}.
		\end{align*}
		
	\noindent Consequently, it is discrete. Consider the convex polyhedron $P$ in $\HH^4$ defined as the set of all points $(x,y,z,t)$ such that $0 < x < 1$, $0 < y < \im(\varpi)$ and $(x,y,z,t)$ is on the exterior of every unit sphere centered at $\ZZ[\varpi]$ and $\ZZ[\varpi] + \tau$---note that these are geodesic spheres in $\HH^4$. As there are only finitely many such spheres that intersect the cone defined above, this polyhedron has finitely many sides. By inspection, all of the generators of $\Gamma$ are hyperbolic reflections through the sides of $P$. Furthermore, for every side of $P$, there is a corresponding reflection through that side. We conclude that $P$ is a fundamental domain of $\Gamma$.
	\end{proof}

	\begin{lemma}\label{TransitiveAction}
	$\Gamma$ acts transitively on $\QQ(i)$ and $\QQ(i) + \tau$.
	\end{lemma}
	
	\begin{proof}
	Let $\Gamma'$ be the subgroup generated by $\varphi_i$ for $i = 1,\ldots 8$. Since it preserves $\QQ(i)$ considered with orientation, we can consider it as a subgroup of $\text{Isom}(\HH^3)$, and so it makes sense to consider its fundamental domain in $\HH^3$. One possible choice is the set of points $(x,y,z) \in \HH^3$ satisfying
	
		\begin{minipage}{.45\textwidth}
		\begin{align*}
		x^2 + y^2 + z^2 &> 1 \\
		(x - 1)^2 + y^2 + z^2 &> 1 \\
		x^2 + (y - \im(\varpi))^2 + z^2 &> 1
		\end{align*}
		\end{minipage}%
		\begin{minipage}{.45\textwidth}
		\begin{align*}
		(x - 1)^2 + (y - \im(\varpi))^2 + z^2 &> 1 \\
		1 > x &> 0 \\
		\im(\varpi) > y &> 0.
		\end{align*}
		\end{minipage}
	
\noindent Note that this region is the fundamental domain of $SL\left(2,\ZZ[\varpi]\right)$, unless $m = 1$ or $m = 3$, in which case it is the union of copies of this fundamental domain, which share a cusp. Furthermore, the action of $\Gamma'$ moves these copies of the fundamental domain of $SL\left(2,\ZZ[\varpi]\right)$ to adjacent copies of the fundamental domain. Consequently, $\Gamma'$ acts transitively on the cusps of the fundamental domain of $SL\left(2,\ZZ[\varpi]\right)$, which occur at every point of $\QQ(i)$. The proof that $\Gamma$ acts transitively on $\QQ(i) + \tau$ is similar---note that $\varphi_i$ for $i = 1, \ldots 4$, and $9, \ldots 12$ generate a group that preserves $\QQ(i) + \tau$ considered with orientation, so we can also consider it as a subgroup of $\text{Isom}(\HH^4)$. Its fundamental domain also contains copies of the fundamental domain of $SL\left(2,\ZZ[\varpi]\right)$, and so we are done.
	\end{proof}
	
	\begin{lemma}\label{LimitSet}
	The limit set of $\Gamma$ is the closure of $\mathcal{P}_{\OO,j}$.
	\end{lemma}
	
	\begin{proof}
	We shall fix the point $o = ij$ and consider the orbit $\Gamma o$. First, note that the limit set of $\Gamma$ must contain $\infty$, since
		\begin{align*}
		\left(\varphi_3 \circ \varphi_1\right)(z) = z + 2,
		\end{align*}
		
	\noindent and therefore
		\begin{align*}
		\left(\varphi_3 \circ \varphi_1\right)^n o \xrightarrow{n\longrightarrow \infty} \infty.
		\end{align*}
		
	\noindent By Lemma \ref{TransitiveAction}, this implies that the limit set of $\Gamma$ contains both $\hat{S}_{j}$ and $-\hat{S}_{j} + \tau$. By Lemma \ref{IntersectionsWithPlane}, we know that the intersection of any sphere of $\mathcal{S}_{\OO,j}$ with $\hat{S}_j$ happens at a rational point, and by Lemma \ref{TransitiveAction} we know that $\Gamma$ acts transitively on $\QQ(i)$. Since $-\hat{S}_j$ is immediately tangent to $\hat{S}_j$, we therefore conclude that the limit set of $\Gamma$ contains all spheres immediately tangent to $\hat{S}_j$. Similarly, since $\Gamma$ acts transitively on $\QQ(i) + \tau$, it follows that the limit set contains all spheres immediately tangent to $-\hat{S}_{j} + \tau$.
	
	The hyperbolic reflection $\varphi_4$ sends $-\hat{S}_j + \tau$ to the sphere
		\begin{align*}
		\begin{pmatrix} 0 & -1 \\ 1 & \tau \end{pmatrix}(\hat{S}_j),
		\end{align*}
		
	\noindent which is the sphere immediately tangent to $\hat{S}_j$ at $0$. Therefore, $\varphi_4$ sends spheres immediately tangent to $-\hat{S}_j + \tau$ to spheres immediately tangent to the sphere immediately tangent to $\hat{S}_j$ at $0$. Similarly, $\varphi_7$ sends $\hat{S}_j$ to the sphere immediately tangent to $-\hat{S}_j + \tau$ at $\tau$. Therefore, $\varphi_7$ sends spheres immediately tangent to $\hat{S}_j$ to spheres immediately tangent to the sphere immediately tangent to $-\hat{S}_j + \tau$ at $\tau$. Using the transitivity of the action on $\hat{S}_j$ and $-\hat{S}_j + \tau$, we can iterate this process, and we see that in fact we must have that the limit set of $\Gamma$ contains $\mathcal{P}_{\OO,j}$. It remains to show that closure of this packing is the entirety of the limit set.
	
	Let $T$ denote the set of points between $\hat{S}_j$ and $-\hat{S}_j + \tau$, inclusive. Consider the set $T \times [0,\infty)$ inside $\HH^4$ union its boundary. Since the generators of $\Gamma$ preserve $T \times [0,\infty)$ and we chose $o$ to lie inside this set, it follows that all the limit points of $\Gamma o$ lie in $T$. However, this shows that the limit points must be contained in the exterior of every sphere in $\mathcal{P}_{\OO,j}$, since we have shown that the action of $\Gamma$ can send any sphere in $\mathcal{P}_{\OO,j}$ to either $\hat{S}_j$ or $-\hat{S}_j + \tau$. Any point in the exterior of every sphere in $\mathcal{P}_{\OO,j}$ is contained in the topological closure of the sphere packing, and therefore we are done.
	\end{proof}
	
Lemmas \ref{FundamentalPolyhedron} and \ref{LimitSet} together prove that $\mathcal{P}_{\OO,j}$ is crystallographic.

\section{Equivalency of Constructions:}\label{ProofOfMainTheorem}

We now seek to show that $\hat{\mathcal{S}}_{\OO,j}$ is the super-packing of $\mathcal{P}_{\OO,j}$. Let $\widetilde{\mathcal{P}}_{\OO,j}$ denote the super-packing of $\mathcal{P}_{\OO,j}$. It is easy to see that $\hat{\mathcal{S}}_{\OO,j} \subset \widetilde{\mathcal{P}}_{\OO,j}$---this follows immediately from the fact that the super-group of $\widetilde{\mathcal{P}}_{\OO,j}$ has to contain $\Gamma$ as defined in the previous section, as well as the maps
	\begin{align*}
	z &\mapsto z + 1 \\
	z &\mapsto z + \varpi \\
	z &\mapsto z + \tau,
	\end{align*}
	
\noindent which is sufficient to generate the group $PSL^\ddagger(2,\OO) \oplus \langle\varphi_1\rangle$---the action of this group on $\mathcal{P}_{\OO,j}$ yields $\hat{\mathcal{S}}_{\OO,j}$. On the other hand, $\hat{\mathcal{S}}_{\OO,j}$ has an important invariance property.

	\begin{lemma}\label{InvariantUnderSuperGroup}
	The action of the super-group $\widetilde{\Gamma}$ preserves $\hat{\mathcal{S}}_{\OO,j}$.
	\end{lemma}
	
	\begin{proof}
	By the proof of Lemma \ref{LimitSet}, we know that by the action of $\Gamma$, we can move any sphere in $\mathcal{P}_{\OO,j}$ to either $-\hat{S}_j$ or $\hat{S}_j + \tau$. Therefore, given any element $g \in \widetilde{\Gamma}$, by composing with elements of $\Gamma$, we can assume that $g$ sends $-\hat{S}_j$ to either $-\hat{S}_j$ or $\hat{S}_j + \tau$. Since the transformation $z \mapsto \overline{z} + \tau$ interchanges $-\hat{S}_j$ and $\hat{S}_j + \tau$, we can in fact assume that $g$ sends $-\hat{S}_j$ to $-\hat{S}_j$. Furthermore, we know that $\hat{S}_j + \tau$ gets mapped to a sphere tangent to $-\hat{S}_j$---the point of intersection must be rational, but by Lemma \ref{TransitiveAction} we know that $\Gamma$ contains a subgroup that has a well-defined, transitive action on the rational points of $-\hat{S}_j$. Consequently, by composing with an element of $\Gamma$, we can assume that $g$ send $-\hat{S}_j$ to $-\hat{S}_j$ and $\hat{S}_j + \tau$ to $\hat{S}_j + \tau$. The only transformations that satisfy this are Euclidean isometries.
	
	We consider the spheres immediately tangent to $-\hat{S}_j$ at $0$, $1$, $\varpi$, and $\tau$---these are all of minimal positive bend, and appear both in $\hat{\mathcal{S}}_{\OO,j}$ and $\mathcal{P}_{\OO,j}$. The action of $g$ must map them to spheres of minimal positive bend tangent to $-\hat{S}_j$. However, by composing with translations and reflections that preserve $\hat{\mathcal{S}}_{\OO,j}$ and $\mathcal{P}_{\OO,j}$, we can assume that the sphere tangent at $0$ is mapped back to itself. After that, there is only a finite number of possible Euclidean isometries that can possibly map the remaining three spheres to other spheres of minimal positive bend immediately tangent to $-\hat{S}_j$. It is an easy exercise to check that all such transformations that preserve $\mathcal{P}_{\OO,j}$ preserve $\hat{\mathcal{S}}_{\OO,j}$.
	\end{proof}
	
	\begin{corollary}
	$\widetilde{\mathcal{P}}_{\OO,j} = \hat{\mathcal{S}}_{\OO,j}$.
	\end{corollary}
	
	\begin{proof}
	As we have already observed, $\widetilde{\mathcal{P}}_{\OO,j}$ contains $\hat{\mathcal{S}}_{\OO,j}$---it must therefore be the union of orbits of $\hat{\mathcal{S}}_{\OO,j}$ under the action of the super-group. However, since $\hat{\mathcal{S}}_{\OO,j}$ is preserved by the super-group, we conclude that these two sets are in fact equal.
	\end{proof}
	
	We have thus proved Theorem \ref{SICP Classification}, as promised.
	
\section{Boyd-Maxwell Packings and the Orthoplicial Packing:}\label{DistinctPackings}

Our final task is to show that the packings that we have constructed are not all conformally equivalent to packings already in the literature, specifically the orthoplicial packing and Boyd-Maxwell packings such as the Soddy packing. We begin by briefly recalling the definition of a Boyd-Maxwell packing. Let $W$ be a group generated by a finite set of generators $S$ and relations
	\begin{align*}
    (st)^{m_{s,t}} = 1, \ \forall s,t \in S
    \end{align*}
    
\noindent where $m_{s,s} = 1$ and if $s \neq t$ then $m_{s,t} \geq 2$ or $\infty$. Then we say that $(W,S)$ is a Coxeter system, and associate to it a matrix $B$ such that
	\begin{align*}
    B_{s,t} = \begin{cases} \cos\left(\frac{\pi}{m_{s,t}}\right) & \text{if } m_{s,t} < \infty \\ 1 & \text{otherwise} \end{cases}
    \end{align*}

\noindent and a Coxeter diagram, where the vertices are elements of $S$, and we include an edge between $s,t \in S$ if $m_{s,t} > 1$---furthermore, if $m_{s,t} > 2$, then we label the edge by $m_{s,t}$.

Define $V$ to be real vector space with basis $e_s$. Then the matrix $B$ defines a bilinear form on $V$, and the group $W$ is isomorphic to a subgroup of $O_B(V)$. If $B$ has signature $(n - 1,1)$, then $V$ is a Lorentzian space, and we can identify the hyperboloid
	\begin{align*}
    \left\{x \in V \middle| B(x,x) = 1\right\}
    \end{align*}
    
\noindent with the collection of spheres in $\RR^{n - 2}$. With this convention, we identify $W$ with a subgroup of $\Isom(\HH^{n - 1})$ generated by spheres corresponding to the vectors $e_s$. Let $\{f_s\}$ denote the dual basis of $e_s$, and consider the sets
	\begin{align*}
    \Omega &= \bigcup_{s \in S} W.f_s \\
    \Omega_r &= \left\{x \in \Omega \middle| B(x,x) > 0\right\}.
    \end{align*}
    
\noindent Normalizing the vectors in $\Omega_r$ to have norm $1$, we can identify them with a collection of spheres in $\RR^{n - 2}$---if the spheres in this collection have disjoint interiors, we call it a Boyd-Maxwell packing. If, furthermore, $\Omega = \Omega_r$, we call it a non-degenerate Boyd-Maxwell packing. These were originally introduced by Boyd \cite{Boyd}; Maxwell proved under what conditions this collection of spheres is disjoint, and attempted to give a full classification \cite{Maxwell}. Maxwell's list was eventually amended by Chen and Labb\'{e} \cite{Chen2015}, which we shall use extensively. We begin by showing that three of the super-integral crystallographic packings that we have defined are in fact non-degenerate Boyd-Maxwell packings.

\begin{lemma}\label{BMGraph16}
For $H = \left(\frac{-1,-6}{\QQ}\right)$, the corresponding $(\OO,j)$-Apollonian sphere packing is a non-degenerate Boyd-Maxwell packing corresponding to the Coxeter diagram below.
	\begin{center}
	\includegraphics[width=0.2\textwidth]{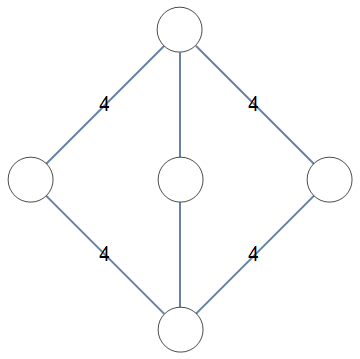}
	\end{center}
\end{lemma}

\begin{proof}
Consider $-\hat{S}_j$, $\hat{S}_j + \tau$, and the three spheres in $\mathcal{P}_{\OO,j}$ tangent to $-\hat{S}_j$ at $0$, $1$, and $i$. By inspection, the reflections through the dual spheres are given by
	\begin{align*}
		\phi_1(z) &= -\overline{z} \\
		\phi_2(z) &= -i\overline{z}i \\
		\phi_3(z) &= \frac{2}{j^2}\left(\frac{j - ij}{2}(z^\ddagger - 1 - i)\frac{j + ij}{2}\right) \\
		\phi_4(z) &= \frac{1 + i + j}{2} + \left(\overline{z} - \frac{1 - i - j}{2}\right)^{-1} \\
		\phi_5(z) &= \varphi_3\left(\left(z - 1 - i\right)\left((1 - i)z - 1\right)^{-1}\right).
	\end{align*}
	
\noindent Note that $\phi_3$ is the reflection through the plane $x + y = 1$---given this, it is clear that each of these reflections preserve $\mathcal{S}_{\OO,j}$. Indeed, since each reflection moves one sphere to a sphere that is immediately tangent to one of the spheres in the configuration, the orbit of the initial sphere cluster must be a subset of $\mathcal{P}_{\OO,j}$. However, the group generated by the $\phi_i$ is a Coxeter group, with relations
	\begin{align*}
	\begin{array}{lll}
	\left(\phi_1 \phi_2\right)^2 = id & \left(\phi_1 \phi_5\right)^4 = id & \left(\phi_2 \phi_5\right)^4 = id \\
	\left(\phi_1 \phi_3\right)^4 = id & \left(\phi_2 \phi_3\right)^4 = id & \left(\phi_3 \phi_4\right)^2 = id \\
	\left(\phi_1 \phi_4\right)^3 = id & \left(\phi_2 \phi_4\right)^3 = id & \left(\phi_4 \phi_5\right)^2 = id.
	\end{array}
	\end{align*}
	
\noindent Therefore, the orbit of the initial sphere cluster is a Boyd-Maxwell packing corresponding to the given Coxeter diagram, and since this packing is non-degenerate, it is maximal \cite{Chen2015}, meaning that no further oriented spheres can be added to the collection such that the interiors do not intersect with any of the existing oriented spheres. Ergo, this Boyd-Maxwell packing is equal to $\mathcal{P}_{\OO,j}$, as claimed.
\end{proof}

\begin{lemma}\label{BMGraph31}
For $H = \left(\frac{-3,-1}{\QQ}\right)$, the corresponding $\OO$-Apollonian sphere packing is the non-degenerate Boyd-Maxwell packing corresponding to the Coxeter diagram below.
	\begin{center}
	\includegraphics[width=0.2\textwidth]{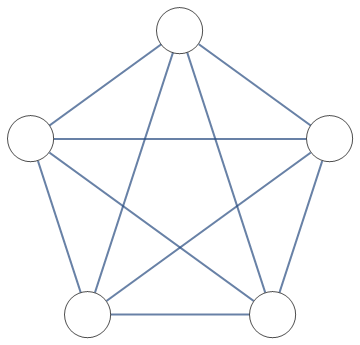}
	\end{center}
\end{lemma}

We note that this Boyd-Maxwell packing is precisely the Soddy packing \cite{Soddy}.

\begin{proof}[Proof of Lemma \ref{BMGraph31}]
Consider $-\hat{S}_j$, $\hat{S}_j + \tau$, and the three spheres in $\mathcal{P}_{\OO,j}$ tangent to $-\hat{S}_j$ at $0$, $1,$, and $\frac{1 + i}{2}$. Let $\phi_r$ denote the rotation of $-\hat{S}_j$ around the point $\frac{3 + i}{6}$ by $\pi/3$ degrees---this rotation cyclically permutes the three aforementioned spheres tangent to $-\hat{S}_j$ and preserves $\mathcal{P}_{\OO,j}$. We note also that the reflection through the plane $z = 1/2$ switches $-\hat{S}_j$ and $\hat{S}_j + \tau$ while preserving $\mathcal{P}_{\OO,j}$, and that the transformation $z \mapsto -z^{-1}$ switches $\hat{S}_j + \tau$ with the sphere tangent at $0$ while preserving $\mathcal{P}_{\OO,j}$.

Since the sphere reflection $z \mapsto -\overline{z}$ preserves $\mathcal{P}_{\OO,j}$, reflections through all the images of this sphere under the above transformations preserve $\mathcal{P}_{\OO,j}$. However, this includes the dual spheres to the above sphere cluster. By symmetry, it is clear that all of the dual spheres intersect at an angle of $\pi/6$, and consequently they correspond to the Boyd-Maxwell packing with the given Coxeter diagram. Since this is a maximal packing and is a subset of $\mathcal{P}_{\OO,j}$, it is in fact equal to $\mathcal{P}_{\OO,j}$.
\end{proof}

\begin{lemma}\label{BMGraph32}
For $H = \left(\frac{-3,-2}{\QQ}\right)$, the corresponding $\OO$-Apollonian sphere packing is the non-degenerate Boyd-Maxwell packing corresponding to the Coxeter diagram below.
	\begin{center}
	\includegraphics[width=0.2\textwidth]{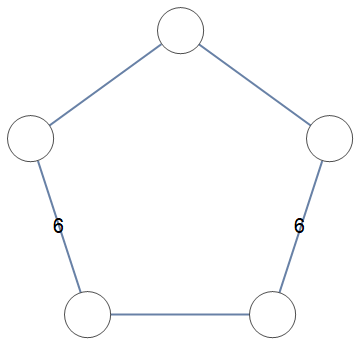}
	\end{center}
\end{lemma}

\begin{proof}[Proof of Lemma \ref{BMGraph32}]
Let $\phi_1, \phi_2, \phi_3$ be plane reflections, where the plane corresponding to $\phi_1$ passes through $0$, with normal vector $(1 - i)/2$; the plane corresponding to $\phi_2$ passes through $0$, with normal vector $i$; and the plane corresponding to $\phi_3$ passes through $i$, with normal vector $3 + i$. By inspection, all three of these plane reflections preserve $\mathcal{P}_{\OO,j}$. Define as well sphere reflections
	\begin{align*}
	\phi_4(z) &= \left(\overline{z - i}\right)^{-1} + i \\
	\phi_5(z) &= \left(\overline{z - j}\right)^{-1} + j.
	\end{align*}
	
\noindent These are also manifestly transformations that preserve $\mathcal{P}_{\OO,j}$. Thus the group generated by $\phi_1,\ldots \phi_5$ preserves $\mathcal{P}_{\OO,j}$. These are generators are the dual spheres to spheres corresponding to
	\begin{align*}
	(0,0,-j), (0,4,j), (4,0,j), (4,8,4+3j), (8,12,6 + 2i + 5j),
	\end{align*}
	
\noindent which are checked to be elements of $\mathcal{P}_{\OO,j}$. The orbit of this initial configuration is therefore a subset of $\mathcal{P}_{\OO,j}$. However, this orbit is the Boyd-Maxwell packing with the given Coxeter digram, and as this is a maximal packing, it is in fact equal to $\mathcal{P}_{\OO,j}$.
\end{proof}

Our claim is that aside from these three exceptions, all the other packings $\mathcal{P}_{\OO,j}$ are neither non-degenerate Boyd-Maxwell packings nor the orthoplicial packing. To show this, we first prove the following technical lemma.

\begin{lemma}\label{No Non Trivial Reflections}
Let $\mathcal{P}_{\OO,j}$ be an $(\OO,j)$-Apollonian packing. If
	\begin{align*}
	\OO \nsubseteq \left(\frac{-1,-6}{\QQ}\right), \left(\frac{-1,-10}{\QQ}\right), \left(\frac{-3,-1}{\QQ}\right),\left(\frac{-3,-2}{\QQ}\right),
	\end{align*}
	
\noindent then the only plane reflections preserving $\mathcal{P}_{\OO,j}$ are normal to either $1$ or $i$.
\end{lemma}

\begin{proof}
Any plane reflection preserving $\mathcal{P}_{\OO,j}$ must preserve the two planes that it contains, and must therefore must have a normal vector in $S_j$. It must also preserve the set of spheres of smallest bend tangent to $-\hat{S}_j$, and in particular must preserve the set of their tangencies. If $i^2 \neq -1,-3$, then this set is a rectangular lattice whose symmetry group is generated by reflections through planes normal to $1$ and $i$. If $i^2 = -1$, then it is a square lattice, which is also preserved by reflections through the diagonals. However, one checks by inspection that if also $j^2 = -7$, then the set of spheres of smallest bend tangent to $\frac{1 + j}{2} + \hat{S}_j$ is not preserved by such reflections. If $i^2 = -3$, then we have a triangular lattice, which is preserved under reflections through the sides of the equilateral triangles that it is composed of. However, if we also have $j^2 = -15$, then one checks that the set of spheres of smallest bend tangent to $\frac{i + j}{3} + \hat{S}_j$ is not preserved by such reflections.
\end{proof}

\begin{theorem}\label{NotBMGraph}
If $\OO \nsubseteq \left(\frac{-1,-6}{\QQ}\right), \left(\frac{-3,-1}{\QQ}\right), \left(\frac{-3,-2}{\QQ}\right)$, then $\mathcal{P}_{\OO,j}$ is neither a non-degenerate Boyd-Maxwell packing nor the orthoplicial packing.
\end{theorem}

\begin{proof}
First, we prove that all these $(\OO,j)$-Apollonian packings are not non-degenerate Boyd-Maxwell packings. Suppose otherwise, and choose any two tangent spheres in the initial configuration. Without loss of generality, these two spheres are the two planes in $\mathcal{P}_{\OO,j}$, in which case there are three generators of the Coxeter group preserving those two planes. Those generators must correspond to plane reflections, which by Lemma \ref{No Non Trivial Reflections} must either be normal to either $1$ or $i$ if $i^2 \neq -1,-3$. As there are three reflections, at least two of them must be reflections through parallel planes---however, this is impossible, as it would mean that there would be two generators $g_1, g_2$ such that $g_1 g_2$ is of infinite order, which does not occur for Boyd-Maxwell sphere packings, per the classification given by Chen and Labb\'{e} \cite{Chen2015}. This leaves the case where $i^2 = -1$, $j^2 = -10$, but we note that the set
	\begin{align*}
	B &= \left\{b_{H,j}(\text{inv}(S_1),\text{inv}(S_2))\middle|S_1, S_2 \in \mathcal{P}_{\OO,j}\right\} \\
    &\subset \left\{\pm\xi_3(\gamma)\middle| \gamma \in SL^\ddagger(2,\OO)\right\},
	\end{align*}

\noindent is invariant under conformal transformations and must satisfy certain congruence restrictions---to be precise, we know that it must be a subset of $\pm 1 + 5\ZZ$. Ergo, we can further restrict the class of possible non-degenerate Maxwell-Boyd packings by computing its first few terms for each Maxwell-Boyd packing in the list given in \cite{Chen2015}. By direct computation, there are no non-degenerate Maxwell-Boyd packings that satisfy the desired condition. It remains to show that none of the $(\OO,j)$-Apollonian packings are conformally equivalent to the orthoplicial packing. However, as discussed in \cite{Dias,Nakamura}, there is a congruence restriction on the bends of the orthoplicial packing---only certain congruence classes modulo $4$ appear as bends. By inspection, for all of the $(\OO,j)$-Apollonian packings, all congruence classes modulo $4$ are represented by bends, and therefore they are not conformally equivalent to the orthoplicial packing.
\end{proof}

\bibliography{Quaternionic}
\bibliographystyle{alpha}
\end{document}